\documentclass[a4paper,10pt]{article} 
\usepackage{amsmath, amsthm, amssymb}
\usepackage{url}

\usepackage{physics}
\usepackage{graphicx,lipsum}
\graphicspath{ {./downloads/} }

\usepackage[english]{babel}
\usepackage[T1]{fontenc}
\usepackage{lmodern}
\usepackage{caption}

\usepackage{titlesec} 
\titlespacing*{\section}{0pt}{2ex}{2ex}

\usepackage[colorlinks,citecolor=red,urlcolor=blue,bookmarks=false,hypertexnames=true]{hyperref}
\usepackage{tikz}
\usetikzlibrary{calc}
\usetikzlibrary{shapes}
\usepackage[autostyle]{csquotes}
\makeatletter

\usepackage[colorlinks]{hyperref}
\usepackage[nameinlink,capitalize]{cleveref}
\newtheorem{theorem}{Theorem}[section]

\newtheorem{lemma}[theorem]{Lemma}
 \newtheorem{prop}{Proposition}[section]
\newtheoremstyle{named}{}{}{\itshape}{}{\bfseries}{.}{.5em}{\thmnote{#3's }#1}
\theoremstyle{named}

\newtheoremstyle{remarkstyle}  
  {8pt}                       
  {8pt}                       
  {\normalfont}                
  {}                           
  {\bfseries}                  
  {.}                          
  {5pt plus 1pt minus 1pt}     
  {}         

\theoremstyle{remarkstyle}
\newtheorem{remark}{Remark}[section]
\theoremstyle{definition}
\newtheorem{definition}{Definition}[section]

\theoremstyle{plain} 

\usepackage{xspace}
\usepackage[margin=1.0in]{geometry}

\numberwithin{equation}{section}
\usepackage{setspace}

\usepackage{comment}

\usepackage{titlesec}

\usepackage{mathtools}

\DeclarePairedDelimiterX{\inp}[2]{\langle}{\rangle}{#1, #2}
\titleformat{\chapter}
  {\Large\bfseries} 
  {}                
  {0pt}            
  {\huge}

\begin{document}

\begin{center}
\fontsize{12pt}{10pt}\selectfont
    \textbf{GLOBAL WELL-POSEDNESS FOR A HIGHER-ORDER BENJAMIN-ONO-SCHRÖDINGER SYSTEM IN THE ENERGY SPACE}
    \end{center}
\vspace{0.1cm}
\begin{center}
   \fontsize{12pt}{10pt}\selectfont
    \textsc{Fáuster Santana}
\end{center}
\vspace{0.2cm}

\begin{abstract}
We study the Cauchy problem associated with the higher-order
Benjamin-Ono-Schrödinger system
\begin{equation*}
\begin{cases}
 \partial_{t}r-a\partial_{x}^{3}r-b\mathcal{H}\partial_{x}^{2}r
 =cr\partial_{x}r
 -d\partial_{x}(r\mathcal{H}\partial_{x}r+\mathcal{H}(r\partial_{x}r))
 +\beta \partial_{x}(|q|^{2}), 
 \quad x,t\in \mathbb{R},\\[0.2cm]
 i\partial_{t}q-\alpha \partial_{x}^{2}q=-\beta qr ,
\end{cases}
\end{equation*}
where $b,c,d,\alpha,\beta$ are positive constants and $a\neq 0$ is a real constant. This system was introduced by Kairzhan, Kennedy, and Sulem in \cite{Higher-order-Benjamin-Ono-NLS-System-Sulem}. We prove that this system is locally well-posed in the energy space
$H^{1}(\mathbb{R})\times H^{1}(\mathbb{R})$. Furthermore, in the case $a<0$, this result extends globally for initial data $(r_{0},q_0)$ with sufficiently small $H^{1}\times H^1$-norm. The proof combines compactness arguments with energy methods. To provide smooth solutions, we have to deal with the lost of the derivatives phenomenon introduced by higher-order derivatives and the Hilbert transform in the nonlinear terms when performing energy estimates. This is overcome by introducing a
modified energy functional that cancels the problematic terms
arising in the standard energy estimates. Once this is done, we extend the method put forward by Molinet and Pilod in \cite{HOBOinH1-Didier-Molinet} to study a single higher-order Benjamin-Ono equation. Their procedure includes the use of a gauge transformation of Tao's type \cite{GWP-BO-Tao}, and delicate bilinear estimates in Bourgain type spaces.
\end{abstract}

\noindent 

\section{Introduction}

We are interested in the study of the initial value problem (IVP) associated to the \textbf{higher-order Benjamin-Ono-Schrödinger (HBOS) system}
\begin{spacing}{1.2}
\begin{equation} \label{BO-NLS}
\begin{cases}
     & \partial_{t}r-a\partial_{x}^{3}r-b\mathcal{H}\partial_{x}^{2}r=cr\partial_{x}r-d\partial_{x}(r\mathcal{H}\partial_{x}r+\mathcal{H}(r\partial_{x}r))+\beta \partial_{x}(|q|^{2}), \\
     & i\partial_{t}q-\alpha \partial_{x}^{2}q=-\beta qr, \\
     & r(x,0)=r_{0}(x), \ \ q(x,0)=q_{0}(x)
\end{cases}
\end{equation}
\end{spacing}
\noindent where $x,t\in \mathbb{R}$, $r=r(x,t)$ is a real-valued function, $q=q(x,t)$ is a complex-valued function, $a$ is a non-zero real number, $b,c,d,\alpha,\beta$ are positive real numbers, and $\mathcal{H}$ denotes the Hilbert transform, defined on the line as
$$
\mathcal{H}(f)(x) = p.v. \frac{1}{\pi}\int_{\mathbb{R}}\frac{f(y)}{x-y}dy.
$$

The system in \eqref{BO-NLS} is a particular case of the following more general system 
\begin{spacing}{1.2}
\begin{equation}\label{BO-NLS-General}
\left\{
\begin{aligned}
 \partial_{t}R-a\partial_{x}^{3}R-b\mathcal{H}\partial_{x}^{2}R&=cR\partial_{x}R-d\partial_{x}(R\mathcal{H}\partial_{x}R+\mathcal{H}(R\partial_{x}R))+\beta \partial_{x}(|S|^{2}) \\
&\quad +ia_{0}\partial_{x}(S\overline{\partial_{x}S}-\overline{S}\partial_{x}S)-a_{1}\mathcal{H}\partial_{x}^{2}(|S|^{2}),\\
i\partial_{t}S-\alpha \partial_{x}^{2}S&=-\beta SR+ia_{0}(\partial_{x}(RS)+R\partial_{x}S)+a_{1}S\mathcal{H}\partial_{x}R, \\
\end{aligned}
\right.
\end{equation}
\end{spacing}
\noindent where $a_{0}$ and $a_{1}$ denote additional real coupling parameters, $R=R(x,t)$ is a real-valued function, and $S=S(x,t)$ is a complex-valued function. This system was derived by Kairzhan, Kennedy, and Sulem in \cite{Higher-order-Benjamin-Ono-NLS-System-Sulem} using perturbation Hamiltonian theory. It corresponds to the evolution  of long internal waves and modulated surface wave packets in a density-stratified fluid composed of two immiscible layers separated by a sharp interface. More precisely, the model describes their resonant interaction through a system of equations where the internal wave is governed by a high-order Benjamin-Ono equation, coupled to a linear Schrödinger equation for the envelope of the free surface.

Our aim in this work is to establish a well-posedness theory for the system \eqref{BO-NLS} in low-regularity Sobolev spaces without using weights. The main difficulty arises from the fact that the first equation corresponds to a higher-order Benjamin-Ono type equation. More precisely, it takes the form
\begin{equation}\label{higher-order-BO}
    \partial_{t}r - a\partial_{x}^{3}r - b\mathcal{H}\partial_{x}^{2}r
    = cr\partial_{x}r - d\partial_{x}\big(r\mathcal{H}\partial_{x}r + \mathcal{H}(r\partial_{x}r)\big).
\end{equation}
Although the presence of the KdV-type dispersive term $\partial_x^3$, the associated dispersive effects are insufficient to yield well-posedness for \eqref{higher-order-BO} via a fixed point argument. Indeed, it has been shown by Pilod in \cite{Didier-HigherOrderDispersiveModels} that the flow map data-solution cannot be $C^2$ in any Sobolev space $H^{s}(\mathbb{R}),\ s\in \mathbb{R}$. It is due to the second-order nonlinear terms present on the right hand side of the equation.

We observe that the following quantities are conserved by the flow associated with \eqref{BO-NLS}
\begin{align}\label{Hamiltoniano}
    E_{1}&=\int_{\mathbb{R}}\left( -\frac{b}{2}r\mathcal{H}\partial_{x}r+\frac{a}{2}|\partial_{x}r|^{2}-\alpha|\partial_{x}q|^{2}-\frac{c}{6}r^{3}-\beta r|q|^{2}+\frac{d}{2}r^{2}\mathcal{H}\partial_{x}r \right) dx, \quad \textrm{(Hamiltonian)}
    \end{align}
\begin{align}\label{Mass}
    E_{2}&=\int_{\mathbb{R}}|q|^{2}dx,
\end{align}
and
\begin{equation}\label{Momento}
E_{3}=\frac{1}{2}\int_{\mathbb{R}}r^{2}dx+\textrm{Im}\int_{\mathbb{R}} q\overline{\partial_{x}q}\ dx.
\end{equation}

Another related model, and of interest of study, describing the interaction between long and short waves under a weakly coupled nonlinear regime, is given by the following system
\begin{equation}\label{extended-BO-NLS}
\left\{
\begin{aligned}
&i\partial_{t}q+\partial_{x}^{2}q=\mu qr+\gamma|q|^{2}q,\quad\quad \quad  x,t\in \mathbb{R}\\
&\partial_{t}r-\vartheta\mathcal{H}\partial_{x}^{2}r+\lambda r\partial_{x}r=\kappa \partial_{x}(|q|^{2}),
\end{aligned}
\right.
\end{equation}
where the parameters $\mu,\kappa >0$, $\lambda,  \gamma \in \mathbb{R}$, and $\vartheta \neq 0$. In the literature, this model is commonly referred to as the extended Schrödinger-Benjamin-Ono system (see \cite{NLSBO-Felipe-Didier}). Note that, except for the cubic term $|q|^{2}q$, the system above can be considered as a reduced version of \eqref{BO-NLS}. In particular, similar difficulties are anticipated regarding the well-posedness theory.

The initial value problem associated with the system \eqref{extended-BO-NLS}, as well as several of its particular cases, has been extensively studied in recent years. To begin with, consider the case $\gamma=\lambda=0$. When $|\vartheta|\neq 1$ (the \textit{non-resonant} case), Bekiranov, Ogawa and Ponce proved in \cite{PonceBekiranov} the local well-posedness of the system \eqref{extended-BO-NLS} in the Sobolev spaces $H^{s}(\mathbb{R})\times H^{s-\frac{1}{2}}(\mathbb{R})$ for all $s\geq 0$. In particular, due to the conservation laws satisfied by \eqref{extended-BO-NLS}, these solutions extend globally in time for $s\geq 1$ whenever $\frac{\mu\vartheta}{\kappa}>0$. Subsequently, Angulo, Matheus and Pilod improved this result in \cite{NLS-BO-Angulo-Pilod} by establishing global well-posedness for all $s\geq 0$. Later, in \cite{Shap-wellposedenss-Leandro}, Domingues established a local well-posedness result in the Sobolev spaces $H^{s}(\mathbb{R})\times H^{s'}(\mathbb{R})$ for indices $(s,s')\in \mathbb{R}^{2}$ satisfying
$
-\frac{1}{2}<s'-(s-\tfrac{1}{2})<1 $ and $
-\frac{1}{2}\leq s'\leq 2s-\frac{1}{2},
$
which is sharp in a certain sense. Moreover, this result improves the previously known local well-posedness results. On the other hand, concerning the case $|\vartheta|=1$ (the \textit{resonant} case), Pecher showed in \cite{NLS-BO-Pecher} the local well-posedness of the system for the same regularity obtained by Bekiranov, Ogawa and Ponce, except for the endpoint $(0,-\frac{1}{2})$. All the aforementioned results were obtained by combining the fixed point argument with the Fourier restriction norm method.

In the case $\mu=\kappa=\vartheta=1$ and $\lambda\neq0$, Linares, Mendez, and Pilod proved in \cite{NLSBO-Felipe-Didier} the local well-posedness of the system \eqref{extended-BO-NLS} in $H^{s+\frac{1}{2}}(\mathbb{R})\times H^{s}(\mathbb{R})$ for $s>\frac{5}{4}$. It is worth observing that, in this case, the system \eqref{extended-BO-NLS} is expected to exhibit difficulties similar to those arising in the Benjamin-Ono equation. In particular, the well-posedness in Sobolev spaces cannot be established using a contraction principle argument. Consequently, due to this feature, their approach relies on the modified energy method introduced by Kwon in \cite{5KDV-Kwon}, together with refined Strichartz estimates proved by Kenig and Koenig in \cite{Benjamin-Ono-Kenig-Koenig}. Recently, in \cite{Linares-Pilod-extended}, a set of results regarding global and local well-posedness for this system will be available in a forthcoming paper.

The well-posedness theory concerning the higher-order Benjamin-Ono equation \eqref{higher-order-BO} was first established by Linares, Pilod, and Ponce in \cite{Higher-order-BO-Linares-Pilod}. They proved that this equation is locally well-posed for initial data in $H^{s}(\mathbb{R}),\  s\geq 2,$ as well as in $H^{k}(\mathbb{R})\cap L^{2}(x^{2}dx),\ k\in \mathbb{Z}_{+},\ k\geq 2$. The result is obtained by using a gauge transformation, the contraction principle, and compactness arguments. However, the fixed point argument is carried out in weighted Sobolev spaces and does not present additional obstructions. Subsequently,  Molinet and Pilod proved in \cite{HOBOinH1-Didier-Molinet} that this equation is locally well-posed in the energy space $H^{1}(\mathbb{R})$. The proof is based on a gauge transformation introduced by Tao in \cite{GWP-BO-Tao}. Furthermore, due to the conserved quantities associated with the flow, this result was extended globally.

Regarding the HBOS system \eqref{BO-NLS}, Kairzhan, Kennedy, and Sulem proved in \cite{Higher-order-Benjamin-Ono-NLS-System-Sulem} the local well-posedness in 
$
H^{2}(\mathbb{R})\cap L^{2}(x^{2}dx) \times H^{3}(\mathbb{R})\cap L^{2}(x^{2}dx)
$. Their method relies on the contraction principle together with a gauge transformation, which is based on the arguments used by Linares, Pilod, and Ponce in \cite{Higher-order-BO-Linares-Pilod}. It is worth noting that, due to the presence of a higher-order Benjamin–Ono in the first equation of \eqref{BO-NLS}, one expects that arguments based on the contraction principle are not effective for establishing well-posedness in Sobolev spaces as well. However, this obstruction does not occur in certain weighted Sobolev spaces, where the contraction principle was effectively applied to obtain this result.

As far as we know, the only well-posedness result available for the system \eqref{BO-NLS} was obtained in weighted Sobolev spaces. A local well-posedness result established in the energy space $H^{1}\times H^{1}$, without weights, would allow one to extend the solutions globally by exploiting the conserved quantities \eqref{Hamiltoniano}–\eqref{Momento}. This is contained in our main result in this work.

Our main result is the following:

\begin{theorem}\label{maintheorem}
Let $s\geq 1$ and $s\leq s'\leq s+1$ be given. Then, for any $(r_{0},q_{0})\in H^{s}(\mathbb{R})\times H^{s'}(\mathbb{R})$, there exists a positive time $T=T(\|(r_{0},q_{0})\|_{H^{1}\times H^{s'-s+1}})$ for which the initial value problem \eqref{BO-NLS} admits a unique solution $(r,q)$ satisfying 
\begin{align}
    &(r,q)\in C([0,T]:H^{s}(\mathbb{R})\times H^{s'}(\mathbb{R})),\\
    &r\in L^{4}_{T}W^{s,4}_{x}\cap L^{2}_{x}L^{\infty}_{T}\cap X^{s-2\theta,\theta}_{T}, \quad \textrm{for all} \quad 0\leq \theta \leq 1,\\
    &q\in L^{2}_{x}L^{\infty}_{T}\cap Y^{s,1}_{T},
\end{align}
and
\begin{equation}\label{identityofw}
    w=\partial_{x}P_{+hi}(e^{iF})\in X^{s,\frac{1}{2},1}_{T},
\end{equation}
where $F=F[r]$ is a spatial primitive of $r$ as defined in subsection \ref{subsectionGT}.

 Moreover, the flow map data-solution is continuous from $H^{s}(\mathbb{R})\times H^{s'}(\mathbb{R})$ into $C([0,T]:H^{s}(\mathbb{R})\times H^{s'}(\mathbb{R}))$.
\end{theorem}

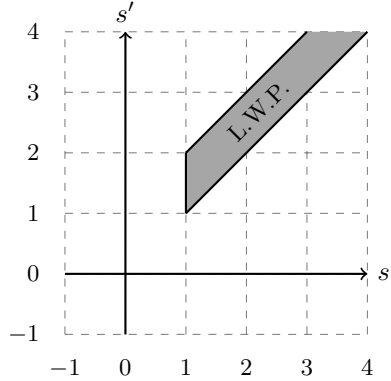
\begin{figure}[h]
    \centering
   \begin{tikzpicture}[scale=0.8]

\draw[dashed, gray] (-1,-1) grid (4,4);

\draw[->, thick] (-1,0) -- (4,0) node[right] {$s$};
\draw[->, thick] (0,-1) -- (0,4) node[above] {$s'$};

\foreach \x in {-1,0,1,2,3,4}
{
  \node[below=6pt] at (\x,-1) {\small $\x$};
}

\foreach \y in {-1,0,1,2,3,4}
{
  \node[left=6pt] at (-1,\y) {\small $\y$};
}

\fill[gray!70]
(1,1) -- (4,4) -- (3,4) -- (1,2) -- cycle;

\draw[thick] (1,1) -- (4,4);   
\draw[thick] (1,2) -- (3,4);   
\draw[thick] (1,1) -- (1,2);   

\node[rotate=45] at (2.2,2.7) {\small{L.W.P.}};

\end{tikzpicture}

\captionsetup{width=0.55\textwidth}
\caption{The gray
region contains indices $(s,s')$ for which local well-posedness is
achieved in Theorem \ref{maintheorem}.}

\end{figure}

In particular, due to the conservation laws satisfied by the solutions of \eqref{BO-NLS} and the local well-posedness at the low-regularity $H^{1}\times H^1$ given by Theorem \ref{maintheorem}, we have the following global well-posedness result.
\begin{theorem}\label{globalwellposedness}
 Let $a<0$ and $s\geq 1$. There exists a sufficiently small constant $\delta_{0}>0$ such that if $(r,q)\in C([0,T]:H^{s}(\mathbb{R})\times H^{s}(\mathbb{R}))$ is the local solution for \eqref{BO-NLS} given by Theorem \ref{maintheorem} and $\|(r_0,q_{0})\|_{H^{1}\times H^1}\leq \delta_{0}$, then $(r,q)$ can be extended globally as 
 \begin{equation*}
     (r,q)\in C_{b}(\mathbb{R}:H^{s}(\mathbb{R})\times H^{s}(\mathbb{R})).
 \end{equation*}
\end{theorem}
\begin{remark}
One can also prove that for any $K>0$ there exists $\delta_{0}>0$ such that, if $\|r_{0}\|_{H^{1}}\leq K$ and $\|q_{0}\|_{H^{1}}\leq \delta_{0}$, then the solution $(r,q)$ with initial data $(r_{0},q_{0})$ extends globally in time. But the same result cannot be prove with $r_{0}$ and $q_{0}$ changing the roles. 
\end{remark}

The proof of Theorem \ref{maintheorem} follows the ideas developed by Molinet and Pilod in \cite{HOBOinH1-Didier-Molinet}, where the gauge transformation introduced by Tao in \cite{GWP-BO-Tao} is employed to establish local well-posedness for the higher-order Benjamin-Ono equation \eqref{higher-order-BO} in $H^{1}(\mathbb{R})$. This transformation allows us to weaken the high-low frequency interactions in the nonlinearity of the first equation in \eqref{BO-NLS}, which constitute the main difficulty in the analysis of the higher-order Benjamin-Ono equation with regard to the well-posedness theory in Sobolev spaces. More precisely, the main idea is to exploit several estimates in Sobolev and Bourgain spaces at the $H^{1}$ level obtained by Molinet and Pilod in that work, and to combine them with suitable estimates on the solution $q$ of the Schrödinger equation in \eqref{BO-NLS}, as well as with new bilinear estimates in Bourgain spaces, which are required to handle the additional term $\partial_x(|q|^2)$ appearing in the first equation of \eqref{BO-NLS}. By employing all these estimates together with a compactness argument, we are able to obtain the desired result.

However, this compactness argument relies on the existence of smooth solutions associated with smooth initial data for the system \eqref{BO-NLS}, which, to the best of our knowledge, has not yet been established. Therefore, in order to justify this step, we combine a parabolic regularization argument with energy estimates, in a way analogous to that used by Linares, Mendez, and Pilod for the extended Schrödinger-Benjamin-Ono system in \cite{NLSBO-Felipe-Didier}. On the other hand, similarly to what occurs in \cite{NLSBO-Felipe-Didier}, due to the higher-order nonlinear and coupling terms present in the first equation in \eqref{BO-NLS}, standard energy methods do not apply directly to obtain a priori estimates. Indeed, if we define $E(t):=\|(r(t),q(t))\|_{H^{s}\times H^{s}}$ and consider $s$ sufficiently large,
this argument only yields the following estimate
\begin{align*}
    \frac{d}{dt}E(t)^{2}&\lesssim P(E(t))+\left| \int_{\mathbb{R}} D_{x}^{s}\partial_{x}(|q|^{2})D_{x}^{s}rdx\right|
     + \left|\int_{\mathbb{R}} D_{x}^{s}r\partial_{x}D_{x}^{s}(r\mathcal{H}\partial_{x}r)dx\right|+\left| \quad \int_{\mathbb{R}}D_{x}^{s}r\partial_{x}D_{x}^{s}\mathcal{H}(r\partial_{x}r) dx\right|,
\end{align*}
where $P$ is a polynomial with no constant or linear terms. Observe that the last three terms cannot be readily controlled. To overcome this difficulty, we employ the modified energy method introduced by Kwon in \cite{5KDV-Kwon} for the fifth-order KdV equation (see also  \cite{M-Paulsen-System}, \cite{KdvHierarchy-Kenig-Didier}, \cite{NLSBO-Felipe-Didier}, \cite{Tanaka-ThirdOrderBenjaminOno}). This approach consists in modifying the standard energy functional by adding suitable correction terms. When this modified energy is differentiated with respect to time, the additional terms generate contributions that cancel the problematic ones, while the remaining terms can be suitably estimated.

It is worth noting that the conservation laws in \eqref{Hamiltoniano}--\eqref{Momento} are not sufficient to yield global well-posedness for initials data $r_0$ and $q_{0}$ with arbitrarily large $H^{1}$-norm. The main obstruction arises from the last term in the Hamiltonian functional \eqref{Hamiltoniano}. A natural approach would be to apply the Gagliardo–Nirenberg and Young inequalities in order to derive an estimate of the form
\begin{equation*}
    \left|\int_{\mathbb{R}} r^{2}\mathcal{H}\partial_{x}r \, dx \right|
    \lesssim C(\varepsilon)\|r\|_{L^{2}_{x}}^{K} + \varepsilon \|\partial_{x} r\|_{L^{2}_{x}}^{2},
\end{equation*}
for some $0 < K < 4$ and for arbitrarily small $\varepsilon >0$ . However, due to the cubic nature of the integrand and the presence of the derivative $\partial_{x}$, such an estimate cannot be rigorously justified using this argument.

To ensure the effectiveness of our methods, specifically to establish the coercivity of the modified energy functional and to derive \textit{a priori} estimates for the solutions in low-regularity Sobolev spaces, we require the initial data to have a sufficiently small norm. To obtain this, we applied the same ideas introduced by Zaiter to deal with similar diffilcuty for the Ostrovsky equation in \cite{Ostrovsky-Zaiter} (see also \cite{NLSBO-Felipe-Didier}). In fact, observe that this smallness condition is not restrictive, since if \((r, q)\) is a solution to the system \eqref{BO-NLS} with initial data \((r_0, q_0)\), we can rescale the solution by a parameter \(0 < \lambda \leq 1\), defining the rescaled functions \(r_\lambda\) and \(q_\lambda\) as follows

\begin{equation}\label{rescaledsolutions}
    r_{\lambda}(x,t):=\lambda r(\lambda x,\lambda^{3} t), \quad q_{\lambda}(x,t)=\lambda q(\lambda x,\lambda^{3}t),
\end{equation}
and $(r_{\lambda},q_{\lambda})$ will be the solution of the following rescaled system
\begin{align}\label{BO-NLS-rescaled}
\begin{cases}
    \partial_{t}r_{\lambda}-b\lambda \mathcal{H}\partial_{x}^{2}r_{\lambda}-a\partial_{x}^{3}r_{\lambda}=\lambda cr_{\lambda}\partial_{x}r_{\lambda}-d\partial_{x}(r_{\lambda}\mathcal{H}\partial_{x}r_{\lambda}+\mathcal{H}(r_{\lambda}\partial_{x}r_{\lambda}))+\lambda\beta\partial_{x}(|q_{\lambda}|^{2}),\\
    \partial_{t}q-\lambda\alpha \partial_{x}^{2}q=-\lambda\beta qr,\\
    r_{\lambda}(x,0)=\lambda r(\lambda x,0),\quad q_{\lambda}(x,0)=\lambda q(\lambda x,0).
\end{cases}
\end{align}

Since
\begin{equation*}
    \|r_{\lambda}(\cdot,0)\|_{H^{{s}}}\lesssim \lambda^{\frac{1}{2}}(1+\lambda^{s})\|r_{0}\|_{H^{s}}, \quad \|q_{\lambda}(\cdot,0)\|_{H^{s}} \lesssim \lambda^\frac{1}{2}(1+\lambda^{s})\|q_{0}\|_{H^{s}},
\end{equation*}
for all $s\in \mathbb{R}$, given $\varepsilon >0$, we always choose $\lambda$ small enough such that 
\begin{equation}
    \|(r_{\lambda}(0),q_{\lambda}(0))\|_{H^{s}\times H^{s'}}\leq \varepsilon,
\end{equation}
for all $s,s'\in \mathbb{R}$.

Moreover, assume that there exists a constant $\varepsilon > 0$ such that, for every $0 < \lambda \leq 1$, the smallness condition
$
\|(r_\lambda(0, \cdot), q_\lambda(\cdot, 0))\|_{H^s \times H^{s'}} \leq \varepsilon
$
ensures the existence of a positive time $T_\lambda > 0$ for which the system \eqref{BO-NLS-rescaled} admits a unique solution $(r_\lambda, q_\lambda) \in C([0, T_\lambda]: H^s(\mathbb{R}) \times H^{s'}(\mathbb{R}))$. Defining the rescaled functions
$$
(r(x,t), q(x,t)) := \lambda^{-1} \left(r_\lambda(\lambda^{-1}x, \lambda^{-3}t),\; q_\lambda(\lambda^{-1}x, \lambda^{-3}t)\right),
$$
it follows that $(r, q) \in C([0, T]: H^s(\mathbb{R}) \times H^{s'}(\mathbb{R}))$ is a solution to the original system \eqref{BO-NLS}, with the existence time satisfying the bound $\lambda^3 T_\lambda \lesssim T$. 

\begin{remark}
We observe that formally the solutions of system \eqref{BO-NLS} with appropriate parameters tend to solutions of system \eqref{extended-BO-NLS}. In a forthcoming work we shall show that this fact can be make precise.
\end{remark}

\setlength{\parskip}{1em}

The paper is organized as follows. In the next section, we introduce the
notation, define the function spaces, recall classical estimates, and prove
some linear and commutator estimates. Section 3 is devoted to establishing the energy estimates required to prove the existence of smooth solutions for the system. In Section 4, we introduce the gauge
transformation and obtain several key \textit{a priori} estimates for the solutions
of the HBOS system and for its gauge function. Finally, in Section 5, we gather
the estimates established in the previous section and prove
Theorems \ref{maintheorem} and \ref{globalwellposedness}.

\section{Notations and preliminaries}

\subsection{Notations} \label{Notations}

For any $A$ and $B$ positive real numbers, the notation $A\lesssim B$ means that there exists a positive constant $C$ such that $A\leq CB$. Furthermore, for all $A\in \mathbb{R}$, the notation $A_{+}$ or $A_{-}$ will denote a number slightly greater and less than $A$, respectively.
\newline
\indent For $f\in \mathcal{S}(\mathbb{R}^{2})$, $\mathcal{F}f=\widehat{f}$ will denote its Fourier transform. On the other hand, the notation $\mathcal{F}_{x}f=\widehat{f}^{x}$ and $\mathcal{F}_{t}f=\widehat{f}^{t}$ will denote its Fourier transform in space and time variables, respectively. For any $s\in \mathbb{R}$, we define the Bessel and Riesz potentials of order $s$, $J^{s}_{x}$ and $D_{x}^{s}$, as follows
\begin{equation*}
    J^{s}_{x}f=\mathcal{F}_{x}^{-1}((1+|\xi|^{2})^{\frac{1}{2}}\mathcal{F}_{x}f) \quad \textrm{and} \quad D^{s}_{x}f=\mathcal{F}_{x}^{-1}(|\xi|^{s}\mathcal{F}_{x}f).
\end{equation*}
\indent In what follows, we introduce the Littlewood-Paley projectors that will be employed in this work. Let $\eta$ be a cutoff function such that
\begin{equation*}
    \eta\in C^{\infty}_{0}(\mathbb{R}),\quad 0\leq \eta \leq 1, \quad \textrm{supp }\eta\subseteq [-2,2], \quad \eta \equiv 1 \ \textrm{on} \ [-1,1].
\end{equation*}
Then, for all $l\in \mathbb{Z}_{+}$, we define
\begin{equation*}
    \phi(\xi)=\eta(\xi)-\eta(2\xi) , \quad \phi_{2^{l}}(\xi)=\phi(2^{-l}\xi), \quad \textrm{and }\quad 
    \psi_{2^{l}}(\xi,\tau)=\phi_{2^l}(\tau -b\xi|\xi|+a\xi^{3}).
\end{equation*}
Furthermore, we also define
\begin{equation*}
    \phi_{0}(\xi)=\eta(2\xi), \quad\textrm{and}\quad   \psi_{0}(\xi,\tau)=\phi_{0}(\tau -b\xi|\xi|+a\xi^{3}) .
\end{equation*}
Throughout this work, summations over capitalized variables such as $N$ or $M$ are presumed to be dyadic with $N,M\geq 0$. Then, we have the following properties
\begin{equation}
    \sum_{N}\phi_{N}(\xi)=1 \quad \textrm{for all} \ \xi\neq 0 , \quad \textrm{supp }\phi_{N}\subseteq \{2^{-1}N\leq |\xi|\leq 2N\}, \ N\geq 1, \quad \textrm{and } \quad \textrm{supp }\phi_{0}\subseteq \{|\xi|\leq 1\} .
\end{equation}
The Littlewood-Paley projectors $P_{N}$ and $Q_{N}$ are defined as follows
\begin{equation}
    P_{N}f:=\mathcal{F}^{-1}_{x}(\phi_{N}\mathcal{F}_{x}f), \quad    Q_{N}f:=\mathcal{F}^{-1}(\psi_{N}\mathcal{F}f),
\end{equation}
and,
\begin{equation}
    P_{\geq N}:=\sum_{M\geq N}P_{M} , \quad  Q_{\geq N}:=\sum_{M\geq N}Q_{M}.
\end{equation}
Moreover, we also define the operators $P_{hi},P_{lo},P_{HI}$ and $P_{LO}$ by
\begin{equation*}
    P_{hi}=\sum_{N\geq 2}P_{N}, \quad P_{lo}=1-P_{hi}, \quad P_{HI}=\sum_{N\geq 2^{4}}P_{N}, \quad P_{LO}=1-P_{HI}.
\end{equation*}
\indent The operators $P_{+}$ and $P_{-}$ will denote the projections in positive and negative frequencies, i.e.
\begin{equation*}
    P_{\pm}f=\mathcal{F}_{x}^{-1}(\mathcal{X}_{\mathbb{R}_{\pm }}\mathcal{F}_{x}f).
\end{equation*}
Thus, we also define $P_{\pm X} := P_{\pm} P_{X}$, where $X \in \{\mathrm{hi}, \mathrm{lo}, \mathrm{HI}, \mathrm{LO}\}$ or $X=N$ dyadic.
Note that all the operators $P_{\mathrm{hi}}, P_{\mathrm{lo}}, P_{\mathrm{HI}}, P_{\mathrm{LO}}$, and $P_{N}$ are bounded on $L^{p}(\mathbb{R})$ for every $1 \leq p \leq \infty$. The same boundedness holds for $P_{\pm}$, except possibly at the endpoints $p = 1$ and $p = \infty$. Observe that the following identity holds true
\begin{equation*}
    \mathcal{H}=-i(P_{+}-P_{-}).
\end{equation*}
\indent Finally, we denote by $V(\cdot)$ and $U(\cdot)$ the linear group associated with the linear parts of the higher-order Benjamin–Ono equation and the linear Schrödinger equation, respectively, as given in equation~\eqref{BO-NLS}, i.e.
\begin{equation}
    \mathcal{F}_{x}(V(t)f)(\xi)=e^{it(b\xi|\xi|-a\xi^{3})}\mathcal{F}_{x}f(\xi) \quad \textrm{and} \quad \mathcal{F}_{x}(U(t)f)(\xi)=e^{\alpha it \xi^{2}}\mathcal{F}_{x}f(\xi).
\end{equation}

\subsection{Functions spaces}
For any $1 \leq p \leq \infty$, we denote by $L^{p}(\mathbb{R})$ the usual Lebesgue space endowed with the norm $\|\cdot\|_{L^{p}}$. Moreover, for any $s \in \mathbb{R}$, we denote by $H^{s}(\mathbb{R})$ and $W^{s,p}(\mathbb{R})$ the Sobolev spaces of order $s$, endowed respectively with the norms
\[
\|f\|_{H^{s}}=\|J^{s}_{x}f\|_{L^{2}} \quad \text{and} \quad \|f\|_{W^{s,p}}=\|J^{s}_{x}f\|_{L^{p}}.
\]
Let \( f = f(x,t) \) be a real or complex valued function with \( x, t \in \mathbb{R} \), and let \( T > 0 \). Let \( B \) denote any one of the above-mentioned spaces. Fix \( 1 \leq p, q \leq \infty \). We define the following mixed-norm spaces
\begin{equation*}
    \|f\|_{L^{p}_{T}B_{x}}:=\left(\int_{0}^{T}\|f(\cdot,t)\|_{B}^{p}\ dt\right)^{\frac{1}{p}}, \quad  \|f\|_{L^{p}_{T}B_{x}}:=\left(\int_{\mathbb{R}}\|f(\cdot,t)\|_{B}^{p}\ dt\right)^{\frac{1}{p}},
\end{equation*}
and 
\begin{equation*}
    \|f\|_{L^{p}_{x}L^{q}_{T}}=\left(\int_{\mathbb{R}}\left(\int_{0}^{T}|f(x,t)|^{q}\ dt\right)^{\frac{p}{q}} dx\right)^{\frac{1}{p}}.
\end{equation*}
Let $X$ denote one of the previous mixed-norm spaces defined above. We define its dyadic version $\widetilde{X}$ as
\begin{equation*}
    \|f\|_{\tilde{X}}=\left(\sum_{N}\|P_{N}f\|_{X}^{2}\right)^{\frac{1}{2}}.
\end{equation*}
Fix $s,b\in \mathbb{R}$. Now, we introduce $X^{s,b}$ and $Y^{s,b}$ as the Bourgain spaces associated with the linear part of the higher-order Benjamin-Ono equation and the Schrödinger equation given in \eqref{BO-NLS}, respectively. That is, $X^{s,b}$ and $Y^{s,b}$ are defined as the completion of $\mathcal{S}(\mathbb{R}^{2})$ under the following norm
\begin{equation*}
    \|r\|_{X^{s,b}}=\|r\|_{X^{s,b}_{\tau=b\xi|\xi|-a\xi^{3}}}:=\left(\int_{\mathbb{R}}\int_{\mathbb{R}}\langle \tau -b\xi|\xi|+a\xi^{3} \rangle^{2b}\langle \xi \rangle^{2s}|\widehat{r}(\xi,\tau)|^{2}d \xi d\tau   \right)^{\frac{1}{2}},
\end{equation*}
and,
\begin{equation}
    \|q\|_{Y^{s,b}}=\|q\|_{Y^{s,b}_{\tau=\alpha\xi^{2}}}:=\left(\int_{\mathbb{R}}\int_{\mathbb{R}}\langle \tau- \alpha\xi^{2} \rangle^{2b}\langle  \xi\rangle^{2b}|\widehat{q}(\xi,\tau)|^{2}d\xi d\tau\right)^{\frac{1}{2}},
\end{equation}
where $\langle x \rangle=(1+x^{2})^{\frac{1}{2}}.$

 In this work, we will also make use of a dyadic version of these spaces, as introduced in \cite{Tataru-WavesMap} in the context of wave maps. Fix $s,b\in \mathbb{R}$ and $1\leq q\leq \infty$. Then, $X^{s,b,q}$ is defined as the completion of $\mathcal{S}(\mathbb{R}^{2})$ under the following norm
\begin{equation*}
    \|r\|_{X^{s,b,q}}=\left(\sum_{N}\left(\sum_{L}\langle N\rangle^{sq}\langle L\rangle^{bq}\|P_{N}Q_{L}r\|_{L^{2}_{x,t}}^{q}\right)^{\frac{2}{q}}\right)^{\frac{1}{2}}.
\end{equation*}
Moreover, we define a time-localized version of these spaces. Fix $T>0$, and let $Z$ denote either $X^{s,b}$, $X^{s,b,q}$, or $Y^{s,b}$. Then, we define the space $Z_{T}$ equipped with the following norm
\begin{equation*}
    \|r\|_{Z_{T}}:=\inf \{\|\tilde{r}\|_{Z}\ |  \ \tilde{r}:\mathbb{R}\times \mathbb{R}\rightarrow\mathbb{C},\quad  \tilde{r}_{|\mathbb{R}\times [0,T]}=r  \},
\end{equation*}
where $r:\mathbb{R}\times [0,T] \rightarrow \mathbb{C}$.

The following are standard properties of the Bourgain spaces defined above (see Proposition 2.1 in \cite{HOBOinH1-Didier-Molinet}).

\begin{prop}\label{propimersionsbourgainspaces}
Fix $\delta>0$, $s\in \mathbb{R}$. Then it holds true
\begin{align}
    &\|r\|_{X^{s,\frac{1}{2}}}\lesssim \|r\|_{X^{s,\frac{1}{2},1}}\lesssim \|r\|_{X^{s,\frac{1}{2}+\delta}},\label{bourgainstandardtodiadic}\\
    & \|r\|_{L^{\infty}_{t}H^{s}_{x}}\lesssim \|\widehat{J^{s}_{x}r}\|_{L^{2}_{\xi}L^{1}_{\tau}}\lesssim \|r\|_{X^{s,\frac{1}{2},1}},\label{bourgaindiadictoHs}
\end{align}
and
\begin{equation}\label{Hstobourgain}
    \|f\|_{X^{s,-\frac{1}{2}+\delta}}+\|f\|_{Y^{s,-\frac{1}{2}+\delta}_{\tau=\pm \alpha\xi^{2}}}\lesssim \|f\|_{L^{1+\delta'}_{t}H^{s}_{x}},
\end{equation}
for $\delta'>0$ satisfying $1+\delta'=\frac{1}{1-\delta}$.
\end{prop}

\subsection{Linear estimates}
We have the following linear estimates associated with the Bourgain spaces defined in the preceding subsection  (see \cite{PonceBekiranov}, \cite{NDE-Tao}, \cite{Tataru-WavesMap}).

\begin{lemma} Let $s\in \mathbb{R}$ and $\frac{1}{2}<b\leq1$. Then
\begin{align}
    &\|\eta(t)V(t)\phi\|_{X^{s,\frac{1}{2},1}}\lesssim \|\phi\|_{H^{s}}, \label{linearestimate}\\ 
    &\|\eta(t)U(t)\phi\|_{Y^{s,b}}\lesssim \|\phi\|_{H^{s}}. \label{linearestimateNLS}
\end{align}
\begin{lemma}
Let $s\in \mathbb{R}$ and $\frac{1}{2}<b\leq 1$. Then
\begin{align}
    &\left\|\eta(t)\int_{0}^{t}V(t-t')g(t')dt'\right\|_{X^{s,\frac{1}{2},1}}\lesssim \|g\|_{X^{s,-\frac{1}{2},1}}\label{nonlinearestimate},\\
    &\left\|\eta(t)\int_{0}^{t}U(t-t')g(t')dt'\right\|_{Y^{s,b}}\lesssim \|g\|_{Y^{s,b-1}}. \label{nonlinearestimateNLS}
\end{align}
\end{lemma}
\end{lemma}

Molinet and Pilod established the following estimates in Bourgain spaces associated with the Higher-order Benjamin-Ono equation (see \cite{HOBOinH1-Didier-Molinet}).

\begin{lemma}\label{lemmaMergulhoLpBpurgain}
Let $T>0$ and $0\leq \theta \leq 1$. Then, it holds that
\begin{equation}\label{MergulhoLpBpurgain1}
    \|r\|_{L^{p_{\theta}}_{x,t}}\lesssim \|r\|_{\widetilde{L^{p_{\theta}}_{x,t}}}\lesssim \|r\|_{X^{0,\frac{\theta}{2}+}},
\end{equation}
and,
\begin{equation}\label{MergulhoLpBpurgain2}
     \|r\|_{L^{p_{\theta}}_{x,T}}\lesssim \|r\|_{X^{0,\frac{\theta}{2}+}_{T}},
\end{equation}
where $\frac{1}{p_{\theta}}=\frac{\theta}{8}+\frac{1-\theta}{2}$.
\end{lemma}

We further establish an analogous result to Lemma \ref{lemmaMergulhoLpBpurgain} in the context of the Bourgain spaces associated with the Schrödinger equation.
\begin{lemma}\label{lemmaimersionBourgainNLS} For all $0\leq \theta \leq 1$ and $T>0$ it holds that
\begin{equation}\label{imersionBourgainNLS1}
    \|q\|_{L^{p_{\theta}}_{x,t}}\lesssim \|q\|_{Y_{\tau=\pm\alpha \xi^{2}}^{0,\frac{\theta}{2}+}},
\end{equation}
and,
\begin{equation}\label{imersionBourgainNLS2}
    \|q\|_{L^{p_{\theta}}_{x,T}}\lesssim \|q\|_{Y^{0,\frac{\theta}{2}+}_{\tau=\pm\alpha \xi^{2},T}},
\end{equation}
where $\frac{1}{p_{\theta}}=\frac{\theta}{6}+\frac{1-\theta}{2}$.
\end{lemma}
\begin{proof} We will give the proof of these estimates only for the space $Y^{s,b}_{\tau=\alpha \xi^{2}}=Y^{s,b}$, and the other case follows from similar arguments. Note that, from Strichartz estimates associated with the Schrödinger equation  (see \cite[Theorem 4.2]{LinaresPonceNDE}), we have 
\begin{equation}\label{lpqinBourgainNLS1}
    \|U(t)\phi \|_{L^{6}_{x,t}}\lesssim \|\phi\|_{L^{2}}.
\end{equation}
Therefore, Lemma 3.3 in \cite{Ginibre-Lpq} ensures us that the estimate \eqref{lpqinBourgainNLS1} can be rewritten in the context of Bourgain's spaces as follows 
\begin{equation}\label{lpqinBourgainNLS3}
    \|q\|_{L^{6}_{x,t}}\lesssim \|q\|_{Y^{0,\frac{1}{2}+}}.
\end{equation}
On the other hand, we also have the following identification
\begin{equation}\label{lpqinBourgainNLS2}
    \|q\|_{L^{2}_{x,t}}=\|q\|_{Y^{0,0}}.
\end{equation}
Thus, from estimates \eqref{lpqinBourgainNLS3} and \eqref{lpqinBourgainNLS2}, along with Stein's interpolation theorem, we deduce that
\begin{equation}
    \|q\|_{L^{6}_{x,t}}\lesssim \|q\|_{Y^{0,\frac{\theta}{2}+}},
\end{equation}
for all $\theta \in [0,1]$ and $p_{\theta}$ as in the statement. Thus, it proves the estimate \eqref{imersionBourgainNLS1}. Now, the estimate \eqref{imersionBourgainNLS2} follows readily from \eqref{imersionBourgainNLS1} and the definition of $Y_{T}^{0,\frac{\theta}{2}+}$.
\end{proof}

The next lemma, proved by Sjölin in \cite{globalmaximalestNLS-Sjolin}, establishes a maximal estimate for the linear group associated with the Schrödinger equation. 

\begin{lemma}\label{maximalestimateNLS}
Let $s>\frac{1}{2} $ and $0<T\leq 1$. The following estimate holds true
\begin{equation}\label{eqmaximalestimateNLS}
    \|U(t)\phi\|_{L^{2}_{x}L^{\infty}_{T}}\lesssim \|\phi\|_{H^{s}_{x}}.
\end{equation}
\end{lemma}
\vspace{-2em}
\subsection{Fractional Leibniz rules and commutator estimates}

The following result is the classical fractional Leibniz rule established by Kenig, Ponce, and Vega in the appendix of \cite{KDV-scattering-KPV}

\begin{lemma}[Fractional Leibniz rule] \label{fracleib}
Let $0<s<1$, $0\leq s_{1}, s_{2}\leq s$ with $s=s_{1}+s_{2}.$ Let $1<p,p_{1},p_{2}<\infty$ such that $\frac{1}{p}=\frac{1}{p_{1}}+\frac{1}{p_{2}}$. Then, 
    \begin{equation}\label{eqfracleib}
        \|D_{x}^{s}(fg)-fD_{x}^{s}g-gD_{x}^{s}f\|_{L^{p}}\lesssim \|D_{x}^{s_{1}}g\|_{L^{p_{1}}}\|D_{x}^{s_{2}}f\|_{L^{p_{2}}}.
    \end{equation}
    Moreover, for $s_{1}=0$, the value $p_{1}=\infty$ is allowed.
\end{lemma}
\begin{remark}
A variant version of \eqref{eqfracleib}, which also covers the case $s\geq1$, was proved by L. Grafakos and S. Oh in \cite{Grafakos03062014} (see also \cite{KatoPonceLeibniz-Li}).
\end{remark}
 \indent Now, we have the following localized version of the estimate \eqref{eqfracleib} established by Molinet in \cite{Molinet-BO-circle}
    \begin{lemma}\label{LemmaFractionalDerivativeMolinet} Let $\alpha \geq 0$ and $1<p<\infty$. Then,
    \begin{equation}\label{FractionalDerivativeMolinet}
    \|D_{x}^{\alpha}P_{+}(fP_{-}\partial_{x}g)\|_{L^{^p}}\lesssim \|D_{x}^{\alpha_{1}}f\|_{L^{p_{1}}}\|D_{x}^{\alpha_{2}}g\|_{L^{p_{2}}},
    \end{equation}
    with $1<p_{j}<\infty,\ \frac{1}{p_{1}}+\frac{1}{p_{2}}=\frac{1}{p}$, $\alpha_{1}\geq \alpha,\alpha_{2}\geq 0$, and $\alpha_{1}+\alpha_{2}=1+\alpha$.
    \end{lemma}
The following commutator lemma is due to Li (see Theorem 5.1 and Corollary 5.3 in \cite{KatoPonceLeibniz-Li}), which extends the classical Kato–Ponce commutator estimates given in \cite{KDV-scattering-KPV}.
\begin{lemma}[Commutator estimates]\label{commestim}
Let $1<p<\infty$. Let $1<p_{1}, p_{2}, p_{3}, p_{4}\leq \infty$, satisfy
\begin{equation*}
    \frac{1}{p_{1}}+\frac{1}{p_{2}}=\frac{1}{p_{3}}+\frac{1}{p_{4}}=\frac{1}{p}.
\end{equation*}
Then, for all $f,g\in S(\mathbb{R})$, the following estimates are true 
\begin{itemize}
    \item[(i)] If $0<s\leq 1$, then
    \begin{equation}
        \|[D_{x}^{s};f]g\|_{L^{p}}\lesssim \|D_{x}^{s-1}\partial_{x}f\|_{L^{p_{1}}}\|g\|_{L^{p_{2}}}.
    \end{equation}
    \item[(ii)] If $s>1$, then
    \begin{equation}
        \|[D_{x}^{s};f]g\|_{L^{p}}\lesssim \|D_{x}^{s}f\|_{L^{p_{1}}}\|g\|_{L^{p_{2}}}+\|\partial_{x}f\|_{L^{p_{3}}}\|D_{x}^{s-1}g\|_{L^{p_{4}}}.
    \end{equation}
\end{itemize}
\end{lemma}
\begin{remark}
The estimates in Lemmas \ref{fracleib} and \ref{commestim} are still valid for $\mathcal{H}D_{x}^{s}$ instead of $D_{x}^{s}.$
\end{remark}
Next, we stated the Coifman-Meyer lemma, which will be our key tool to derive the generalized commutator estimates used in this work.

\begin{lemma}[Coifman-Meyer] \label{CoifmanMeyerThm}
Let $\sigma \in C^{\infty}(\mathbb{R}^{n}\times \mathbb{R}^{n}-\{(0,0)\})$ satisfy
\begin{equation}\label{coifmanmeyer1}
|\partial_{\xi_{1}}^{\gamma}\partial_{\xi_{2}}^{\theta}\sigma(\xi_{1},\xi_{2})|\lesssim_{\gamma,\theta} (|\xi_{1}+\xi_{2}|)^{-|\gamma|-|\theta|}
\end{equation}
for $(\xi_{1},\xi_{2})\neq (0,0)$ and any $\gamma,\theta \in (\mathbb{Z}_{+})^{n}$. If $\sigma(D)$ denotes the bilinear operator
\begin{equation*}
    \sigma(D)(u,v)(x)=\iint e^{ix\cdot (\xi_{1}+\xi_{2})}\sigma(\xi_{1},\xi_{2})\widehat{u}(\xi_{1}) \widehat{v\ }(\xi_{2}) d\xi_{1}d\xi_{2},
\end{equation*}
then
$$\|\sigma(D)(u,v)\|_{L^{2}}\lesssim \|u\|_{L^{\infty}}\|v\|_{L^{2}}.$$
\end{lemma}

In the following lemma, we establish commutator estimates that allow us to simplify the problematic terms in \eqref{integraisproblematicas} below, which arise when attempting to perform standard energy estimates in the system \eqref{BO-NLS}. The key observation is that differentiation of a product distributes across the factors, giving rise to commutator terms that can be handled more effectively. Our approach follows the method introduced by Kwon in \cite{5KDV-Kwon}, where he used the Coifman-Meyer lemma to do this.

\begin{lemma}[Generalized commutator estimates] \label{geralcommestm} Let $s>1$. Then, the estimates
\vspace{-3em}
\begin{spacing}{1.5}
    \begin{equation}\label{gce1}
    \begin{split}
        \displaystyle \left\|D_{x}^{s}\partial_{x}\mathcal{H}(u\partial_{x}v)+uD_{x}^{s-1}\partial_{x}^{3}v+\partial_{x}vD_{x}^{s-1}\partial_{x}^{2}u+(s+1)\partial_{x}uD_{x}^{s-1}\partial_{x}^{2}v\right\|_{L^{2}}&\lesssim \|\partial_{x}^{2}v\|_{L^{\infty}}\|D_{x}^{s}u\|_{L^{2}}\\
        &\quad +\|\partial_{x}^{2}u\|_{L^{\infty}}\|D_{x}^{s}v\|_{L^{2}},
        \end{split}
   \end{equation}
   \vspace{-3em}
   \begin{align}
        &\left\|D_{x}^{s}(u\mathcal{H}\partial_{x}^{2}v)-uD_{x}^{s+1}\partial_{x}v-s\partial_{x}uD_{x}^{s+1}v\right\|_{L^{2}}\lesssim \|\partial_{x}^{2}u\|_{L^{\infty}}\|D_{x}^{s}v\|_{L^{2}}+\|\mathcal{H}\partial_{x}^{2}v\|_{L^{\infty}}\|D_{x}^{s}u\|_{L^{2}},\label{gce2}\\
    &\left\|D_{x}^{s}(\partial_{x}u\mathcal{H}\partial_{x}v)-\mathcal{H}\partial_{x}vD_{x}^{s}\partial_{x}u-\partial_{x}uD_{x}^{s+1}v\right\|_{L^{2}}\lesssim \|\partial_{x}^{2}u\|_{L^{\infty}}\|D_{x}^{s}v\|_{L^{2}}+\|\mathcal{H}\partial_{x}^{2}v\|_{L^{\infty}}\|D_{x}^{s}u\|_{L^{2}},\label{gce3}
    \end{align}
\end{spacing}
\vspace{-2em}
hold true.
\end{lemma}
\begin{proof}
The proof of this lemma can be obtained by an application of the Coifman-Meyer lemma (Lemma \ref{CoifmanMeyerThm}), and it is similar to the proof of the Kato-Ponce commutator estimate (see \cite{Kato-Ponce-NavierStokes}). Let us prove \eqref{gce1}. Note that, since $D_{x}^{1}=\partial_{x}\mathcal{H}$, we have
\begin{equation*}
    D^{s}_{x}\partial_{x}\mathcal{H}(u\partial_{x} v)=D_{x}^{s+1}(u\partial_{x}v)
\end{equation*}
In this case, we can write this pseudo-differential operator in the following integral form
\begin{equation}
    D_{x}^{s+1}(u\partial_{x}v)(x)=c\int_{\mathbb{R}}\int_{\mathbb{R}}e^{ix(\xi_{1}+\xi_{2})}\sigma_{1}(\xi_{1},\xi_{2})\widehat{u}(\xi_{1})\widehat{v}(\xi_{2})d\xi_{1}d\xi_{2},
\end{equation}
where we have that
\begin{equation}
    \sigma_{1}(\xi_{1},\xi_{2}):=i|\xi_{1}+\xi_{2}|^{s+1}\xi_{2}
\end{equation}
is the symbol of the pseudo-differential operator above. Also observe that
\begin{align}\label{pgce10}
    \sigma_{1}(\xi_{1},\xi_{2})&=i|\xi_{1}+\xi_{2}|^{s-1}(\xi_{1}+\xi_{2})^{2}\xi_{2}\nonumber \\
    &=i|\xi_{1}+\xi_{2}|^{s-1}\xi_{1}^{2}\xi_{2}+2i|\xi_{1}+\xi_{2}|^{s-1}\xi_{2}^{2}\xi_{1}+i|\xi_{1}+\xi_{2}|^{s-1}\xi_{2}^{3}.
\end{align}
The idea now is to deal with such functions in the above decomposition at low-low, low-high, and high-high frequencies (paraproducts) and then apply the Coifman-Meyer theorem. Then let $\{P_{N}\}_{N\in 2^{\mathbb{Z}}}$ be the standard Littlewood-Paley projectors, namely, we have $
P_{N}(f):=\mathcal{F}^{-1}_{x}\big[\phi_{N}\mathcal{F}_{x}(f)\big]
$ where $\phi_{N}(\xi)=\eta(N^{-1}\xi)$ ($\phi$ as in Notations \ref{Notations}). 
Then define smooth functions $\pi_{hh},\pi_{lh},\pi_{hl}:\mathbb{R}^{2}\rightarrow [0,+\infty)$ such that $\pi_{hh}+\pi_{lh}+\pi_{hl}=1$ as follows
\begin{align*}
    &\pi_{hh}(\xi_{1},\xi_{2})=\sum_{j}\sum_{|k-j|\leq 4}\phi_{2^{j}}(\xi_{1})\phi_{2^{k}}(\xi_{2});
    \quad \pi_{lh}(\xi_{1},\xi_{2})=\sum_{j}\sum_{k>j+4}\phi_{2^{j}}(\xi_{1})\phi_{2^{k}}(\xi_{2});\\
    &\pi_{hl}(\xi_{1},\xi_{2})=\sum_{j}\sum_{k<j-4}\phi_{2^{j}}(\xi_{1})\phi_{2^{k}}(\xi_{2}).
\end{align*}
Then, we decompose any symbol $\sigma$ into the low-high, high-low, and high-high paraproducts as follows
\begin{equation*}
    \sigma(\xi_{1},\xi_{2})=\sigma^{lh}(\xi_{1},\xi_{2})+\sigma^{hl}(\xi_{1},\xi_{2})+\sigma^{hh}(\xi_{1},\xi_{2}),
\end{equation*}
where $\sigma^{k}(\xi_{1},\xi_{2}):=\sigma(\xi_{1},\xi_{2})\pi_{k}(\xi_{1},\xi_{2})$ for $k=hl,\ lh$ or $hh$. Let us look at frequencies high-low first, i.e., $(\xi_{1},\xi_{2})$ in the support of $\pi_{hl}$. In this case, we have that
$
|\xi_{1}|\geq 4|\xi_{2}|.
$
Note first that, since $s>1$, then we can write
\begin{align*}
    |\xi_{1}+\xi_{2}|^{s-1}\xi_{1}^{2}\xi_{2}\frac{1}{\xi_{1}\xi_{2}^{2}|\xi_{1}|^{s-1}}\pi_{hl}(\xi_{1},\xi_{2})&=\frac{\xi_{1}}{\xi_{2}}\left( 1+\frac{\xi_{2}}{\xi_{1}} \right)^{s-1}\pi_{hl}(\xi_{1},\xi_{2})\\
    &=\frac{\xi_{1}}{\xi_{2}}\sum_{k=1}^{\infty}\binom{s-1}{k}\left(\frac{\xi_{2}}{\xi_{1}}\right)^{k}\pi_{hl}(\xi_{1},\xi_{2})\\
    &=\left[\frac{\xi_{1}}{\xi_{2}}+\sum_{k=1}^{\infty}\binom{s-1}{k}\left(\frac{\xi_{2}}{\xi_{1}}\right)^{k-1}\right]\pi_{hl}(\xi_{1},\xi_{2}).
\end{align*}
Thus,
\begin{equation}
\left(|\xi_{1}+\xi_{2}|^{s-1}\xi_{1}^{2}\xi_{2}-\xi_{1}^{2}\xi_{2}|\xi_{1}|^{s-1}\right)\frac{1}{\xi_{1}\xi_{2}^{2}|\xi_{1}|^{s-1}}\pi_{hl}(\xi_{1},\xi_{2})=\sum_{k=1}^{\infty}\binom{s-1}{k}\left(\frac{\xi_{2}}{\xi_{1}}\right)^{k-1}\pi_{hl}(\xi_{1},\xi_{2}),
\end{equation}
where the above series and all its derivatives converge absolutely since $\displaystyle \left|\frac{\xi_{2}}{\xi_{1}}\right|\leq \frac{1}{4}$. So, defining
\begin{equation}\label{pgce12}
    \sigma_{1,1}(\xi_{1},\xi_{2}):=|\xi_{1}+\xi_{2}|^{s-1}\xi_{1}^{2}\xi_{2}-\xi_{1}^{2}\xi_{2}|\xi_{1}|^{s-1},
\end{equation}
we can see readily that $ \sigma_{1,1}^{hl}(\xi_{1},\xi_{2})\frac{1}{\xi_{1}\xi_{2}^{2}|\xi_{1}|^{s-1}}$ is a smooth Coifman-Meyer multiplier, i.e., it is in $ C^{\infty}(\mathbb{R}\times \mathbb{R}-\{0,0\})$ and satisfies hypothesis \eqref{coifmanmeyer1}. Thus, from Lemma \ref{CoifmanMeyerThm}, we have that
\begin{align}
    &\left\| \int_{\mathbb{R}}\int_{\mathbb{R}} e^{ix(\xi_{1}+\xi_{2})}\sigma_{1,1}^{hl}(\xi_{1},\xi_{2})\widehat{u}(\xi_{1})\widehat{v}(\xi_{2})\  d\xi_{1}d\xi_{2}\right\|_{L^{2}} \nonumber\\
    =&\left\| \int_{\mathbb{R}}\int_{\mathbb{R}} e^{ix(\xi_{1}+\xi_{2})}\sigma_{1,1}^{hl}(\xi_{1},\xi_{2})\frac{1}{\xi_{1}(\xi_{2})^{2}|\xi_{1}|^{s-1}}\widehat{D_{x}^{s-1}\partial_{x}u}(\xi_{1})\widehat{\partial_{x}^{2}v}(\xi_{2})\ d\xi_{1}d\xi_{2} \right\|_{L^{2}}
    \lesssim\|\partial_{x}^{2}v\|_{L^{\infty}}\|D_{x}^{s}u\|_{L^{2}}.\label{pgce13}
\end{align}
Now, regarding low-high frequencies, i.e., $(\xi_{1},\xi_{2})$ in the support of $\pi_{lh}$ satisfying
$
\left|\frac{\xi_{1}}{\xi_{2}}\right|\leq \frac{1}{4},
$
we estimate the terms of $\sigma_{11}^{lh}$ separately. In fact, we can write
\begin{align}
    &|\xi_{1}+\xi_{2}|^{s-1}\xi_{1}^{2}\xi_{2}\frac{1}{\xi_{1}^{2}\xi_{2}|\xi_{2}|^{s-1}}\pi_{hl}(\xi_{1},\xi_{2})=\left(1+\frac{\xi_{1}}{\xi_{2}}\right)^{s-1}\pi_{hl}(\xi_{1},\xi_{2}),\label{pgec14}\\
    &\xi_{1}^{2}\xi_{2}|\xi_{1}|^{s-1}\frac{1}{\xi_{1}\xi_{2}^{2}|\xi_{1}|^{s-1}}\pi_{hl}(\xi_{1},\xi_{2})=\frac{\xi_{1}}{\xi_{2}}\pi_{hl}(\xi_{1},\xi_{2})\label{pgce15},
\end{align}
and, in this case, it is not difficult to check that the expressions in \eqref{pgec14} and \eqref{pgce15} define a smooth Coifman-Meyer multiplier. Therefore, from Lemma \ref{CoifmanMeyerThm} and calculations similar to those made in \eqref{pgce13}, we obtain
\begin{equation}\label{pgce16}
\left\| \int_{\mathbb{R}}\int_{\mathbb{R}} e^{ix(\xi_{1}+\xi_{2})}\sigma_{1,1}^{lh}(\xi_{1},\xi_{2})\widehat{u}(\xi_{1})\widehat{v}(\xi_{2})\  d\xi_{1}d\xi_{2}\right\|_{L^{2}}
\lesssim \|\partial_{x}^{2}u\|_{L^{\infty}}\|D_{x}^{s}v\|_{L^{2}}+\|\partial_{x}^{2}v\|_{L^{\infty}}\|D_{x}^{s}u\|_{L^{2}}. 
\end{equation}
\indent When it comes to high-high frequencies we have the problem that $|\xi_{1}+\xi_{2}|^{s-1}\xi_{1}^{2}\xi_{2}$ does not define a smooth function on $C^{\infty}(\mathbb{R}\times \mathbb{R}-\{(0,0)\})$, since it has singularities at $(\xi,-\xi)$, and obviously the support of $\pi_{hh}$ contains points to this term. So, we cannot apply the Coifman-Meyer lemma here. To get around this problem we approach this case using the Littlewood-Paley projectors defined earlier, as it is done in \cite{5KDV-Kwon}.
In fact, define for all $j\in \mathbb{Z}$
$$
\widetilde{P}_{2^{j}}:=\sum_{|k-j|\leq 4}P_{2^{k}}.
$$
Thus, notice that
\begin{align*}
&\int_{\mathbb{R}}\int_{\mathbb{R}}e^{ix(\xi_{1}+\xi_{2})}|\xi_{1}+\xi_{2}|^{s-1}\xi_{1}^{2}\xi_{2}\ \pi_{hh}(\xi_{1},\xi_{2})\widehat{u}(\xi_{1})\widehat{v}(\xi_{2})\ d\xi_{1}d\xi_{2}\\
=&\sum_{j}\sum_{|k-j|\leq 4}\int_{\mathbb{R}}\int_{\mathbb{R}}e^{ix(\xi_{1}+\xi_{2})}|\xi_{1}+\xi_{2}|^{s-1}\xi_{1}^{2}\xi_{2}\widehat{P_{2^{j}}u}(\xi_{1})\widehat{P_{2^{k}}v}(\xi_{2})\ d\xi_{1}d\xi_{2}\\
=&i\sum_{N}D_{x}^{s-1}(P_{N}\partial_{x}^{2}u\ \widetilde{P}_{N}\partial_{x}v).
\end{align*}
Then, using the Littlewood-Paley inequality and that $\sum_{N}\widetilde{P}_{N}=9$, it follows that
\begin{align}
&\left\| \int_{\mathbb{R}}\int_{\mathbb{R}}e^{ix(\xi_{1}+\xi_{2})}|\xi_{1}+\xi_{2}|^{s-1}\xi_{1}^{2}\xi_{2}\ \pi_{hh}(\xi_{1},\xi_{2})\widehat{u}(\xi_{1})\widehat{v}(\xi_{2})\ d\xi_{1}d\xi_{2} \right\|_{L^{2}}\nonumber\\
\lesssim& \sum_{j}\left(\sum_{N} \left\|P_{2^{-j}N}D_{x}^{s-1}(P_{N}\partial_{x}^{2}u\ \widetilde{P}_{N}\partial_{x}v)\right\|_{L^{2}}^{2} \right)^{\frac{1}{2}}\nonumber \\
\lesssim&_{s} \sum_{j}2^{-j(s-1)}\left(\sum_{N} 
 N^{2(s-1)}\|P_{N}\partial_{x}^{2}u\|_{L^{\infty}}^{2}\|\widetilde{P}_{N}\partial_{x}v\|_{L^{2}}^{2}\right)^{\frac{1}{2}}\nonumber \\
\lesssim & \|\partial_{x}^{2}u\|_{L^{\infty}}\sum_{j}2^{-j(s-1)}\left( 
\sum_{N}N^{2(s-1)} \|\widetilde{P}_{N}\partial_{x}v\|_{L^{2}}^{2}\right)^{\frac{1}{2}}\nonumber \\
\lesssim&_{s} \|\partial_{x}^{2}u\|_{L^{\infty}}\sum_{j}2^{-j(s-1)}\|D_{x}^{s-1}\partial_{x}v\|_{L^{2}}\nonumber \\
\lesssim&_{s}\|\partial_{x}^{2}u\|_{L^{\infty}}\|D_{x}^{s}v\|_{L^{2}}.
\label{pgce17}
\end{align}
On the other hand, regarding the second term of $\sigma_{1,1}^{hh}$, still at high-high frequencies, we can write
\begin{align*}
    \xi_{1}^{2}\xi_{2}|\xi_{1}|^{s-1}\frac{1}{\xi_{1}\xi_{2}^{2}|\xi_{1}|^{s-1}}\pi_{hh}(\xi_{1},\xi_{2})=\frac{\xi_{1}}{\xi_{2}}\pi_{hh}(\xi_{1},\xi_{2}),
\end{align*}
and thus, $  \xi_{1}^{2}\xi_{2}|\xi_{1}|^{s-1}\frac{1}{\xi_{1}\xi_{2}^{2}|\xi_{1}|^{s-1}}\pi_{hh}(\xi_{1},\xi_{2})$ is a smooth Coifman-Meyer multiplier. Therefore, from Lemma \ref{CoifmanMeyerThm}, we have that
\begin{equation}\label{pgct18}
   \left\| \int_{\mathbb{R}} \int_{\mathbb{R}}e^{ix(\xi_{1}+\xi_{2})}\xi_{1}^{2}\xi_{2}|\xi_{1}|^{s-1}\pi_{hh}(\xi_{1},\xi_{2})\widehat{u}(\xi_{1})\widehat{v}(\xi_{2})d\xi_{1}d\xi_{2}\right\|_{L^{2}}\lesssim \|\partial_{x}^{2}v\|_{L^{\infty}}\|D_{x}^{s}u\|_{L^{2}}.
\end{equation}
By \eqref{pgce13}, \eqref{pgce16}, \eqref{pgce17} and \eqref{pgct18}, it follows that
\begin{equation}\label{pgce19}
 \left\| \int_{\mathbb{R}} \int_{\mathbb{R}}e^{ix(\xi_{1}+\xi_{2})}\sigma_{1,1}(\xi_{1},\xi_{2})\widehat{u}(\xi_{1})\widehat{v}(\xi_{2})d\xi_{1}d\xi_{2}\right\|_{L^{2}}\lesssim \|\partial_{x}^{2}v\|_{L^{\infty}}\|D_{x}^{s}u\|_{L^{2}}+\|\partial_{x}^{2}u\|_{L^{\infty}}\|D_{x}^{s}v\|_{L^{2}}.
\end{equation}
Now, defining
\begin{equation}\label{pgce20}
    \sigma_{1,2}(\xi_{1},\xi_{2}):=|\xi_{2}+\xi_{1}|^{s-1}\xi_{2}^{2}\xi_{1}-\xi_{2}^{2}\xi_{1}|\xi_{2}|^{s-1},
\end{equation}
similarly to \eqref{pgce19}, by symmetry, we have that
\begin{equation}\label{pgce21}
 \left\| \int_{\mathbb{R}} \int_{\mathbb{R}}e^{ix(\xi_{1}+\xi_{2})}\sigma_{1,2}(\xi_{1},\xi_{2})\widehat{u}(\xi_{1})\widehat{v}(\xi_{2})d\xi_{1}d\xi_{2}\right\|_{L^{2}}\lesssim \|\partial_{x}^{2}v\|_{L^{\infty}}\|D_{x}^{s}u\|_{L^{2}}+\|\partial_{x}^{2}u\|_{L^{\infty}}\|D_{x}^{s}v\|_{L^{2}}.
\end{equation}
Finally, note that at low-high frequencies we can write
\begin{align*}
    |\xi_{1}+{\xi_{2}}|^{s-1}\xi_{2}^{3}\frac{1}{\xi_{1}^{2}\xi_{2}|\xi_{2}|^{s-1}}\pi_{lh}(\xi_{1},\xi_{2})&=\left(\frac{\xi_{2}}{\xi_{1}}\right)^{2}\left(1+\frac{\xi_{1}}{\xi_{2}}\right)^{s-1}\pi_{lh}(\xi_{1},\xi_{2})\\
    &=\left(\frac{\xi_{2}}{\xi_{1}}\right)^{2}\left[\sum_{k=0}^{\infty}\binom{s-1}{k}\left(\frac{\xi_{1}}{\xi_{2}}\right)^{k}\right]\pi_{lh}(\xi_{1},\xi_{2})\\
    &=\left[\left(\frac{\xi_{2}}{\xi_{1}}\right)^{2}+(s-1)\frac{\xi_{2}}{\xi_{1}}+\sum_{k=2}^{\infty}\binom{s-1}{k}\left(\frac{\xi_{1}}{\xi_{2}}\right)^{k-2} \right]\pi_{lh}(\xi_{1},\xi_{2}),
\end{align*}
and therefore,
\begin{align*}
    &\left(|\xi_{1}+\xi_{2}|^{s-1}\xi_{2}^{3}-\xi_{2}^{3}|\xi_{2}|^{s-1}-(s-1)\xi_{1}\xi_{2}^{2}|\xi_{2}|^{s-1}\right)\frac{1}{\xi_{1}^{2}\xi_{2}|\xi_{2}|^{s-1}}\pi_{lh}(\xi_{1},\xi_{2})\\
    =&\sum_{k=2}^{\infty}\binom{s-1}{k}\left(\frac{\xi_{1}}{\xi_{2}}\right)^{k-2}\pi_{lh}(\xi_{1},\xi_{2}).
\end{align*}
Thus, defining the symbol
\begin{equation}\label{pgce22}
    \sigma_{1,3}(\xi_{1},\xi_{2}):=|\xi_{1}+\xi_{2}|^{s-1}\xi_{2}^{3}-\xi_{2}^{3}|\xi_{2}|^{s-1}-(s-1)\xi_{1}\xi_{2}^{2}|\xi_{2}|^{s-1},
\end{equation}
we have that $\sigma_{1,3}^{lh}\frac{1}{\xi_{1}^{2}\xi_{2}|\xi_{2}|^{s-1}}$ is a smooth Coifman-Meyer multiplier. Therefore, from Lemma \ref{CoifmanMeyerThm}, it follows that
\begin{equation}\label{pgec23}
    \left\|\int_{\mathbb{R}}\int_{\mathbb{R}} e^{ix(\xi_{1}+\xi_{2})}\sigma_{1,3}^{lh}(\xi_{1},\xi_{2})\widehat{u}(\xi_{1})\widehat{v}(\xi_{2}) \ d\xi_{1}d\xi_{2}\right\|_{L^{2}}\lesssim \|\partial_{x}^{2}u\|_{L^{\infty}}\|D_{x}^{s}v\|_{L^{2}}.
\end{equation}
Regarding to high-low frequencies, we can write
\begin{align*}
    |\xi_{1}+\xi_{2}|^{s-1}\xi_{2}^{3}\frac{1}{\xi_{2}^{2}\xi_{1}|\xi_{1}|^{s-1}}\pi_{hl}(\xi_{1},\xi_{2})&=\frac{\xi_{2}}{\xi_{1}}\left(1+\frac{\xi_{2}}{\xi_{1}}\right)^{s-1}\pi_{hl}(\xi_{1},\xi_{2})=\left[\sum_{k=1}^{\infty}\binom{s-1}{k}\left(\frac{\xi_{2}}{\xi_{1}}\right)^{k}\right]\pi_{hl}(\xi_{1},\xi_{2}),\\
    \xi_{2}^{3}|\xi_{2}|^{s-1}\frac{1}{\xi_{1}^{2}\xi_{2}|\xi_{2}|^{s-1}}\pi_{hl}(\xi_{1},\xi_{2})&=\left(\frac{\xi_{2}}{\xi_{1}}\right)^{2}\pi_{hl}(\xi_{1},\xi_{2}),\\
    \xi_{1}\xi_{2}^{2}|\xi_{2}|^{s-1}\frac{1}{\xi_{1}^{2}\xi_{2}|\xi_{2}|^{s-1}}\pi_{hl}(\xi_{1},\xi_{2})&=\frac{\xi_{2}}{\xi_{1}}\pi_{hl}(\xi_{1},\xi_{2}),
\end{align*}
and it is not difficult to see that the above expressions define smooth Coifman-Meyer multipliers. Thus, the Lemma \ref{CoifmanMeyerThm} yields that
\begin{equation}\label{pgec24}
    \left\| \int_{\mathbb{R}}\int_{\mathbb{R}}e^{ix(\xi_{1}+\xi_{2})}\sigma_{1,3}^{hl}(\xi_{1},\xi_{2})\widehat{u}(\xi_{1})\widehat{v}(\xi_{2}) \ d\xi_{1}d\xi_{2} \right\|_{L^{2}}\lesssim \|\partial_{x}^{2}v\|_{L^{\infty}}\|D_{x}^{s}u\|_{L^{2}}+\|\partial_{x}^{2}u\|_{L^{\infty}}\|D_{x}^{s}v\|_{L^{2}}.
\end{equation}
At high-high frequencies we have the same problem as before regarding the term $\displaystyle |\xi_{1}+\xi_{2}|^{s-1}\xi_{2}^{3}$. But, also using the Littlewood-Paley projectors as in \eqref{pgce17}, we can obtain
\begin{equation}\label{pgec25}
    \|D_{x}^{s-1}(u\partial_{x}^{3}v)\pi_{hh}\|_{L^{2}}\lesssim \|\partial_{x}^{2}v\|_{L^{\infty}}\|D_{x}^{s}u\|_{L^{2}},
\end{equation}
where we observe that $D_{x}^{s-1}(u\partial_{x}^{3}v)\pi_{hh}$ does not make sense as a multiplication, but as a differential operator with the symbol $|\xi_{1}+\xi_{2}|^{s-1}(i\xi_{2})^{3}\pi_{hh}(\xi_{1},\xi_{2})$. Continuing, we can write
\begin{align*}
    &\xi_{2}^{3}|\xi_{2}|^{s-1}\frac{1}{\xi_{1}^{2}\xi_{2}|\xi_{2}|^{s-1}}\pi_{hh}(\xi_{1},\xi_{2})=\left(\frac{\xi_{1}}{\xi_{2}}\right)^{2}\pi_{hh}(\xi_{1},\xi_{2}),\\
    &\xi_{1}\xi_{2}^{2}|\xi_{2}|^{s-1}\frac{1}{\xi_{1}^{2}\xi_{2}|\xi_{2}|^{s-1}}\pi_{hh}(\xi_{1},\xi_{2})=\frac{\xi_{2}}{\xi_{1}}\pi_{hh}(\xi_{1},\xi_{2}),
\end{align*}
and thus, since we are considering high-high frequencies, the above terms define smooth Coifman-Meyer multipliers. Then, from Lemma \ref{CoifmanMeyerThm}, we deduce that
\begin{equation}\label{pgec26}
    \left\| \int_{\mathbb{R}}\int_{\mathbb{R}} e^{ix(\xi_{1}+\xi_{2})}(\sigma_{1,3}(\xi_{1},\xi_{2})-|\xi_{1}+\xi_{2}|^{s-1}\xi_{2}^{3})\pi_{hh}(\xi_{1},\xi_{2})\widehat{u}(\xi_{1})\widehat{v}(\xi_{2})\ d\xi_{1}d\xi_{2} \right\|_{L^{2}}\lesssim \|\partial_{x}^{2}u\|_{L^{\infty}}\|D_{x}^{s}v\|_{L^{2}}.
\end{equation}
It follows from \eqref{pgec23}-\eqref{pgec26} that
\begin{align}\label{pgec27}
      &\left\| \int_{\mathbb{R}}\int_{\mathbb{R}}e^{ix(\xi_{1}+\xi_{2})}\sigma_{1,3}(\xi_{1},\xi_{2})\widehat{u}(\xi_{1})\widehat{v}(\xi_{2}) \ d\xi_{1}d\xi_{2} \right\|_{L^{2}} 
      \lesssim\|\partial_{x}^{2}u\|_{L^{\infty}}\|D_{x}^{s}v\|_{L^{2}}+\|\partial_{x}^{2}v\|_{L^{\infty}}\|D_{x}^{s}u\|_{L^{2}}.
\end{align}
Hence, by \eqref{pgce10}, \eqref{pgce12}, \eqref{pgce19}, \eqref{pgce21}, \eqref{pgce22}, and \eqref{pgec27}, it follows \eqref{gce1}. 
\newline
\indent In the proof of the estimate \eqref{gce1}, we provided a systematic approach to constructing the commutator associated with the pseudo-differential operator $D_{x}^{s}\partial_{x}\mathcal{H}(u\partial_{x}v)$,  ensuring throughout that the resulting estimates did not involve more than \( s \) derivatives in the \( L^2 \)-norm of the functions \( u \) and \( v \). Therefore, building upon the techniques and arguments previously introduced, the proof of estimates \eqref{gce2} and \eqref{gce3} does not present any extra difficulty, and it will be omitted.
\end{proof}
Dawson, McGahagan and Ponce proved the following estimate in the appendix of \cite{NLSDecaimentoDeSolucoes}.

\begin{lemma}\label{commutatorhilbertestimate} Let $1<p,q<\infty$. Then, for all $f,g\in \mathcal{S}(\mathbb{R})$ and $l,m$ integers such that $l,m\geq 0$, the following estimate is true
\begin{equation}
    \|\partial_{x}^{l}[\mathcal{H};f]\partial_{x}^{m}g\|_{L^{p}}\lesssim_{l,m,p} \|\partial_{x}^{l+m}f\|_{L^{\infty}}\|g\|_{L^{p}}.
\end{equation}

\end{lemma}

The next lemma establishes estimates for given operators in terms of the Riesz and Bessel potentials.

\begin{lemma}\label{boundedoperator1}
The linear operators $T_{1}(D):=J_{x}^{-1}\mathcal{H}\partial_{x}^{2}-\partial_{x}$ and $T_{2}(D)=J_{x}^{-2}\partial_{x}^{3}+\partial_{x}$ are bounded in $L^{2}(\mathbb{R})$, i.e., there exists a constant $c_{1}>0$ such that
\begin{equation}
    \|T_{1}(D)f\|_{L^{2}}\leq c\|f\|_{L^{2}}, \quad \forall f\in L^{2}(\mathbb{R}),
\end{equation}
and 
\begin{equation}\label{boundedoperator2}
\|T_{2}(D)f\|_{L^{2}}\leq c\|f\|_{L^{2}}, \quad \forall f\in L^{2}(\mathbb{R}).
\end{equation}
\begin{proof}
In fact, it is sufficient to observe that $T_{j}(D)$ is a Fourier multiplier given by
\begin{equation*}
    T_{j}(D)f=\left(m_{j}(\xi)\widehat{f}\ \right)^{\vee},\ \quad j=1,2, \quad 
\end{equation*}
where
\begin{equation*}
     m_{1}(\xi)=\frac{i\xi}{(1+\xi^{2})^{\frac{1}{2}}}(|\xi|-(1+\xi^{2})^{\frac{1}{2}}), \quad m_{2}(\xi)=\frac{i\xi}{1+\xi^{2}},
\end{equation*}
and $m_{j}\in L^{\infty}(\mathbb{R})$. Then, the inequality \eqref{boundedoperator1} follows from Plancherel's identity.
\end{proof}
\end{lemma}
The following result is due to Kenig and Pilod (see Lemma 2.2 in \cite{KdvHierarchy-Kenig-Didier})

\begin{lemma}\label{sum3derivative}
Let $f,\ g$ and $h$ be smooth functions defined in $\mathbb{R}$. Then
\begin{equation}
    \int_{\mathbb{R}}\partial_{x}^{3}fgh \ dx+\int_{\mathbb{R}}f\partial_{x}^{3}gh \ dx+\int_{\mathbb{R}}fg\partial_{x}^{3}h =3\int_{\mathbb{R}}\partial_{x}f\partial_{x}g\partial_{x}h \ dx.
\end{equation}
\end{lemma}
Next, we state a crucial lemma that enables us to derive the \textit{a priori} estimates required in this work. Essentially, it provides an estimate in Sobolev spaces for the product of a term of the form \( e^{\pm iF} \) with an arbitrary function, where $F$ is the primitive of a function. This result is due to Molinet and Pilod (see \cite{HOBOinH1-Didier-Molinet}). With only minor modifications, one can also obtain a corresponding estimate for the difference of exponential terms.

\begin{lemma}\label{lemmaestimateexponencialinHs}
Let $2\leq p<\infty$ and $1\leq s\leq \frac{3}{2}$. Consider $F$ a real-valued function such that $r=\partial_{x}F$ belongs to $L^{2}(\mathbb{R})$. Then, it holds that
\begin{equation}\label{estimateexponencialinHs}
    \|J^{s}_{x}(e^{\pm iF}g)\|_{L^{p}}\lesssim (1+\|r\|_{H^{1}}^{2})\|J^{s}_{x}g\|_{L^{p}}.
\end{equation}
Furthermore, assume that $2\leq p\leq 4$ and consider $F_{1},F_{2}$ real-valued functions such that $\partial_{x}F_{j}=r_{j}$ belong to $L^ {2}(\mathbb{R})$. Then, it holds that
\begin{equation}\label{estimateexponencialinHs2}
    \|J^{s}_{x}[(e^{\pm iF_{1}}-e^{i\pm F_{2}})g]\|_{L^{p}}\lesssim (1+\|r_{1}\|_{H^{1}}^{2}+\|r_{2}\|_{H^{1}}^{2})( \|r_{1}-r_{2}\|_{H^{1}}+\|e^{\pm iF_{1}}-e^{i\pm F_{2}}\|_{L^{\infty}})\|J^{s}_{x}g\|_{L^{p}}.
\end{equation}
\end{lemma}

\begin{proof}
We give only the proof of \eqref{estimateexponencialinHs2} since \eqref{estimateexponencialinHs} was already established in \cite{HOBOinH1-Didier-Molinet}. It follows from similar arguments to those used in Lemma 2.7 in \cite{BOinL2-Didier}. First, note that
\begin{align}\label{extexp1,1}
    &\|J_{x}^{s}(e^{\pm iF_{1}}-e^{\pm iF_{2}})g\|_{L^{p}}\nonumber\\
    &\lesssim \|(e^{\pm iF}-e^{\pm iF_{2}})g\|_{L^{p}}+\|D_{x}^{s}(e^{\pm iF_{1}}-e^{\pm iF_{2}})g\|_{L^{p}}\nonumber \\
    &\lesssim \|e^{\pm i F_{1}}-e^{\pm i F_{2}}\|_{L^{\infty}}\|g\|_{L^{p}}+\|D_{x}^{s}(P_{lo}(e^{\pm iF_{1}}-e^{{\pm iF_{2}}})g)\|_{L^{p}}+\|D_{x}^{s}(P_{hi}(e^{\pm iF_{1}}-e^{\pm i F_{2}})g)\|_{L^{p}}.
\end{align}
On the one hand, from the fractional Leibniz rule (Lemma \ref{fracleib}) and the Bernstein inequality, it follows that
\begin{align}\label{estexplo1,1}
    \|D_{x}^{s}(P_{lo}(e^{\pm iF_{1}}-e^{\pm iF_{2}} )g)\|_{L^{p}}&\lesssim \|D_{x}^{s}P_{lo}(e^{\pm iF_{1}}-e^{\pm iF_{2}})\|_{L^{\infty}}\|g\|_{L^{p}}+\|P_{lo}(e^{\pm iF_{1}}-e^{\pm iF_{2}})\|_{L^{\infty}}\|D_{x}^{s}g\|_{L^{p}} \nonumber \\
    &\lesssim \|e^{\pm iF_{1}}-e^{\pm iF_{2}}\|_{L^{\infty}}\|J_{x}^{s}g\|_{L^{p}}.
\end{align}
On the other hand, again by the fractional Leibniz rule, we have
\begin{align}\label{estexphi1,4}
    \|D_{x}^{s}(P_{hi}(e^{\pm iF_{1}}-e^{\pm iF_{2}})g)\|_{L^{p}}\lesssim \|P_{hi}(e^{\pm iF_{1}}-e^{\pm iF_{2}})\|_{L^{\infty}}\|D_{x}^{s}g\|_{L^{p}}+\|g\|_{L^{p_{1}}}\|D_{x}^{s}P_{hi}(e^{\pm iF_{1}}-e^{\pm iF_{2}})\|_{L^{p_{2}}},
\end{align}
where $\frac{1}{p_{1}}+\frac{1}{p_{2}}=\frac{1}{p}$ and $1<p_{1},p_{2}\leq \infty$. So, assuming that $1\leq s\leq 1+\frac{1}{p}$, we can choose $p_{1}$ and $p_{2}$ satisfying
\begin{equation*}
    s-1=\frac{1}{p}-\frac{1}{p_{1}}, \quad \frac{3}{2}-s=\frac{1}{2}-\frac{1}{p_{2}}.
\end{equation*}
Thus, from Sobolev embedding and the Bernstein inequality, we obtain
\begin{equation}\label{estexphi1,5}
\begin{split}
     \|g\|_{L^{p_{1}}}\|D_{x}^{s}P_{hi}(e^{\pm iF_{1}}-e^{\pm iF_{2}})\|_{L^{p_{2}}}&\lesssim \|D_{x}^{s-1}g\|_{L^{p}}\|D_{x}^{\frac{3}{2}-s}P_{hi}(e^{\pm iF_{1}}-e^{\pm iF_{2}})\|_{L^{2}} \\
    &\lesssim \|J^{s}_{x}g\|_{L^{p}}\|\partial_{x}^{2}P_{hi}(e^{\pm iF_{1}}-e^{\pm F_{2}})\|_{L^{2}} \\
    &\lesssim \|J^{s}_{x}g\|_{L^{p}}(1+\|r_{1}\|_{H^{1}}^{2}+\|r_{2}\|_{H^{1}}^{2})(\|r_{1}-r_{2}\|_{H^{1}}+\|e^{\pm iF_{1}}-e^{\pm iF_{2}}\|_{L^{\infty}}).
\end{split}
\end{equation}
From \eqref{estexphi1,4} and \eqref{estexphi1,5}, it follows that
\begin{equation}\label{estexphi1,7}
    \|D_{x}^{s}(P_{hi}(e^{\pm iF_{1}}-e^{\pm i F_{2}})g)\|_{L^{p}}\lesssim (1+\|r_{1}\|_{H^{1}}^{2}+\|r_{2}\|_{H^{1}}^{2})(\|r_{1}-r_{2}\|_{H^{1}}+\|e^{\pm iF_{1}}-e^{\pm iF_{2}}\|_{L^{\infty}})\|J^{s}_{x}g\|_{L^{p}}.
\end{equation}
Therefore, from \eqref{extexp1,1}, \eqref{estexplo1,1}, and \eqref{estexphi1,7}, it follows the result in the first case. Now, assume that $1+\frac{1}{p}< s\leq \frac{3}{2}$. In this case, we choose $p_{1}$ and $p_{2}$ satisfying
\begin{equation*}
    s-\frac{3}{2}+\frac{1}{p}=\frac{1}{p}-\frac{1}{p_{1}}, \quad 2-s-\frac{1}{p}=\frac{1}{2}-\frac{1}{p_{2}}.
\end{equation*}
Since $p\leq 4$, it is not difficult to see that $\displaystyle 0<\frac{1}{p_{2}}\leq \frac{1}{p}$. Therefore, from Sobolev embedding, we have
\begin{align}\label{estexphi1.6}
 \|g\|_{L^{p_{1}}}\|D_{x}^{s}P_{hi}(e^{\pm iF_{1}}-e^{\pm iF_{2}})\|_{L^{p_{2}}}&\lesssim \|D_{x}^{s-\frac{3}{2}+\frac{1}{p}}g\|_{L^{p}}\|D_{x}^{\frac{3}{2}}P_{hi}(e^{\pm iF_{1}}-e^{\pm iF_{2}})\|_{L^{2}}\nonumber \\
    &\lesssim \|J_{x}^{s}g\|_{L^{p}}(1+\|r_{1}\|_{H^{1}}^{2}+\|r_{2}\|_{H^{1}}^{2})(\|r_{1}-r_{2}\|_{H^{1}}+\|e^{\pm iF_{1}}-e^{\pm iF_{2}}\|_{L^{\infty}}).
\end{align}
Thus, the estimates \eqref{extexp1,1}, \eqref{estexplo1,1}, \eqref{estexphi1,4}, and \eqref{estexphi1.6} imply \eqref{estimateexponencialinHs}, and this concludes the result.
\end{proof}

\begin{remark}\label{remarkS>3/2}
A version of the above Lemma \ref{lemmaestimateexponencialinHs} could also be stated for $s>\frac{3}{2}$. This extension is worth emphasizing, as the proof of Theorem \ref{maintheorem} is presented only for the range $1\leq s\leq \frac{3}{2}$. However, analogous techniques can be employed to establish the result for $s>\frac{3}{2}$ as well.
\end{remark}

\section{Energy estimates and existence of smooth solutions for \eqref{BO-NLS}}

For simplicity in notation, let us consider the constants in \eqref{BO-NLS} satisfying $-a=b=c=d=\alpha=\beta =1.$

\subsection{Energy estimates for solutions of \eqref{BO-NLS}}
Firstly, we observe that we cannot obtain a priori estimates considering the usual energy associated with system \eqref{BO-NLS}, i.e.
$$
E^{s}(t)=\|r(t)\|_{H^{s}_{x}}^{2}+\|q(t)\|_{H^{s}_{x}}^{2}.
$$
In fact, we have terms that cannot be handled directly. More precisely,
\begin{equation}\label{integraisproblematicas}
\int_{\mathbb{R}} D_{x}^{s}\partial_{x}(|q|^{2})D_{x}^{s}r\ dx, \quad \int_{\mathbb{R}} D_{x}^{s}r\partial_{x}D_{x}^{s}(r\mathcal{H}\partial_{x}r)dx, \quad \textrm{and} \quad \int_{\mathbb{R}}D_{x}^{s}r\partial_{x}D_{x}^{s}\mathcal{H}(r\partial_{x}r)\ dx.
\end{equation}
\indent So, in order to obtain a priori estimates, we need to modify the usual energy. In this case, we define the following functional:

\begin{definition}
For all $t\geq 0$ and $s>\frac{3}{2}$, we define the modified energy as being
\begin{align} \label{em1}
    E_{m}^{s}(t)&=\|r(t)\|_{L^{2}_{x}}^{2}+\frac{1}{2}\|D_{x}^{s}r(t)\|_{L^{2}_{x}}^{2}+\|q(t)\|_{L^{2}_{x}}^{2}+\|D_{x}^{s}q(t)\|_{L^{2}_{x}}^{2}\\
    &\quad +\displaystyle\int_{\mathbb{R}} D_{x}^{s}(|q|^{2})D_{x}^{s-2}r\  dx +
     \tau_{s}\cdot \displaystyle\int_{\mathbb{R}}D_{x}^{s-1}\mathcal{H}r\cdot r\cdot D_{x}^{s-1}\partial_{x}r\ dx, \nonumber
\end{align}
where $\displaystyle \tau_{s}=\frac{2s+1}{3}.$
\end{definition}
Now, considering this modified energy functional, we can indeed establish the a priori estimates we need. Namely, we have the following result.

\begin{prop}\label{propenergyestimate}
Let $s>\frac{3}{2}$ and $T>0$. There exist positive constants $A_{s}$ and $B_{s}$ such that for any $(r,q)\in C([0,T]:H^{s}(\mathbb{R})\times H^{s}(\mathbb{R}))$ solution of \eqref{BO-NLS} satisfying
$$
\|(r(t),q(t))\|_{H^{s}_{x}\times H^{s}_{x}}:=\big(\|r\|_{L^{2}_{x}}^{2}+\|D_{x}^{s}r\|_{L^{2}_{x}}^{2}+\|q\|_{L^{2}_{x}}^{2}+\|D_{x}^{s}q\|_{L^{2}_{x}}^{2}\big)^{\frac{1}{2}}\leq \frac{1}{A_{s}},
$$
for all $t\in [0,T]$, the following estimates hold true.
\begin{enumerate}
    \item \textbf{Coercivity:}
    \begin{equation}\label{es1}
        \frac{1}{2}\left( \|r(t)\|_{H^{s}_{x}}^{2}+\|q(t)\|_{H^{s}_{x}}^{2}\right)\leq E_{m}^{s}(t)\leq \frac{3}{2}\left( \|r(t)\|_{H^{s}_{x}}^{2}+\|q(t)\|_{H^{s}_{x}}^{2}\right);
    \end{equation}

    \item \textbf{Energy estimate:}
    \begin{equation}\label{es2}
        \frac{d}{dt}E_{m}^{s}(t) \lesssim \left(1+\|\partial_{x}^{2}r\|_{L^{\infty}_{x}}+\|\mathcal{H}\partial^{2}_{x}r\|_{L^{\infty}_{x}}\right)E_{m}^{s}(t),
    \end{equation}
    and as a consequence,
    \begin{align}\label{es3}
        \sup_{t\in [0,T]}\|(r(t),q(t))\|_{H^{s}_{x}\times H^{s}_{x}}\leq 2\textrm{exp}\left\{B_{s} \left(T+\|\partial_{x}^{2}r\|_{L^{1}_{T}L^{\infty}_{x}}+\|\mathcal{H}\partial^{2}_{x}r\|_{L^{1}_{T}L^{\infty}_{x}}\right)\right\}\|(r_{0},q_{0})\|_{H^{s}\times H^{s}}.
    \end{align}
\end{enumerate}
\end{prop}
\begin{remark}\label{remarkenergyestimate1}
The estimates above can also be established for the system \eqref{BO-NLS-rescaled}, with constants that are uniform with respect to $\gamma$.
\end{remark}
\begin{proof}
The proof of \eqref{es1} is not difficult and will be omitted. Let us prove \eqref{es2}. Considering $E_{m}^{s}(t)$ as in \eqref{em1} and differentiating with respect to $t$, we have 
\begin{align} \label{estenerg1}
    \frac{d}{dt}E_{m}^{s}(t)& =  \frac{d}{dt}\left(\int_{\mathbb{R}}r^{2} dx\right)+\frac{1}{2}\frac{d}{dt}\left(\int_{\mathbb{R}}(D_{x}^{s}r)^{2}dx\right) +\frac{d}{dt}\left(\int_{\mathbb{R}}|q|^{2}dx\right)
    +\frac{d}{dt}\left(\int_{\mathbb{R}}|D_{x}^{s}q|^{2} dx\right)\nonumber \\
    &\quad +\frac{d}{dt}\left(\int_{\mathbb{R}}D_{x}^{s}(|q|^{2})D_{x}^{s-2}r dx\right) +\tau_{s}\frac{d}{dt}\left(\int_{\mathbb{R}} D_{x}^{s-1}\mathcal{H}r\cdot r\cdot D_{x}^{s-1}\partial_{x}r \ dx\right)\nonumber \\
     &= I+II+III+IV+V+VI.
\end{align}
Note that using the first equation in \eqref{BO-NLS}, we have
\begin{align*}
    I&=2\int_{\mathbb{R}}r(r\partial_{x}r)dx-2\int_{\mathbb{R}}r\partial_{x}(r\mathcal{H}\partial_{x}r)dx-2\int_{\mathbb{R}}r\partial_{x}\mathcal{H}(r\partial_{x}r)dx+2\int_{\mathbb{R}}r\partial_{x}(|q|^{2})dx \nonumber\\
    &= I_{1}+I_{2}+I_{3}+I_{4}.
\end{align*}
On the one hand, by integration by parts, we can readily obtain that
\begin{equation}
    I_{1}=0\quad \textrm{and } \quad I_{2}+I_{3}=0.
\end{equation}
On the other hand, by Hölder's inequality and the fact that $H^{1}$ is a algebra, it follows that
\begin{equation}
    |I_{4}|\lesssim \|r\|_{L^{2}_{x}}\||q|^{2}\|_{H^{1}_{x}}\lesssim \|r\|_{H^{s}_{x}}\|q\|_{H^{s}_{x}}^{2}.
\end{equation}
\indent Now, using once again the first equation in \eqref{BO-NLS}, we see that
\begin{align}
II&=\int_{\mathbb{R}}D^{s}_{x}r D_{x}^{s}(r\partial_{x}r) dx-\int_{\mathbb{R}}D_{x}^{s}rD_{x}^{s}\partial_{x}(r\mathcal{H}\partial_{x}r)dx-\int_{\mathbb{R}}D_{x}^{s}rD_{x}^{s}\partial_{x}\mathcal{H}(r\partial_{x}r)dx 
+ \int_{\mathbb{R}}D_{x}^{s}rD_{x}^{s}\partial_{x}(|q|^{2})dx \nonumber \\
&= II_{1}+II_{2}+II_{3}+II_{4}.
\end{align}
But note that by integrating by parts, we have
\begin{align*}
II_{1}=\int_{\mathbb{R}}D_{x}^{s}r[D_{x}^{s};r]\partial_{x}rdx+\int_{\mathbb{R}}D_{x}^{s}r\cdot r \cdot D_{x}^{s}\partial_{x}rdx&= \int_{\mathbb{R}}D_{x}^{s}r[D_{x}^{s};r]\partial_{x}rdx-\frac{1}{2}\int_{\mathbb{R}}(D_{x}^{s}r)^{2}\partial_{x}rdx\\
&= II_{1,1}+II_{1,2},
\end{align*}
Thus, by Lemma \ref{commestim} and the Sobolev embedding $H^{s-1}(\mathbb{R})\hookrightarrow L^{\infty}(\mathbb{R})$, since $s-1>\frac{1}{2}$, we get
\begin{align}
    &|II_{1,1}|+|II_{1,2}|\lesssim \|\partial_{x}r\|_{L^{\infty}_{x}}\|D^{s}_{x}r\|_{L^{2}_{x}}^{2}+\|\partial_{x}r\|_{L^{\infty}_{x}}\|D^{s-1}_{x}\partial_{x}r\|_{L^{2}_{x}}\lesssim \|r\|_{H^{s}_{x}}^{3}.
\end{align}
\indent Now, considering $II_{3}$, observe that
\begin{align*}
    II_{3}&=-\int_{\mathbb{R}}D_{x}^{s}r\left\{D_{x}^{s}\partial_{x}\mathcal{H}(r\partial_{x}r)+rD_{x}^{s-1}\partial_{x}^{3}r+(s+2)\partial_{x}rD_{x}^{s-1}\partial_{x}^{2}r\right\}dx\\
    &\quad +\int_{\mathbb{R}}D_{x}^{s}r\ r\ D_{x}^{s-1}\partial_{x}^{3}r\ dx+(s+2)\int_{\mathbb{R}}D_{x}^{s}r\partial_{x}rD_{x}^{s-1}\partial_{x}^{2}r \ dx\\
    &=II_{3,1}+II_{3,2}+II_{3,3}.
\end{align*}
Then, by Lemma \ref{geralcommestm},  we have that
\begin{equation}
    |II_{3,1}|\lesssim \|D_{x}^{s}r\|_{L^{2}_{x}}\|\partial_{x}^{2}r\|_{L^{\infty}_{x}}\|D_{x}^{s}r\|_{L^{2}_{x}}.
\end{equation}
We now observe that the term \( II_{3,2} \) presents a significant difficulty, as a straightforward modification of the energy functional is insufficient to control it. Indeed, such a modification invariably leads to residual terms that cannot be bounded 
suitably. However, it is possible to decompose \( II_{3,2} \) into a sum of terms that can be estimated, along with the remainder term \( II_{3,3} \). But, as we will see, the adjustment of the energy functional made in \eqref{em1} is effective in controlling \( II_{3,3} \).
A similar argument was employed by Tanaka in \cite{Tanaka-ThirdOrderBenjaminOno} in the context of the third-order Benjamin-Ono equation. In fact, integrating by parts and using that $\partial_{x}=-D_{x}^{1}\mathcal{H}$, we have that
\begin{align}
II_{3,2}&=-\int_{\mathbb{R}}D_{x}^{s}\partial_{x}r \ r \ D_{x}^{s-1}\partial_{x}^{2}r \ dx-\int_{\mathbb{R}}D_{x}^{s}r \ \partial_{x}r \ D_{x}^{s-1}\partial_{x}^{2}r \ dx\nonumber \\
&=\int_{\mathbb{R}}D_{x}^{s}\partial_{x}r \ r \ D_{x}^{s}\mathcal{H}\partial_{x}r \ dx-\int_{\mathbb{R}}D_{x}^{s}r \ \partial_{x}r \ D_{x}^{s-1}\partial_{x}^{2}r \ dx \label{eq1forII32}
\end{align}
But one more integration by parts yields
\begin{align}
\int_{\mathbb{R}}D_{x}^{s}\partial_{x}r \ r \ D_{x}^{s}\mathcal{H}\partial_{x}r \ dx&=-\int_{\mathbb{R}}D_{x}^{s}\partial_{x}^{2}r \ r \ D_{x}^{s}\mathcal{H}r \ dx-\int_{\mathbb{R}}D_{x}^{s}\partial_{x}r \ \partial_{x}r \ D_{x}^{s}\mathcal{H}r \ dx\nonumber \\
&=\int_{\mathbb{R}}\mathcal{H}(D_{x}^{s}\partial_{x}^{2}r \ r)D_{x}^{s}r \ dx+\int_{\mathbb{R}}D_{x}^{s}r\ \partial_{x}^{2}r \ D_{x}^{s}\mathcal{H}r \ dx+\int_{\mathbb{R}}D_{x}^{s}r \ \partial_{x}r \ D_{x}^{s}\mathcal{H}\partial_{x}r \ dx\nonumber \\
&=\int_{\mathbb{R}}[\mathcal{H};r]D_{x}^{s}\partial_{x}^{2}r \ D_{x}^{s}r \ dx-\int_{\mathbb{R}}D_{x}^{s}r \ r \ D_{x}^{s-1}\partial_{x}^{3}r \ dx +\int_{\mathbb{R}}D_{x}^{s}r \ \partial_{x}^{2}r \ D_{x}^{s}\mathcal{H}r \ dx \nonumber\\
&\quad -\int_{\mathbb{R}}D_{x}^{s}r \ \partial_{x}r \ D_{x}^{s-1}\partial_{x}^{2}r \ dx. \label{eq2forII32}
\end{align}
Therefore, from \eqref{eq1forII32} and \eqref{eq2forII32}, we deduce that
\begin{align}
II_{3,2}&= \frac{1}{2}\int_{\mathbb{R}}[\mathcal{H};r]D_{x}^{s}\partial_{x}^{2}r \ D_{x}^{s}r \ dx+\frac{1}{2}\int_{\mathbb{R}}D_{x}^{s}r \ \partial_{x}^{2}r \ D_{x}^{s}\mathcal{H}r \ dx-\int_{\mathbb{R}}D_{x}^{s}r \ \partial_{x}r \ D_{x}^{s-1}\partial_{x}^{2}r \ dx \nonumber \\
&=II_{3,2,1}+II_{3,2,2}+II_{3,2,3}.\label{eqforII32}
\end{align}
Now, by Hölder's inequality and Lemma \ref{commutatorhilbertestimate}, it follows that
\begin{align}
    |II_{3,2,1}|+ |II_{3,2,2}|\lesssim\|\partial_{x}^{2}r\|_{L^{\infty}_{x}}\|D_{x}^{s}r\|_{L^{2}_{x}}^{2}. 
\end{align}
\indent We will examine the terms $II_{3,2,3}$ and $II_{3,3}$ later, as they will require a detailed description of the terms contained in \( VI \). Specifically, certain terms arising in the decomposition of \( VI \) will cancel with \( II_{3,2,3} \) and \( II_{3,3} \).

Concerning $II_{2}$, note that 
\begin{align*}
II_{2}&= -\int_{\mathbb{R}}D_{x}^{s}r\left[D_{x}^{s}(\partial_{x}r\mathcal{H}\partial_{x}r)-\mathcal{H}\partial_{x}rD_{x}^{s}\partial_{x}r-\partial_{x}rD_{x}^{s+1}r\right]dx-\int_{\mathbb{R}}D_{x}^{s}r\ \mathcal{H}\partial_{x}r \ D_{x}^{s}\partial_{x}r \ dx\\
    &\quad -\int_{\mathbb{R}} D_{x}^{s}r\left[D_{x}^{s}(r\mathcal{H}\partial_{x}^{2}r)-D_{x}^{s+1}\partial_{x}r -s\partial_{x}rD_{x}^{s+1}r\right]dx
    +\int_{\mathbb{R}}D_{x}^{s}r\ r\ D_{x}^{s-1}\partial_{x}^{3}r\ dx\\
    &\quad +(s+1)\int_{\mathbb{R}}D_{x}^{s}r \partial_{x} r D_{x}^{s-1}\partial_{x}^{2}r\ dx\\
    &= II_{2,1}+II_{2,2}+II_{2,3}+II_{2,4}+II_{2,5}.
\end{align*}
From Lemma \ref{geralcommestm} we obtain
\begin{align}
    |II_{2,1}|+|II_{2,3}|\lesssim (\|\mathcal{H}\partial_{x}^{2}r\|_{L^{\infty}}+\|\partial_{x}^{2}r\|_{L^{\infty}})\|D_{x}^{s}r\|_{L^{2}}^{2}.
\end{align}
On the other hand, integrating by parts, we deduce that
\begin{equation}
    |II_{2,2}|=\frac{1}{2}\left|\int_{\mathbb{R}} (D_{x}^{s}r)^{2}\mathcal{H}\partial_{x}^{2}r\ dx\right|\lesssim \|\mathcal{H}\partial_{x}^{2}r\|_{L^{\infty}}\|D_{x}^{s}r\|_{L^{2}}^{2}.
\end{equation}
Furthermore, as in \eqref{eqforII32}, we have
\begin{align*}
II_{2,4}&= \frac{1}{2}\int_{\mathbb{R}}[\mathcal{H};r]D_{x}^{s}\partial_{x}^{2}r \ D_{x}^{s}r \ dx+\frac{1}{2}\int_{\mathbb{R}}D_{x}^{s}r \ \partial_{x}^{2}r \ D_{x}^{s}\mathcal{H}r \ dx-\int_{\mathbb{R}}D_{x}^{s}r \ \partial_{x}r \ D_{x}^{s-1}\partial_{x}^{2}r \ dx \nonumber \\
&=II_{2,4,1}+II_{2,4,2}+II_{2,4,3},
\end{align*}
where from Lemma \ref{commutatorhilbertestimate}, it follows that
\begin{equation}
    |II_{2,4,1}|+|II_{2,4,2}|\lesssim \|\partial_{x}^{2}r\|_{L^{\infty}_{x}}\|D_{x}^{s}r\|_{L^{2}_{x}}^{2}.
\end{equation}
In the same way as for the terms $II_{3,2,3},\ II_{3,3}$, we will treat the terms $II_{2,4,3},\ II_{2,5}$ and $II_{4}$ later.

Regarding the term $III$ we can use the fact that the $L^{2}$-norm of $q$ is conserved, and thus
 \begin{equation}
     III=0.
 \end{equation}
\indent Next, notice that using the second equation in \eqref{BO-NLS} and integrating by parts, we obtain
\begin{align*}
IV=2\textrm{Re}\left(\int_{\mathbb{R}} D_{x}^{s}\partial_{t}qD_{x}^{s}\bar{q} dx\right)&=-2\textrm{Re}\left(\int_{\mathbb{R}}D_{x}^{s}(i\partial_{x}^{2}q)D_{x}^{s}\bar{q} \ dx\right)+2\textrm{Re}\left(\int_{\mathbb{R}} D_{x}^{s}(irq)D_{x}^{s}\bar{q} \ dx\right)\\
&=2\textrm{Im}\left(\int_{\mathbb{R}}|D_{x}^{s}\partial_{x}q|^{2} \ dx\right)-2\textrm{Im}\left(\int_{\mathbb{R}} D_{x}^{s}(rq)D_{x}^{s}\bar{q} \ dx\right)\\
&=-2\textrm{Im}\left(\int_{\mathbb{R}} D_{x}^{s}(rq)D_{x}^{s}\bar{q} \ dx\right).
\end{align*}
Thus, using that $H^{s}(\mathbb{R})$ is a Banach algebra, since $s>\frac{3}{2}$, we have that
\begin{equation}
    |IV|\lesssim \|r\|_{H^{s}_{x}}\|q\|_{H^{s}_{x}}^{2}.
\end{equation}
 \indent Now, for the term $V$, observe that
\begin{align*}
    V&=\int_{\mathbb{R}}D_{x}^{s}(|q|^{2})D_{x}^{s-2}\partial_{t}r\ dx +2\textrm{Re}\left(\int_{\mathbb{R}}D_{x}^{s}(\partial_{t}q\bar{q})D_{x}^{s-2}r\ dx\right)  \\
    &=V_{1}+V_{2}.
\end{align*}
On the other hand, using the first equation in \eqref{BO-NLS}, it follows that
\begin{align*}
    V_{1}&=-\int_{\mathbb{R}} D_{x}^{s}(|q|^{2})D_{x}^{s-2}\partial_{x}^{3}r \ dx+\int_{\mathbb{R}} D_{x}^{s}(|q|^{2})D_{x}^{s-2}\mathcal{H}\partial_{x}^{2}r \ dx+\int_{\mathbb{R}} D_{x}^{s}(|q|^{2})D_{x}^{s-2}(r\partial_{x}r) \ dx\\
    &\quad -\int_{\mathbb{R}} D_{x}^{s}(|q|^{2})D_{x}^{s-2}\partial_{x}(r\mathcal{H}\partial_{x}r) \ dx-\int_{\mathbb{R}} D_{x}^{s}(|q|^{2})D_{x}^{s-2}\partial_{x}H(r\partial_{x}r) \ dx+\int_{\mathbb{R}} D_{x}^{s}(|q|^{2})D_{x}^{s-2}\partial_{x}(|q|^{2}) \ dx\\
    &=V_{1,1}+V_{1,2}+V_{1,3}+V_{1,4}+V_{1,5}+V_{1,6}.
\end{align*}
First, an integration by parts and the identity $\partial_{x}=-\mathcal{H}D_{x}^{1}$ yield
\begin{equation*}
    V_{1,1}=\int_{\mathbb{R}} D_{x}^{s}\partial_{x}(|q|^{2})D_{x}^{s-2}\mathcal{H}^{2}D_{x}^{2}r \ dx=-\int_{\mathbb{R}} D_{x}^{s}\partial_{x}(|q|^{2})D_{x}^{s}r \ dx.
\end{equation*}
Thus,
\begin{equation}
    V_{1,1}+II_{4}=0.
\end{equation}
Furthermore, we can write
\begin{align*}
    \sum_{j=2}^{6}V_{1,j}=&-\int_{\mathbb{R}} D_{x}^{s}(|q|^{2})D_{x}^{s}\mathcal{H}r \ dx+\int_{\mathbb{R}} D_{x}^{s}(|q|^{2})D_{x}^{s-2}(r\partial_{x}r) \ dx
    +\int_{\mathbb{R}} D_{x}^{s}(|q|^{2})D_{x}^{s-1}\mathcal{H}(r\mathcal{H}\partial_{x}r) \ dx\\
    &-\int_{\mathbb{R}} D_{x}^{s}(|q|^{2})D_{x}^{s-1}(r\partial_{x}r) \ dx-\int_{\mathbb{R}} D_{x}^{s}(|q|^{2})D_{x}^{s-1}\mathcal{H}(|q|^{2}) \ dx.
\end{align*}
Thus, since $H^{s}(\mathbb{R})$ and $H^{s-1}(\mathbb{R})$ are Banach algebras for $s>\frac{3}{2}$, it follows that
\begin{equation}
    \left| \sum_{i=2}^{6}V_{1,i} \right|\lesssim \|q\|_{H^{s}_{x}}^{2}\|r\|_{H^{s}_{x}}+\|q\|_{H^{s}_{x}}^{2}\|r\|_{H^{s}_{x}}^{2}+\|q\|_{H^{s}_{x}}^{4}
\end{equation}
On the other hand, now using the second equation in \eqref{BO-NLS} and integrating by parts, we see that
\begin{align*}
    V_{2}&=-2\textrm{Re}\left(i\int_{\mathbb{R}} D_{x}^{s}(\partial_{x}^{2}q\ \bar{q})D_{x}^{s-2}r \ dx\right)+2\textrm{Re}\left(\int_{\mathbb{R}} D_{x}^{s}(r|q|^{2})D_{x}^{s-2}r \ dx\right)\\
    &=2\textrm{Re}\left(i\int_{\mathbb{R}} D_{x}^{s-1}(\partial_{x}q\ \partial_{x}\bar{q})D_{x}^{s-1}r \ dx\right)+2\textrm{Re}\left(i\int_{\mathbb{R}} D_{x}^{s-1}(\partial_{x}q\ \bar{q})D_{x}^{s-1}\partial_{x}r \ dx\right)\\
    &+2\textrm{Re}\left(\int_{\mathbb{R}} D_{x}^{s}(r|q|^{2})D_{x}^{s-2}r \ dx\right).
\end{align*}
Therefore, using once again that $H^{s-1}(\mathbb{R})$ is a Banach algebra, we obtain
\begin{equation}
    |V_{2}|\lesssim \|q\|_{H^{s}_{x}}^{2}\|r\|_{H^{s}_{x}}+\|q\|_{H^{s}_{x}}^{2}\|r\|_{H^{s}_{x}}^{2}.
\end{equation}

We are left to examine the term \( VI \). But note that 
\begin{align*}
\frac{1}{\tau_{s}}\cdot VI&=\int_{\mathbb{R}}D_{x}^{s-1}\mathcal{H}\partial_{t}r\ rD_{x}^{s-1}\partial_{x}r\ dx+\int_{\mathbb{R}}D_{x}^{s-1}\mathcal{H}r\cdot \partial_{t}r\cdot D_{x}^{s-1}\partial_{x}r\ dx+\int_{\mathbb{R}}D_{x}^{s-1}\mathcal{H}r\cdot r\cdot D_{x}^{s-1}\partial_{x}\partial_{t}r\ dx\\
&=VI_{1}+VI_{2}+VI_{3}.
\end{align*}
Now, by the first equation in \eqref{BO-NLS}, it follows that 
\begin{align*}
    VI_{1}&=-\int_{\mathbb{R}}D_{x}^{s-1}\mathcal{H}\partial_{x}^{3}r\ r\ D_{x}^{s-1}\partial_{x}r\ dx+\int_{\mathbb{R}}D_{x}^{s-1}\mathcal{H}\mathcal{H}\partial_{x}^{2}r \ r\ D_{x}^{s-1}\partial_{x}r\ dx\\
    &\quad +\int_{\mathbb{R}}D_{x}^{s-1}\mathcal{H}(r\partial_{x}r) r\ D_{x}^{s-1}\partial_{x}r\ dx
    -\int_{\mathbb{R}}D_{x}^{s-1}\mathcal{H}\partial_{x}(r\mathcal{H}\partial_{x}r) \ r\ D_{x}^{s-1}\partial_{x}r\ dx\\
    &\quad -\int_{\mathbb{R}}D_{x}^{s-1}\mathcal{H}\partial_{x}\mathcal{H}(r\partial_{x}r) \ r\ D_{x}^{s-1}\partial_{x}r\ dx+\int_{\mathbb{R}}D_{x}^{s-1}\mathcal{H}\partial_{x}(|q|^{2}) \ r\ D_{x}^{s-1}\partial_{x}r\ dx=\sum_{j=1}^{6}VI_{1,j}.
\end{align*}
The term $VI_{1,1}$ will be discussed later. Integrating by parts again, we obtain
\begin{equation*}
VI_{1,2}= -\int_{\mathbb{R}} D_{x}^{s-1}\partial_{x}^{2}r\ r\ D_{x}^{s-1}\partial_{x}r\ dx\\
=\frac{1}{2}\int_{\mathbb{R}} (D_{x}^{s-1}\partial_{x}r)^{2}\partial_{x}r \ dx.
\end{equation*}
Thus, from Hölder's inequality and the Sobolev embedding, along with the fact taht $H^{s-1}(\mathbb{R})$ is a algebra, we obtain 
\begin{align}
&|VI_{1,2}|+|VI_{1,3}|\lesssim \|\partial_{x} r\|_{L^{\infty}}\|D_{x}^{s}r\|^{2}_{L ^{2}_{x}}+ \|r\|_{L^{\infty}_{x}}\|r\|_{H^{s}}^{3}\lesssim \|r\|_{H^{s}_{x}}^{3}+\|r\|_{H^{s}_{x}}^{4}.
\end{align}
For term $VI_{1,4}$, using that $D^{1}_{x}=\partial_{x}\mathcal{H}$ and integrating by parts, we have that
\begin{align*}
    VI_{1,4}
    &=-\int_{\mathbb{R}}[D_{x}^{s};r] \mathcal{H}\partial_{x}r \ r\ D_{x}^{s-1}\partial_{x}r \ dx-\int_{\mathbb{R}}D_{x}^{s}\mathcal{H}\partial_{x}r\ r^{2} \ D_{x}^{s-1}\partial_{x}r \ dx\\
    &=-\int_{\mathbb{R}}[D_{x}^{s};r] \mathcal{H}\partial_{x}r \ r\ D_{x}^{s-1}\partial_{x}r \ dx-\frac{1}{2}\int_{\mathbb{R}}(D_{x}^{s-1}\partial_{x}r)^{2}\ \partial_{x}(r^{2})  \ dx \\
    &=VI_{1,4,1}+V_{1,4,2}.
\end{align*}
Therefore, from Lemma \ref{commestim} and Sobolev's embedding, it follows that
\begin{align}
|V_{1,4,1}|+|V_{1,4,2}|&\lesssim \|r\|_{L^{\infty}_{x}}\|\mathcal{H}\partial_{x}r\|_{L^{\infty}_{x}}\|D_{x}^{s}r\|^{2}+\|r\|_{L^{\infty}_{x}}\|\partial_{x}r\|_{L^{\infty}_{x}}\|D_{x}^{s}r\|^{2}\lesssim \|r\|_{H^{s}_{x}}^{4}.
\end{align}
Concerning the term $VI_{1,5}$, the same arguments applied in $V_{1,4}$ yield that
\begin{equation}
    |V_{1,5}|\lesssim \|r\|_{H^{s}_{x}}^{4}.
\end{equation}
Now, the term $VI_{1,6}$ can be estimated readily using the Sobolev embedding and that $H^{s}(\mathbb{R})$ is a Banach algebra, namely
\begin{equation}
|VI_{1,6}|\lesssim \|q\|_{H^{s}_{x}}^{2}\|r\|_{H^{s}_{x}}^{2}.
\end{equation}
\indent Regarding the term $VI_{2}$, we can again use the first equation in \eqref{BO-NLS}, and note that
\begin{align*}
 VI_{2}&=-\int_{\mathbb{R}}D_{x}^{s-1}\mathcal{H}r (\partial_{x}^{3}r) D_{x}^{s-1}\partial_{x}r\ dx +\int_{\mathbb{R}}D_{x}^{s-1}\mathcal{H}r (\mathcal{H}\partial^{2}_{x}r) D_{x}^{s-1}\partial_{x}r\ dx \\
 &\quad +\int_{\mathbb{R}}D_{x}^{s-1}\mathcal{H}r (r\partial_{x}r) D_{x}^{s-1}\partial_{x}r\ dx-\int_{\mathbb{R}}D_{x}^{s-1}\mathcal{H}r\  \partial_{x}(r\mathcal{H}\partial_{x}r)  D_{x}^{s-1}\partial_{x}r\ dx  \\
 &\quad -\int_{\mathbb{R}}D_{x}^{s-1}\mathcal{H}r\ \partial_{x}\mathcal{H}(r\partial_{x}r)  D_{x}^{s-1}\partial_{x}r\ dx+\int_{\mathbb{R}}D_{x}^{s-1}\mathcal{H}r \ \partial_{x}(|q|^{2})  D_{x}^{s-1}\partial_{x}r\ dx= \sum_{j=1}^{6}VI_{2,j}.
\end{align*}
\indent In the same way as for the term $VI_{1,1}$, we will treat the term $VI_{2,1}$ later. Now, using the Sobolev embedding, we can readily obtain that
\begin{equation}
   |VI_{2,2}|+|VI_{2,3}|\lesssim \|\mathcal{H}\partial_{x}^{2}r\|_{L^{\infty}_{x}}\|r\|_{H^{s}_{x}}^{2}+\|r\|_{L^{\infty}_{x}}\|\partial_{x}r\|_{L^{\infty}_{x}}\|r\|_{H^{s}_{x}}^{2}\lesssim\|\mathcal{H}\partial_{x}^{2}r\|_{L^{\infty}_{x}}\|r\|_{H^{s}_{x}}^{2}+\|r\|_{H^{s}_{x}}^{4}.
\end{equation}
On the other hand, using again the Sobolev embedding, the terms $VI_{2,j},\ j=4,...,6$ can be estimated as follows
\begin{align}
    |V_{2,4}|+|VI_{2,5}|+|VI_{2,6}| \lesssim& \|D_{x}^{s-1}\mathcal{H}r\|_{L^{\infty}_{x}}\left(\|\partial_{x}(r\mathcal{H}\partial_{x}r)\|_{L^{2}_{x}}+\|\partial_{x}\mathcal{H}(r\partial_{x}r)\|_{L^{2}_{x}}+\|\partial_{x}(|q|^{2})\|_{L^{2}_{x}}\right) \|D_{x}^{s}r\|_{L^{2}_{x}}\nonumber \\
    \lesssim &\|r\|_{H^{s}_{x}}\left( \|\partial_{x}r\|_{L^{\infty}_{x}}\|\partial_{x}r\|_{L^{2}_{x}}+\|\mathcal{H}\partial_{x}^{2}r\|_{L^{\infty}_{x}}\|r\|_{L^{2}_{x}}+\|\partial_{x}^{2}r\|_{^{\infty}_{x}}\|r\|_{L^{2}_{x}}+\|q\|_{H^{1}_{x}}^{2} \right)\|r\|_{H^{s}_{x}} \nonumber\\
    \lesssim&(\|\mathcal{H}\partial_{x}^{2}r\|_{L^{\infty}_{x}}+\|\partial_{x}^{2}r\|_{L^{\infty}_{x}})\|r\|_{H^{s}_{x}}^{3}+\|r\|_{H^{s}_{x}}^{4}+\|q\|_{H^{s}_{x}}^{2}\|r\|_{H^{s}_{x}}^{2}.
\end{align}
\indent For the term $VI_{3}$, once again by the first equation in \eqref{BO-NLS}, we have
\begin{align*}
VI_{3}&=-\int_{\mathbb{R}}D_{x}^{s-1}\mathcal{H}r\ r\ D_{x}^{s-1}\partial_{x}^{4}r \ dx+\int_{\mathbb{R}}D_{x}^{s-1}\mathcal{H}r\ r\ D_{x}^{s-1}\partial_{x}\mathcal{H}\partial_{x}^{2}r \ dx\\
&\quad +\int_{\mathbb{R}}D_{x}^{s-1}\mathcal{H}r\ r\ D_{x}^{s-1}\partial_{x} (r\partial_{x}r)\ dx-\int_{\mathbb{R}}D_{x}^{s-1}\mathcal{H}r\ r\ D_{x}^{s-1}\partial_{x}^{2}(r\mathcal{H}\partial_{x}r) \ dx\\
&\quad -\int_{\mathbb{R}}D_{x}^{s-1}\mathcal{H}r\ r\ D_{x}^{s-1}\partial_{x}^{2}\mathcal{H}(r\partial_{x}r) \ dx+\int_{\mathbb{R}}D_{x}^{s-1}\mathcal{H}r\ r\ D_{x}^{s-1}\partial_{x}^{2}(|q|^{2}) \ dx=\sum_{j=1}^{6}VI_{3,j}.
\end{align*}
Note that, from Lemma \ref{sum3derivative}, we have 
\begin{equation*}
VI_{1,1}+VI_{2,1}+VI_{3,1}=-3\int_{\mathbb{R}}D_{x}^{s}r \ \partial_{x}r \ D_{x}^{s-1}\partial_{x}^{2}r \ dx,
\end{equation*}
and then,
\begin{equation}
    II_{3,2,3}+II_{3,3}+II_{2,4,3}+II_{2,5}+\tau_{s}\cdot (VI_{1,1}+VI_{2,1}+VI_{3,1})=0.
\end{equation}
For the term $VI_{3,2}$, by integration by parts, note that
\begin{align*}
VI_{3,2}&=-\int_{\mathbb{R}} D_{x}^{s}r \ r \ D_{x}^{s}\partial_{x}r \ dx-\int_{\mathbb{R}} D_{x}^{s-1}\mathcal{H}r \ \partial_{x}r \ D_{x}^{s}\partial_{x}r \ dx\\
&= \frac{1}{2}\int_{\mathbb{R}} (D_{x}^{s}r)^{2}\partial_{x}r \ dx+\int_{\mathbb{R}}D_{x}^{s}r \ \partial_{x}r \ D_{x}^{s}r \ dx +\int_{\mathbb{R}} D_{x}^{s-1}\mathcal{H}r \ \partial_{x}^{2}r \ D_{x}^{s}r \ dx. 
\end{align*}
Thus, from the Sobolev embedding, it follows that
\begin{equation}
    |VI_{3,2}|\lesssim \left( \|\partial_{x}r\|_{L^{\infty}_{x}}+\|\partial_{x}^{2}r\|_{L^{\infty}_{x}} \right)\|r\|_{H^{s}_{x}}^{2}\lesssim \|r\|_{H^{s}_{x}}^{3}+\|\partial_{x}^{2}r\|_{L^{\infty}_{x}}\|r\|_{H^{s}_{x}}^{2}
\end{equation}
See also that, integrating by parts, we obtain
\begin{equation*}
    VI_{3,3}=-\int_{\mathbb{R}} D_{x}^{s}r \ r \ D_{x}^{s-1}(r\partial_{x}r) \ dx-\int_{\mathbb{R}} D_{x}^{s-1}\mathcal{H}r \ \partial_{x}r\ D_{x}^{s-1}(r\partial_{x}r) \ dx.    
\end{equation*}
Thus, from the Sobolev embedding and the fact that $H^{s-1}(\mathbb{R})$ is a Banach algebra, we get
\begin{equation}
    |VI_{3,3}|\lesssim \|r\|_{H^{s}_{x}}^{4}.
\end{equation}
Concerning the term $VI_{3,4}$ note that, integrating by parts, we have that
\begin{align*}
    VI_{3,4}&=\int_{\mathbb{R}} D_{x}^{s}r \ r \ D_{x}^{s-1}\partial_{x}(r\mathcal{H}\partial_{x}r)\ dx+\int_{\mathbb{R}} D_{x}^{s-1}\mathcal{H}r \ \partial_{x}r \ D_{x}^{s-1}\partial_{x}(r\mathcal{H}\partial_{x}r) \ dx\\
    &=\int_{\mathbb{R}} D_{x}^{s}r \ r \ D_{x}^{s-1}(\partial_{x}r\mathcal{H}\partial_{x}r) \ dx+\int_{\mathbb{R}}D_{x}^{s}r \ r \ D_{x}^{s-1}(r\mathcal{H}\partial_{x}^{2}r) \ dx\\
    &\quad -\int_{\mathbb{R}}D_{x}^{s}r \ \partial_{x}r \ D_{x}^{s-1}(r\mathcal{H}\partial_{x}r) \ dx- \int_{\mathbb{R}} D_{x}^{s-1}\mathcal{H}r \ \partial_{x}^{2}r \ D_{x}^{s-1}(r\mathcal{H}\partial_{x} r) \ dx\\
    &=VI_{3,4,1}+VI_{3,4,2}+VI_{3,4,3}+VI_{3,4,4}.
\end{align*}
But, using the Sobolev embedding and that $H^{s-1}(\mathbb{R})$ is a Banach algebra, we obtain
\begin{equation}
    |VI_{3,4,1}|+|VI_{3,4,3}|+|VI_{3,4,4}|\lesssim \|r\|_{H^{s}_{x}}^{4}+\|\partial_{x}^{2}r\|_{L^{\infty}_{x}}\|r\|_{H^{s}_{x}}^{3}.
\end{equation}
On the other hand, see that
\begin{align*}
    VI_{3,4,2}&=\int_{\mathbb{R}} D_{x}^{s}r \ r \ [D_{x}^{s-1};r]\mathcal{H}\partial_{x}^{2}r\ dx+\int_{\mathbb{R}}D_{x}^{s}r \ r^{2} \ D_{x}^{s}\partial_{x}r \ dx\\
    &= \int_{\mathbb{R}} D_{x}^{s}r \ r \ [D_{x}^{s-1};r]\mathcal{H}\partial_{x}^{2}r\ dx-\frac{1}{2}\int_{\mathbb{R}}(D_{x}^{s}r)^{2} \ \partial_{x}(r^{2}) \ dx,
\end{align*}
and thus, by the Lemma \ref{commestim} and the Sobolev embedding, we have that
\begin{equation}
    |VI_{3,4,2}|\lesssim (\|\mathcal{H}\partial_{x}^{2}r\|_{L^{\infty}_{x}}+\|\partial_{x}r\|_{L^{\infty}_{x}})\|r\|_{H_{x}^{s}}^{3}\lesssim \|\mathcal{H}\partial_{x}^{2}r\|_{L^{\infty}_{x}}\|r\|_{H^{s}_{x}}^{3}+\|r\|_{H^{s}_{x}}^{4}.
\end{equation}
The term $VI_{3,5}$ can be treated similarly as it was done for the term $VI_{3,4}$. Namely, we can obtain
\begin{equation}
    |VI_{3,5}|\lesssim \|\partial_{x}^{2}r\|_{L^{\infty}_{x}}\|r\|_{H^{s}_{x}}^{3}+\|r\|_{H^{s}_{x}}^{4}.
\end{equation}
Finally, an integration by parts and the Sobolev embedding lead us to
\begin{equation} \label{estenerg2}
    |VI_{3,6}|\lesssim \|r\|_{H^{s}_{x}}^{2}\|q\|_{H^{s}_{x}}^{2}.
\end{equation}
Therefore, since 
\begin{equation*}
    \|(r(t),q(t))\|_{H_{x}^{s}\times H_{x}^{s}}\leq \frac{1}{A_{s}},
\end{equation*}
then \eqref{es2} follows from combining estimates \eqref{estenerg1}-\eqref{estenerg2}. The proof of the estimate \eqref{es3} is basically an application of \eqref{es1}, \eqref{es2} and the Gronwall inequality.
\end{proof}

\subsection{Energy estimates for the difference of two solutions of \eqref{BO-NLS}}

Let $(r_{1},q_{1})$ and $(r_{2},q_{2})$ be two solutions of the system \eqref{BO-NLS}. We define
\begin{equation*}
\left\{
\begin{array}{ll}
     & r=r_{1}-r_{2},\quad
      u=r_{1}+r_{2},\\
     &q=q_{1}-q_{2},\quad
     v=q_{1}+q_{2}.
\end{array}
\right.
\end{equation*}
In this case, we have that $(r,q)$ is a solution of the following system
\begin{equation}\label{NLS-BO2}
\begin{cases}
     &  \partial_{t}r-\mathcal{H}\partial_{x}^{2}r+\partial_{x}^{3}r=\frac{1}{2}\partial_{x}(ur)-\partial_{x}(r\mathcal{H}\partial_{x}r_{1}+\mathcal{H}(r\partial_{x}r_{1}))-\partial_{x}(r_{2}\mathcal{H}\partial_{x}r+\mathcal{H}(r_{2}\partial_{x}r))+\partial_{x}\textrm{Re}(\bar{v}q),\\
     & i\partial_{t}q-\partial_{x}^{2}q=-\frac{1}{2}(rv+qu).
\end{cases}
\end{equation}

The purpose of this subsection is to establish energy estimates for the difference between two solutions $(r,q)$. In this regard, similarly to \eqref{em1}, we define a modified energy associated with $(r,q)$ as follows:
\begin{definition}
For all $t\in \mathbb{R}$, we define the modified energy as follows:
\begin{itemize}
\item[(i)] Case $s>\frac{3}{2}$.
\begin{align} \label{em2}
    \widetilde{E}_{m}^{s}(t)&=\|r(t)\|_{L^{2}_{x}}^{2}+\frac{1}{2}\|D_{x}^{s}r(t)\|_{L^{2}_{x}}^{2}+\|q(t)\|_{L^{2}_{x}}^{2}+\|D_{x}^{s}q(t)\|_{L^{2}_{x}}^{2}\nonumber \\
    &\quad +\textrm{Re}\int_{\mathbb{R}}D_{x}^{s-1}(\bar{v}q)D_{x}^{s-1}r \ dx 
    \nonumber+ \tau_{s,1}\int_{\mathbb{R}}D_{x}^{s-1}\mathcal{H}r \ r_{1}\ D_{x}^{s-1}\partial_{x}r \ dx\nonumber
    \\ &\quad +\tau_{s,2}\int_{\mathbb{R}}D_{x}^{s-1}\mathcal{H}r \ r_{2}\ D_{x}^{s-1}\partial_{x}r \ dx ,
\end{align}
where $\tau_{s,1}:=\frac{1}{3}$ and $\tau_{s,2}:=\frac{2s}{3}$.
\item[(ii)] Case $s=0$.
\begin{align}\label{em3}
    \widetilde{E}_{m}^{0}(t)=\frac{1}{2}\|r(t)\|_{L^{2}_{x}}^{2}+\|q(t)\|_{L^{2}_{x}}^{2}-\tau_{0}\int_{\mathbb{R}}\mathcal{H}r \ J_{x}^{-1}\mathcal{H}r \ r_{1} \ dx+\textrm{Re}\int_{\mathbb{R}}J_{x}^{-2}r \ \bar{v} q \ dx,
\end{align}
where $\tau_{0}=\frac{1}{3}$.
\end{itemize}
\end{definition}
In this case, the following proposition provides the energy estimates we require:
\begin{prop}\label{propenergyestimatedifference}
Let $s>\frac{3}{2}$ and $T>0$. There exists a positive constant $\widetilde{A}_{s}$ such that for any $(r_{1},q_{1}),\ (r_{2},q_{2})\in C([0,T]:H^{s}(\mathbb{R})\times H^{s}(\mathbb{R}))$ solutions of \eqref{BO-NLS} satisfying
\begin{equation}\label{boundedrq}
\|(r_{i}(t),q_{i}(t))\|_{H^{s}_{x}\times H^{s}_{x}}\leq \frac{1}{\widetilde{A}_{s}}, \quad i=1,2,
\end{equation}
the following estimates hold true.
\begin{enumerate}
    \item \textbf{Coercivity:} for $\sigma =0$ or $\sigma=s>\frac{3}{2}$
    \begin{equation}\label{esd1}
        \frac{1}{2}\left( \|r(t)\|_{H^{\sigma}_{x}}^{2}+\|q(t)\|_{H^{\sigma}_{x}}^{2}\right)\leq \widetilde{E}_{m}^{\sigma}(t)\leq \frac{3}{2}\left( \|r(t)\|_{H^{\sigma}_{x}}^{2}+\|q(t)\|_{H^{\sigma}_{x}}^{2}\right).
    \end{equation}

    \item $H^{s}\times H^{s}$ \textbf{- energy estimate:}
    \begin{equation}\label{esd2}
        \frac{d}{dt}\widetilde{E}_{m}^{s}(t) \lesssim \sum_{i=1}^{2}\left(1+\|\partial_{x}^{2}r_{i}(t)\|_{L^{\infty}_{x}}+\|\mathcal{H}\partial^{2}_{x}r_{i}\|_{L^{\infty}_{x}}\right)\widetilde{E}_{m}^{s}(t)+f_{s}(t),
    \end{equation}
    where $f_{s}=f_{s}(t)$ is defined by
\begin{align*}
f_{s}(t) &:= 
\|D_{x}^{s}r\|_{L^{2}_{x}} \|r\|_{L^{2}_{x}} \|D_{x}^{s}\partial_{x}u\|_{L^{\infty}_{x}} 
+ \|D_{x}^{s}r\|_{L^{2}_{x}} \|\partial_{x}^{2}r\|_{L^{\infty}_{x}} \|D_{x}^{s}r_{1}\|_{L^{2}_{x}} \\
&\quad + \|D_{x}^{s}r\|_{L^{2}_{x}} \|r\|_{L^{2}_{x}} \|D_{x}^{s+1}\partial_{x}r_{1}\|_{L^{\infty}_{x}}
+\|D_{x}^{s}r\|_{L^{2}_{x}} \|\partial_{x}r\|_{L^{2}_{x}} \|D_{x}^{s+1}r_{1}\|_{L^{\infty}_{x}} \\
&\quad + \|D_{x}^{s}r\|_{L^{2}_{x}} \|\mathcal{H}\partial_{x}^{2}r\|_{L^{\infty}_{x}} \|D_{x}^{s}r_{2}\|_{L^{2}_{x}} 
+ \|D_{x}^{s}r\|_{L^{2}_{x}} \|\partial_{x}r\|_{L^{2}_{x}} \|D_{x}^{s}\partial_{x}r_{2}\|_{L^{\infty}_{x}} \\
&\quad + \|D_{x}^{s}r\|_{L^{2}_{x}} \|\partial_{x}^{2}r\|_{L^{\infty}_{x}} \|D_{x}^{s}r_{2}\|_{L^{2}_{x}} 
+ \|D_{x}^{s}r\|_{L^{2}_{x}} \|\partial_{x}r\|_{L^{2}_{x}} \|D_{x}^{s+1}r_{2}\|_{L^{\infty}_{x}} \\
&\quad + \|D_{x}^{s+1}r_{1}\|_{L^{\infty}_{x}} \|r_{1}\|_{L^{\infty}_{x}} \|r\|_{L^{2}_{x}} \|D_{x}^{s}r\|_{L^{2}_{x}} 
 +\|D_{x}^{s}r\|_{L^{2}_{x}} \|r\|_{L^{2}_{x}} \|r_{1}\|_{L^{\infty}_{x}} \|D_{x}^{s}\partial_{x}r_{1}\|_{L^{\infty}_{x}}\\
&\quad + \|D_{x}^{s}r\|_{L^{2}_{x}} \|r\|_{L^{2}_{x}} \|D_{x}^{s+1}\partial_{x}r_{2}\|_{L^{\infty}_{x}}
 + \|D_{x}^{s}r\|_{L^{2}_{x}} \|\mathcal{H}\partial_{x}^{2}r\|_{L^{\infty}_{x}} \|D_{x}^{s}r_{1}\|_{L^{2}_{x}} \\
&\quad + \|D_{x}^{s}r\|_{L^{2}_{x}} \|\partial_{x}r\|_{L^{2}_{x}} \|D_{x}^{s}\partial_{x}r_{1}\|_{L^{\infty}_{x}} +\|D_{x}^{s+1}r_{2}\|_{L^{\infty}_{x}} \|r_{2}\|_{L^{\infty}_{x}} \|r\|_{L^{2}_{x}} \|D_{x}^{s}r\|_{L^{2}_{x}} 
 \\
 &\quad +\|D_{x}^{s}r\|_{L^{2}_{x}} \|r\|_{L^{2}_{x}} \|r_{2}\|_{L^{\infty}_{x}} \|D_{x}^{s}\partial_{x}r_{2}\|_{L^{\infty}_{x}}.
\end{align*}
\item $L^{2}\times L^{2}$ \textbf{- energy estimate:}
\begin{align}\label{esd3}
\frac{d}{dt}\widetilde{E}_{m}^{0}(t)\lesssim \sum_{i=1}^{2}\left(1+\|\partial_{x}^{2}r_{i}(t)\|_{L^{\infty}_{x}}+\|\mathcal{H}\partial^{2}_{x}r_{i}\|_{L^{\infty}_{x}}\right) \widetilde{E}_{m}^{0}(t).
\end{align}
\end{enumerate}
\end{prop}

\begin{remark}\ \label{remarkenergyestimate2}
\vspace{-1em}
\begin{itemize}
    \item[1.] As in Proposition \ref{propenergyestimate}, the estimates above can also be established for the system \eqref{BO-NLS-rescaled}, with constants that are uniform with respect to $0<\gamma \leq 1$.
    
    \item[2.] Note that in the expression for the function $f_s(t)$ above, several derivatives of order higher than $s$ appear. However, this does not introduce any extra difficulty, since all such derivatives make part of the terms $r_i, q_{i}$ not of their differences $r$ and $q$. In what follows, these terms will be handled with the Bona–Smith argument.
\end{itemize}
\end{remark}

\begin{proof}
The estimate \eqref{esd1} follows readily from the Cauchy-Schwarz inequality, the Sobolev embedding, and the fact that $H^{s}(\mathbb{R})$ is an algebra.

Considering now $\widetilde{E}_{m}^{s}(t)$ as in \eqref{em2} and differentiating with respect to $t$, we obtain
\begin{align}\label{ddtE~s}
    \frac{d}{dt}\widetilde{E}_{m}^{s}(t)&=\frac{d}{dt}\left(\int_{\mathbb{R}}r^{2}\ dx\right)+\frac{1}{2}\frac{d}{dt}\left(\int_{\mathbb{R}}(D_{x}^{s}r)^{2}\ dx\right)+\frac{d}{dt}\left(\int_{\mathbb{R}}|q|^{2}\ dx\right)+\frac{d}{dt}\left(\int_{\mathbb{R}}|D_{x}^{s}q|^{2}\ dx\right)\nonumber\\
    &\quad +\textrm{Re}\left(\frac{d}{dt}\int_{\mathbb{R}}D_{x}^{s-1}(\bar{v}q)D_{x}^{s-1}r \ dx\right)
    + \tau_{s,1}\frac{d}{dt}\int_{\mathbb{R}}D_{x}^{s-1}\mathcal{H}r \ r_{1}\ D_{x}^{s-1}\partial_{x}r \ dx\nonumber \\
    &\quad +\tau_{s,2}\frac{d}{dt}\int_{\mathbb{R}}D_{x}^{s-1}\mathcal{H}r \ r_{2}\ D_{x}^{s-1}\partial_{x}r \ dx \nonumber \\
    &= \widetilde{I}+\widetilde{II}+\widetilde{III}+\widetilde{IV}+\widetilde{V}+\widetilde{VI}+\widetilde{VII}.
\end{align}
From the first equation in \eqref{NLS-BO2} and integration by parts, we obtain
\begin{align*}
    \widetilde{I}&=\int_{\mathbb{R}}r\partial_{x}(ur)\ dx-2\int_{\mathbb{R}}r\partial_{x}(r\mathcal{H}\partial_{x}r_{1})\ dx-2\int_{\mathbb{R}}r\partial_{x}(r_{2}\mathcal{H}\partial_{x}r)\ dx -2\int_{\mathbb{R}}r\partial_{x}\mathcal{H}(r\partial_{x}r_{1})\ dx\nonumber \\
    &\quad -2\int_{\mathbb{R}}r\partial_{x}\mathcal{H}(r_{2}\partial_{x}r)\ dx+2\textrm{Re}\int_{\mathbb{R}}r\partial_{x}(\bar{v}q)\ dx \nonumber \\
    &=\int_{\mathbb{R}}r\partial_{x}(ur)\ dx+2\int_{\mathbb{R}}\partial_{x}r\ r\ \mathcal{H}\partial_{x}r_{1}\ dx+2\int_{\mathbb{R}}\partial_{x}r\ r_{2}\ \mathcal{H}\partial_{x}r\ dx +2\int_{\mathbb{R}}\partial_{x}r\mathcal{H}(r\partial_{x}r_{1})\ dx\nonumber \\
    &\quad +2\int_{\mathbb{R}}\partial_{x}r\mathcal{H}(r_{2}\partial_{x}r)\ dx+2\textrm{Re}\int_{\mathbb{R}}r\partial_{x}(\bar{v}q)\ dx.
\end{align*}
Thus, by using Hölder's inequality and the Sobolev embedding, the term $\widetilde{I}$ can be estimated as follows
\begin{equation}
    |\widetilde{I}|\lesssim (\|r_{1}\|_{H^{s}_{x}}+\|r_{2}\|_{H^{s}_{x}}+\|u\|_{H^{s}_{x}}+\|v\|_{H^{s}_{x}})(\|r\|_{H^{s}_{x}}^{2}+\|q\|_{H^{s}_{x}}^{2}).
\end{equation}
\indent Now, considering the term $\widetilde{II}$, we can again use the first equation in \eqref{NLS-BO2} and obtain
\begin{align*}
    \widetilde{II}&=\frac{1}{2}\int_{\mathbb{R}} D_{x}^{s}rD_{x}^{s}\partial_{x}(ur)  \ dx-\int_{\mathbb{R}} D_{x}^{s}rD_{x}^{s}\partial_{x}(r\mathcal{H}\partial_{x}r_{1}) \ dx-\int_{\mathbb{R}} D_{x}^{s}rD_{x}^{s}\partial_{x}\mathcal{H}(r\partial_{x}r_{1}) \ dx\\
    &\quad -\int_{\mathbb{R}} D_{x}^{s}rD_{x}^{s}\partial_{x}(r_{2}\mathcal{H}\partial_{x}r) \ dx-\int_{\mathbb{R}} D_{x}^{s}rD_{x}^{s}\partial_{x}\mathcal{H}(r_{2}\partial_{x}r) \ dx+\textrm{Re}\int_{\mathbb{R}} D_{x}^{s}rD_{x}^{s}\partial_{x}(\bar{v}q) \ dx\\
    &=\widetilde{II}_{1}+\widetilde{II}_{2}+\widetilde{II}_{3}+\widetilde{II}_{4}+\widetilde{II}_{5}+\widetilde{II}_{6}.
\end{align*}
Concerning the term $\widetilde{II}_{1}$, note that
\begin{align*}
    \widetilde{II}_{1}&=\frac{1}{2}\int_{\mathbb{R}} D_{x}^{s}rD_{x}^{s}(\partial_{x}u\  r)\ dx+\frac{1}{2}\int_{\mathbb{R}} D_{x}^{s}rD_{x}^{s}(u \partial_{x}r)\ dx\\
    &=\frac{1}{2}\int_{\mathbb{R}}D_{x}^{s}r[D_{x}^{s};r]\partial_{x}u \ dx+\frac{1}{2}\int_{\mathbb{R}} D_{x}^{s}r \ r\ D_{x}^{s}\partial_{x}u \ dx+
    \frac{1}{2}\int_{\mathbb{R}}D_{x}^{s}r[D_{x}^{s};u]\partial_{x}r \ dx\\
    &\quad +\frac{1}{2}\int_{\mathbb{R}}D_{x}^{s}r \ u\ D_{x}^{s}\partial_{x}r \ dx\\
&=\widetilde{II}_{1,1}+\widetilde{II}_{1,2}+\widetilde{II}_{1,3}+\widetilde{II}_{1,4}.
\end{align*}
Thus, by Lemma \ref{commestim}, integration by parts, and the Sobolev embedding, we deduce that
\begin{align}
    |\widetilde{II}_{1,1}|+|\widetilde{II}_{1,3}|+|\widetilde{II}_{1,4}|&\lesssim \|D_{x}^{s}r\|_{L^{2}_{x}}^{2}\|\partial_{x}u\|_{L^{\infty}}+\|D_{x}^{s}r\|_{L^{2}_{x}}\|\partial_{x}r\|_{L^{\infty}_{x}}\|D_{x}^{s}u\|_{L^{2}_{x}}  \lesssim \|r\|_{H^{s}_{x}}^{2}\|u\|_{H^{s}_{x}};\\
    |\widetilde{II}_{1,2}|&\lesssim \|D_{x}^{s}r\|_{L^{2}_{x}}\|r\|_{L^{2}_{x}}\|D_{x}^{s}\partial_{x}u\|_{L^{\infty}_{x}}.
\end{align}
The term $II_{2}$ can be treated as follows
\begin{align*}
\widetilde{II}_{2}&=-\int_{\mathbb{R}} D_{x}^{s}r\left\{D_{x}^{s}\partial_{x}(r\mathcal{H}\partial_{x}r_{1})-rD_{x}^{s+1}\partial_{x}r_{1}-s\partial_{x}rD_{x}^{s+1}r_{1}-\mathcal{H}\partial_{x}r_{1}D_{x}^{s}\partial_{x}r-\partial_{x}rD_{x}^{s+1}r_{1}  \right\}\ dx \\
&\quad -\int_{\mathbb{R}}D_{x}^{s}r \ r\ D_{x}^{s+1}\partial_{x}r_{1} \ dx-s\int_{\mathbb{R}}D_{x}^{s}r \ \partial_{x}r\ D_{x}^{s+1}r_{1} \ dx-\int_{\mathbb{R}}D_{x}^{s}r \mathcal{H}\partial_{x}r_{1} D_{x}^{s}\partial_{x}r \ dx\\
&\quad -\int_{\mathbb{R}}D_{x}^{s}r \partial_{x}rD_{x}^{s+1}r_{1} \ dx\\
&=\widetilde{II}_{2,1}+\widetilde{II}_{2,2}+\widetilde{II}_{2,3}+\widetilde{II}_{2,4}+\widetilde{II}_{2,5}.
\end{align*}
The Lemma \ref{geralcommestm}, the Hölder inequality, and integration by parts yield the following
\begin{align}
    |\widetilde{II}_{2,1}|&\lesssim \|D_{x}^{s}r\|_{L^{2}_{x}}\left(\|\mathcal{H}\partial_{x}^{2}r_{1}\|_{L^{\infty}_{x}}\|D_{x}^{s}r\|_{L^{2}_{x}}+\|\partial_{x}^{2}r\|_{L^{\infty}_{x}}\|D_{x}^{s}r_{1}\|_{L^{2}_{x}}\right);\\
    |\widetilde{II}_{2,2}|&\lesssim \|D_{x}^{s}r\|_{L^{2}_{x}}\|r\|_{L^{2}_{x}}\|D_{x}^{s+1}\partial_{x}r_{1}\|_{L^{\infty}_{x}};\\
    |\widetilde{II}_{2,3}|+|\widetilde{II}_{2,5}|&\lesssim \|D_{x}^{s}r\|_{L^{2}_{x}}\|\partial_{x}r\|_{L^{2}_{x}}\|D^{s+1}_{x}r_{1}\|_{L^{\infty}_{x}};\\
    |\widetilde{II}_{2,4}|&\lesssim \|\mathcal{H}\partial_{x}^{2}r_{1}\|_{L^{\infty}_{x}}\|D_{x}^{s}r\|_{L^{2}_{x}}^{2}.
\end{align}
Proceeding similarly with the term $\widetilde{II}_{4}$ we have
\begin{align*}
\widetilde{II}_{4}&=-\int_{\mathbb{R}} D_{x}^{s}r\left\{D_{x}^{s}\partial_{x}(r_{2}\mathcal{H}\partial_{x}r)-r_{2}D_{x}^{s+1}\partial_{x}r-s\partial_{x}r_{2}D_{x}^{s+1}r-\mathcal{H}\partial_{x}rD_{x}^{s}\partial_{x}r_{2}-\partial_{x}r_{2}D_{x}^{s+1}r  \right\}\ dx \\
&\quad -\int_{\mathbb{R}}D_{x}^{s}r \ r_{2}\ D_{x}^{s+1}\partial_{x}r \ dx-\int_{\mathbb{R}}D_{x}^{s}r \mathcal{H}\partial_{x}rD_{x}^{s}\partial_{x}r_{2} \ dx-(s+1)\int_{\mathbb{R}}D_{x}^{s}r \ \partial_{x}r_{2}\ D_{x}^{s+1}r \ dx\\
&=\widetilde{II}_{4,1}+\widetilde{II}_{4,2}+\widetilde{II}_{4,3}+\widetilde{II}_{4,4}.
\end{align*}
Therefore, by using the Hölder inequality and Lemma \ref{geralcommestm}, we obtain
\begin{align}
    &|\widetilde{II}_{4,1}|\lesssim \|D_{x}^{s}r\|_{L^{2}_{x}}\left( 
\|\partial_{x}^{2}r_{2}\|_{L^{\infty}_{x}}\|D_{x}^{s}r\|_{L^{2}_{x}}+\|\mathcal{H}\partial_{x}^{2}r\|_{L^{\infty}_{x}}\|D_{x}^{s}r_{2}\|_{L^{2}_{x}} \right);\\
&|\widetilde{II}_{4,3}|\lesssim \|D_{x}^{s}r\|_{L^{2}_{x}}\|\partial_{x}r\|_{L^{2}_{x}}\|D_{x}^{s}\partial_{x}r_{2}\|_{L^{\infty}_{x}}.
\end{align}
On the other hand, by following the argument used in the derivation of \eqref{eqforII32}, we have
\begin{align*}
\widetilde{II}_{4,2}&=\frac{1}{2}\int_{\mathbb{R}}[\mathcal{H};r_{2}]D_{x}^{s}\partial_{x}^{2}r\ D_{x}^{s}r \ dx+\frac{1}{2}\int_{\mathbb{R}}D_{x}^{s}r \ \partial_{x}^{2}r_{2} \ D_{x}^{s}\mathcal{H}r \ dx+\int_{\mathbb{R}}D_{x}^{s}r \ \partial_{x}r_{2} \ D_{x}^{s+1}r \ dx\\
&=\widetilde{II}_{4,2,1}+\widetilde{II}_{4,2,2}+\widetilde{II}_{4,2,3}.
\end{align*}
Thus, the Hölder inequality and Lemma \ref{commutatorhilbertestimate} yield
\begin{align}
    |\widetilde{II}_{4,2,1}|+|\widetilde{II}_{4,2,2}|&\lesssim \|\partial_{x}^{2}r_{2}\|_{L^{\infty}_{x}}\|D_{x}^{s}r\|_{L^{2}_{x}}^{2};
\end{align}
\indent The terms $\widetilde{II}_{4,2,3}$ and $\widetilde{II}_{4,4}$ will be addressed later, since they require a complete description of the term $\widetilde{VII}$, i.e., there are terms obtained in the description of $\widetilde{VII}$ which will cancel with the terms  $\widetilde{II}_{4,2,3}$ and $\widetilde{II}_{4,4}$.

Regarding the term $\widetilde{II}_{3}$, we can write
\begin{align*}
    \widetilde{II}_{3}&=-\int_{\mathbb{R}} D_{x}^{s}r\left[D_{x}^{s}\partial_{x}\mathcal{H}(r\partial_{x}r_{1}) +rD_{x}^{s-1}\partial_{x}^{3}r_{1}+\partial_{x}r_{1}D_{x}^{s-1}\partial_{x}^{2}r+(s+1)\partial_{x}rD_{x}^{s-1}\partial_{x}^{2}r_{1}\right]\ dx\\
    &\quad +\int_{\mathbb{R}}D_{x}^{s}r\ r\ D_{x}^{s-1}\partial_{x}^{3}r_{1}  \ dx+\int_{\mathbb{R}}D_{x}^{s}r\ \partial_{x}r_{1}\ D_{x}^{s-1}\partial_{x}^{2}r  \ dx+(s+1)\int_{\mathbb{R}}D_{x}^{s}r\ \partial_{x}r\ D_{x}^{s-1}\partial_{x}^{2}r_{1}  \ dx\\
    &=\widetilde{II}_{3,1}+\widetilde{II}_{3,2}+\widetilde{II}_{3,3}+\widetilde{II}_{3,4}.
\end{align*}
Thus, applying the Hölder inequality and Lemma \ref{geralcommestm}, it follows that
\begin{align}
&|\widetilde{II}_{3,1}|\lesssim \|D_{x}^{s}r\|_{L^{2}_{x}}\left( 
\|\partial_{x}^{2}r_{1}\|_{L^{\infty}_{x}}\|D_{x}^{s}r\|_{L^{2}_{x}}+\|\partial_{x}^{2}r\|_{L^{\infty}_{x}}\|D_{x}^{s}r_{1}\|_{L^{2}} \right);\\
&|\widetilde{II}_{3,2}|\lesssim \|D_{x}^{s}r\|_{L^{2}_{x}}\|r\|_{L^{2}_{x}}\|D_{x}^{s-1}\partial_{x}^{3}r_{1}\|_{L^{\infty}_{x}};\\
&|\widetilde{II}_{3,4}|\lesssim \|D_{x}^{s}r\|_{L^{2}_{x}}\|\partial_{x}r\|_{L^{2}_{x}}\|D_{x}^{s-1}\partial_{x}^{2}r_{1}\|_{L^{\infty}_{x}}.
\end{align}
As before, the term $\widetilde{II}_{3,3}$ will also be treated later, because we need to describe $\widetilde{VI}$ to do this.

The term $\widetilde{II}_{5}$ can be handled in the same way as $II_{3}$, namely, note that
\begin{align*}
    \widetilde{II}_{5}&=-\int_{\mathbb{R}} D_{x}^{s}r\left[D_{x}^{s}\partial_{x}\mathcal{H}(r_{2}\partial_{x}r)+r_{2}D_{x}^{s-1}\partial_{x}^{3}r+\partial_{x}rD_{x}^{s-1}\partial_{x}^{2}r_{2}+(s+1)\partial_{x}r_{2}D_{x}^{s-1}\partial_{x}^{2}r \right]\ dx\\
    &\quad +\int_{\mathbb{R}}D_{x}^{s}r\ r_{2}\ D_{x}^{s-1}\partial_{x}^{3}r \ dx+\int_{\mathbb{R}}D_{x}^{s}r \partial_{x}rD_{x}^{s-1}\partial_{x}^{2}r_{2} \ dx+(s+1)\int_{\mathbb{R}}D_{x}^{s}r\partial_{x}r_{2}D_{x}^{s-1}\partial_{x}^{2}r \ dx  \\
    &=\widetilde{II}_{5,1}+\widetilde{II}_{5,2}+\widetilde{II}_{5,3}+\widetilde{II}_{5,4}.
\end{align*}
Now, by using the Hölder inequality and Lemma \ref{geralcommestm}, we see that
\begin{align}
    &|\widetilde{II}_{5,1}|\lesssim \|D_{x}^{s}r\|_{L^{2}_{x}}\left( \|\partial_{x}^{2}r\|_{L^{\infty}_{x}}\|D_{x}^{s}r_{2}\|_{L^{2}_{x}}+\|\partial_{x}^{2}r_{2}\|_{L^{\infty}_{x}}\|D_{x}^{s}r\|_{L^{2}_{x}}\right);\\
 &|\widetilde{II}_{5,3}|\lesssim \|D_{x}^{s}r\|_{L^{2}_{x}}\|\partial_{x}r\|_{L^{2}_{x}}\|D_{x}^{s-1}\partial_{x}^{2}r_{2}\|_{L^{\infty}_{x}}.
\end{align}
On the other hand, as before, the same arguments applied in \eqref{eqforII32} yield
\begin{align*}
\widetilde{II}_{5,2}&=\frac{1}{2}\int_{\mathbb{R}}[\mathcal{H};r_{2}]D_{x}^{s}\partial_{x}^{2}r\ D^{s}_{x}r\ dx+\frac{1}{2}\int_{\mathbb{R}}D_{x}^{s}r \ \partial_{x}^{2}r_{2} \ D_{x}^{s}\mathcal{H}r \ dx-\int_{\mathbb{R}}D_{x}^{s}r \ \partial_{x}r_{2} \ D_{x}^{s-1}\partial_{x}^{2}r \ dx\\
&=\widetilde{II}_{5,2,1}+\widetilde{II}_{5,2,2}+\widetilde{II}_{5,2,3}.
\end{align*}
Thus, Lemma \ref{commutatorhilbertestimate} implies that
\begin{align}
    |\widetilde{II}_{5,2,1}|+|\widetilde{II}_{5,2,2}|\lesssim \|\partial_{x}^{2}r_{2}\|_{L^{\infty}_{x}}\|D_{x}^{s}r\|_{L^{2}_{x}}^{2}.
\end{align}
\indent We will come back to $\widetilde{II}_{5,2,3},\ \widetilde{II}_{5,4}$ and  $\widetilde{II}_{6}$ later, since we first need to describe $\widetilde{V}$ and $\widetilde{VII}$ completely.

Now, using the second equation in \eqref{NLS-BO2} and integration by parts, observe that
\begin{equation*}
    \widetilde{III}=2\textrm{Re}\int_{\mathbb{R}} (-i\partial_{x}^{2}q+\frac{i}{2}(rv+qu))\bar{q} \ dx
    =-\textrm{Im}\int_{\mathbb{R}}rv\bar{q} \ dx-\textrm{Im}\int_{\mathbb{R}} u|q|^{2} \ dx.
\end{equation*}
Thus, by Hölder's inequality
\begin{equation}\label{estimativadeql2}
    |\widetilde{III}|\lesssim (\|v\|_{L^{\infty}_{x}}+\|u\|_{L^{\infty}_{x}})(\|r\|_{L^{2}_{x}}^{2}+\|q\|_{L^{2}_{x}}^{2}).
\end{equation}
\indent Similarly, once again using the second equation in \eqref{NLS-BO2}, we deduce that
\begin{align*}
\widetilde{IV}&=2\textrm{Re}\left(\int_{\mathbb{R}} D_{x}^{s}\left\{-i\partial_{x}^{2}q+\frac{i}{2}(rv+qu)\right\}D_{x}^{s}\bar{q} \ dx  \right) =-\textrm{Im}\int_{\mathbb{R}} D_{x}^{s}(rv)D_{x}^{s}\bar{q} 
 \ dx-\textrm{Im}\int_{\mathbb{R}} D_{x}^{s}(qu)D_{x}^{s}\bar{q} 
 \ dx.
\end{align*}
Now, since $H^{s}(\mathbb{R})$ is a Banach algebra, it follows that
\begin{equation}
    |\widetilde{IV}|\lesssim (\|v\|_{H^{s}_{x}}+\|u\|_{H^{s}_{x}})(\|r\|_{H^{s}_{x}}^{2}+\|q\|_{H^{s}_{x}}^{2})
\end{equation}
\indent Considering the term $\widetilde{V}$, we can write
\begin{align*}
\widetilde{V}&=\textrm{Re}\int_{\mathbb{R}}D_{x}^{s-1}(\bar{v}q)D_{x}^{s-1}\partial_{t}r \ dx+\textrm{Re}\int_{\mathbb{R}}D_{x}^{s-1}(\bar{v}\partial_{t}q)D_{x}^{s-1}r \ dx+\textrm{Re}\int_{\mathbb{R}}D_{x}^{s-1}(\partial_{t}\bar{v}q)D_{x}^{s-1}r \ dx\\
&=\widetilde{V}_{1}+\widetilde{V}_{2}+\widetilde{V}_{3}.
\end{align*}
Then, using the first equation in \eqref{BO-NLS}, note that
\begin{align*}
\widetilde{V}_{1}&=\textrm{Re}\int_{\mathbb{R}}D_{x}^{s-1}(\bar{v}q)D_{x}^{s-1}\mathcal{H}\partial_{x}^{2}r \ dx-\textrm{Re}\int_{\mathbb{R}}D_{x}^{s-1}(\bar{v}q)D_{x}^{s-1}\partial_{x}^{3}r \ dx\\
&\quad +\frac{1}{2}\textrm{Re}\int_{\mathbb{R}}D_{x}^{s-1}(\bar{v}q)D_{x}^{s-1}\partial_{x}(ur) \ dx
 -\textrm{Re}\int_{\mathbb{R}}D_{x}^{s-1}(\bar{v}q)D_{x}^{s-1}\partial_{x}(r\mathcal{H}\partial_{x}r_{1}+\mathcal{H}(r\partial_{x}r_{1})) \ dx\\
 &\quad -\textrm{Re}\int_{\mathbb{R}}D_{x}^{s-1}(\bar{v}q)D_{x}^{s-1}\partial_{x}(r_{2}\mathcal{H}\partial_{x}r+\mathcal{H}(r_{2}\partial_{x}r)) \ dx +\textrm{Re}\int_{\mathbb{R}}D_{x}^{s-1}(\bar{v}q)D_{x}^{s-1}\partial_{x}\textrm{Re}(\bar{v}q) \ dx\\
&=\widetilde{V}_{1,1}+\widetilde{V}_{1,2}+\widetilde{V}_{1,3}+\widetilde{V}_{1,4}+\widetilde{V}_{1,5}+\widetilde{V}_{1,6}.
\end{align*}
Thus, integration by parts and the fact that $H^{s}(\mathbb{R})$ and $H^{s-1}(\mathbb{R})$ are Banach algebras imply that
\begin{align} &|\widetilde{V}_{1,1}|+|\widetilde{V}_{1,3}|+|\widetilde{V}_{1,4}|+|\widetilde{V}_{1,5}|+|\widetilde{V}_{1,6}|\nonumber\\
\lesssim &\|v\|_{H^{s}_{x}}(1+\|u\|_{H^{s}_{x}}+\|r_{1}\|_{H^{s}_{x}}+\|r_{2}\|_{H^{s}_{x}}+\|v\|_{H^{s}_{x}})(\|r\|_{H^{s}_{x}}^{2}+\|q\|_{H^{x}_{x}}^{2}).
\end{align}
On the other hand, we obtain that
\begin{equation*}
    \widetilde{V}_{1,2}=-\textrm{Re}\int_{\mathbb{R}}D_{x}^{s-1}\partial_{x}^{2}(\bar{v}q)D_{x}^{s-1}\partial_{x}r \ dx=-\textrm{Re}\int_{\mathbb{R}}D_{x}^{s}\partial_{x}(\bar{v}q)D_{x}^{s}r \ dx.
\end{equation*}
Therefore,
\begin{equation}
    \widetilde{V}_{1,2}+\widetilde{II}_{6}=0.
\end{equation}
\indent Now, from the second equation in \eqref{NLS-BO2}, we deduce
\begin{align*}
    \widetilde{V}_{2}&=-\textrm{Re}\left(i\int_{\mathbb{R}}D_{x}^{s-1}(\bar{v}\ \partial_{x}^{2}q)D_{x}^{s-1}r \ dx\right)+\frac{1}{2}\textrm{Re}\left(i\int_{\mathbb{R}}D_{x}^{s-1}(r|v|^{2})D_{x}^{s-1}r \ dx\right)\\
    &\quad +\frac{1}{2}\textrm{Re}\left(i\int_{\mathbb{R}}D_{x}^{s-1}(\bar{v}qu)D_{x}^{s-1}r \ dx\right)\\
&=\widetilde{V}_{2,1}+\widetilde{V}_{2,2}+\widetilde{V}_{2,3}.
\end{align*}
But from integration by parts on the first term $\widetilde{VI}_{2,1}$, and using again that $H^{s-1}(\mathbb{R})$ is a Banach algebra, it follows that
\begin{equation}
|\widetilde{V}_{2,1}|+|\widetilde{V}_{2,2}|+|\widetilde{V}_{2,3}|\lesssim \|v\|_{H^{s}_{s}}(\|r\|_{H^{s}_{x}}^{2}+\|q\|_{H^{s}_{x}}^{2})+ \|v\|_{H^{s}_{x}}^{2}\|r\|_{H^{s}_{x}}^{2}+\|v\|_{H^{s}_{x}}\|u\|_{H^{s}_{x}}(\|r\|_{H^{s}_{x}}^{2}+\|q\|_{H^{s}_{x}}^{2}).   
\end{equation}
To conclude the analysis involving the term $\widetilde{V}$, we observe that $v=v(x,t)$ satisfies the following equation
\begin{equation}\label{equationofv}
    i\partial_{t}v-\partial_{x}^{2}v=-rq_{1}+r_{2}q.
\end{equation}
Hence, we can estimate the term $\widetilde{V}_{3}$ exactly as we did with $\widetilde{V_{2}}$. In this case, we obtain the following
\begin{equation}
|\widetilde{V}_{3}|\lesssim \|v\|_{H^{s}_{x}}(\|r\|_{H^{s}_{x}}^{2}+\|q\|_{H^{s}_{x}}^{2})+(\|r\|_{H^{s}_{x}}\|q_{1}\|_{H^{s}_{x}}+\|r_{2}\|_{H^{s}_{x}}\|q\|_{H^{s}_{x}})(\|r\|_{H^{s}_{x}}^{2}+\|q\|_{H^{s}_{x}}^{2}).
\end{equation}
\indent Let us now treat with the term $\widetilde{VI}$. Note then that
\begin{align*}
\frac{1}{\tau_{s,1}}\widetilde{VI}&=\int_{\mathbb{R}}D_{x}^{s-1}\mathcal{H}\partial_{t}r \ r_{1}\ D_{x}^{s-1}\partial_{x}r \ dx+\int_{\mathbb{R}}D_{x}^{s-1}\mathcal{H}r \ \partial_{t}r_{1}\ D_{x}^{s-1}\partial_{x}r \ dx +\int_{\mathbb{R}}D_{x}^{s-1}\mathcal{H}r \ r_{1}\ D_{x}^{s-1}\partial_{x}\partial_{t}r \ dx\\
&=\widetilde{VI}_{1}+\widetilde{VI}_{2}+\widetilde{VI}_{3}.
\end{align*}
Thus, using the first equation \eqref{NLS-BO2}, we have that
\begin{align*}
\widetilde{VI}_{1}&=\int_{\mathbb{R}}D_{x}^{s-1}\mathcal{H}\mathcal{H}\partial_{x}^{2}r  \ r_{1}\ D_{x}^{s-1}\partial_{x}r \ dx-\int_{\mathbb{R}}D_{x}^{s-1}\mathcal{H}\partial_{x}^{3}r  \ r_{1}\ D_{x}^{s-1}\partial_{x}r \ dx\\
&\quad +\frac{1}{2}\int_{\mathbb{R}}D_{x}^{s-1}\mathcal{H}\partial_{x}(ur)  \ r_{1}\ D_{x}^{s-1}\partial_{x}r \ dx
-\int_{\mathbb{R}}D_{x}^{s-1}\mathcal{H}\partial_{x}(r\mathcal{H}\partial_{x}r_{1})  \ r_{1}\ D_{x}^{s-1}\partial_{x}r \ dx\\
&\quad -\int_{\mathbb{R}}D_{x}^{s-1}\mathcal{H}\partial_{x}\mathcal{H}(r\partial_{x}r_{1})  \ r_{1}\ D_{x}^{s-1}\partial_{x}r \ dx
-\int_{\mathbb{R}}D_{x}^{s-1}\mathcal{H}\partial_{x}(r_{2}\mathcal{H}\partial_{x}r)  \ r_{1}\ D_{x}^{s-1}\partial_{x}r \ dx\\
&\quad -\int_{\mathbb{R}}D_{x}^{s-1}\mathcal{H}\partial_{x}\mathcal{H}(r_{2}\partial_{x}r)  \ r_{1}\ D_{x}^{s-1}\partial_{x}r \ dx
+\textrm{Re}\int_{\mathbb{R}}D_{x}^{s-1}\mathcal{H}\partial_{x}(\bar{v}q)  \ r_{1}\ D_{x}^{s-1}\partial_{x}r \ dx=\sum_{j=1}^{8}\widetilde{VI}_{1,j}.
\end{align*}
Now, by using integration by parts, we deduce that
\begin{equation} \label{VI~11}
    |\widetilde{VI}_{1,1}|\lesssim \|\partial_{x}r_{1}\|_{L^{\infty}_{x}}\|D_{x}^{s}r\|_{L^{2}_{x}}^{2}.
\end{equation}
\indent The term $\widetilde{VI}_{1,2}$ will be addressed later, since its analysis relies on the complete description of $\widetilde{VI}_{2}$ and $\widetilde{VI}_{3}$.

See now that the terms $\widetilde{VI}_{1,j},\ j=3,...,8,$ can be estimated using integration by parts, Lemma \ref{commestim}, Sobolev's embedding, along with the fact that $H^{s}(\mathbb{R})$ is an algebra. Namely, we can obtain that
\begin{align}
|\widetilde{VI}_{1,3}|+|\widetilde{VI}_{1,8}|&\lesssim (\|u\|_{H^{s}_{x}}+\|v\|_{H^{s}_{x}})\|r_{1}\|_{H^{s}_{x}}(\|r\|_{H^{s}_{x}}^{2}+\|q\|_{H^{s}_{x}}^{2}),\\
|\widetilde{VI}_{1,4}|+|\widetilde{VI}_{1,5}|&\lesssim \|r_{1}\|_{H^{s}_{x}}^{2}\|r\|_{H^{s}_{x}}^{2}+\|D_{x}^{s+1}r_{1}\|_{L^{\infty}_{x}}\|r_{1}\|_{L^{\infty}_{x}}\|r\|_{L^{2}_{x}}\|D_{x}^{s}r\|_{L^{2}_{x}},\\
|\widetilde{VI}_{1,6}|+|\widetilde{VI}_{1,7}|&\lesssim \|r_{1}\|_{H^{s}_{x}}\|r_{2}\|_{H^{s}_{x}}\|r\|_{H^{s}_{x}}^{2}.
\end{align}
\indent Next, regarding the term $\widetilde{VI}_{2}$, using the first equation in \eqref{BO-NLS}, we can write
\begin{align*}
\widetilde{VI}_{2}&= \int_{\mathbb{R}} D_{x}^{s-1}\mathcal{H}r \ \mathcal{H}\partial_{x}^{2}r_{1} \ D_{x}^{s-1}\partial_{x}r \ dx-\int_{\mathbb{R}} D_{x}^{s-1}\mathcal{H}r \ \partial_{x}^{3}r_{1}\ D_{x}^{s-1}\partial_{x}r \ dx\\
&\quad +\int_{\mathbb{R}} D_{x}^{s-1}\mathcal{H}r\  (r_{1}\partial_{x}r_{1}) \ D_{x}^{s-1}\partial_{x}r \ dx-\int_{\mathbb{R}} D_{x}^{s-1}\mathcal{H}r \  \partial_{x}(r_{1}\mathcal{H}\partial_{x}r_{1}) \ D_{x}^{s-1}\partial_{x}r \ dx\\
&\quad -\int_{\mathbb{R}} D_{x}^{s-1}\mathcal{H}r \ \partial_{x}\mathcal{H}(r_{1}\partial_{x}r_{1}) \ D_{x}^{s-1}\partial_{x}r \ dx+\int_{\mathbb{R}} D_{x}^{s-1}\mathcal{H}r \ \partial_{x}(|q_{1}|^{2})  \ D_{x}^{s-1}\partial_{x}r \ dx=\sum_{j=1}^{6}\widetilde{VI}_{2,j}.
\end{align*}
Then, note that the first term $\widetilde{VI}_{2,1}$ can be estimated as follows
\begin{align}
|\widetilde{VI}_{2,1}|\lesssim \|\mathcal{H}\partial_{x}^{2}r_{1}\|_{L^{\infty}_{x}}\|r\|_{H^{s}_{x}}^{2}.
\end{align}
\indent The term $\widetilde{VI_{2,2}}$ will be treated later together with the term $\widetilde{VI}_{1,2}$ above.
\newline
\indent On the other hand, the terms $\widetilde{VI}_{2,j},\ j=3,\ldots,6$, can be estimated similarly using the Sobolev embedding as follows
\vspace{-2em}
\begin{spacing}{1.3}
\noindent
\begin{align}
&|\widetilde{VI}_{2,3}|+|\widetilde{VI}_{2,4}|+|\widetilde{VI}_{2,5}|+|\widetilde{VI}_{2,6}|\nonumber \\
&\lesssim \|D_{x}^{s-1}\mathcal{H}r\|_{L^{\infty}_{x}}(\|r_{1}\partial_{x}r_{1}\|_{L^{2}_{x}}+\|\partial_{x}(r_{1}\mathcal{H}\partial_{x}r_{1})\|_{L^{2}_{x}}+\|\partial_{x}\mathcal{H}(r_{1}\partial_{x}r_{1})\|_{L^{2}_{x}}+\|\partial_{x}(|q_{1}|^{2})\|_{L^{2}_{x}})\|D_{x}^{s}r\|_{L^{2}_{x}}\nonumber \\
&\lesssim (\|r_{1}\|_{H^{s}_{x}}^{2}+\|r_{1}\|_{H^{s}_{x}}^{2}+\|r_{1}\|_{H^{s}_{x}}\|\mathcal{H}\partial_{x}^{2}r_{1}\|_{L^{\infty}_{x}}+\|r_{1}\|_{H^{s}_{x}}^{2}+\|r_{1}\|_{H^{s}_{x}}\|\partial_{x}^{2}r_{1}\|_{L^{\infty}_{x}}+\|q_{1}\|_{H^{s}_{x}}^{2})  \|r\|_{H^{s}_{x}}^{2}\nonumber \\
&\lesssim \|r_{1}\|_{H^{s}_{x}}(\|r_{1}\|_{H^{s}_{x}}+\|\mathcal{H}\partial_{x}^{2}r_{1}\|_{L^{\infty}_{x}}\|\partial_{x}^{2}r_{1}\|_{L^{\infty}})\|r\|_{H^{s}_{x}}^{2}+\|q_{1}\|_{H^{s}_{x}}^{2}\|r\|_{H^{s}_{x}}^{2}.
\end{align}
\end{spacing}
\vspace{-2em}
Now, concerning the term $\widetilde{VI}_{3}$, from the first equation in \eqref{NLS-BO2}, it follows that
\begin{align*}
\widetilde{VI}_{3} &= \int_{\mathbb{R}}D_{x}^{s-1}\mathcal{H}r \ r_{1}\ D_{x}^{s-1}\partial_{x}\mathcal{H}\partial_{x}^{2}r \ dx-\int_{\mathbb{R}}D_{x}^{s-1}\mathcal{H}r \ r_{1}\ D_{x}^{s-1}\partial_{x}\partial_{x}^{3}r \ dx\\
&\quad +\frac{1}{2}\int_{\mathbb{R}}D_{x}^{s-1}\mathcal{H}r \ r_{1}\ D_{x}^{s-1}\partial_{x}\partial_{x}(ur) \ dx-\int_{\mathbb{R}}D_{x}^{s-1}\mathcal{H}r \ r_{1}\ D_{x}^{s-1}\partial_{x}\partial_{x}(r\mathcal{H}\partial_{x}r_{1}) \ dx\\
&\quad -\int_{\mathbb{R}}D_{x}^{s-1}\mathcal{H}r \ r_{1}\ D_{x}^{s-1}\partial_{x}\partial_{x}\mathcal{H}(r\partial_{x}r_{1}) \ dx-\int_{\mathbb{R}}D_{x}^{s-1}\mathcal{H}r \ r_{1}\ D_{x}^{s-1}\partial_{x}\partial_{x}(r_{2}\mathcal{H}\partial_{x}r) \ dx\\
&\quad -\int_{\mathbb{R}}D_{x}^{s-1}\mathcal{H}r \ r_{1}\ D_{x}^{s-1}\partial_{x}\partial_{x}\mathcal{H}(r_{2}\partial_{x}r) \ dx+\textrm{Re}\int_{\mathbb{R}}D_{x}^{s-1}\mathcal{H}r \ r_{1}\ D_{x}^{s-1}\partial_{x}\partial_{x}(\bar{v}q) \ dx=\sum_{j=1}^{8}\widetilde{VI}_{3,j}.
\end{align*}
Thus, note that integrating by parts and Sobolev's embedding yield
\begin{equation}
|\widetilde{VI}_{3,1}|\lesssim \|\partial_{x}r_{1}\|_{L^{\infty}_{x}}\|D_{x}^{s}r\|_{L^{2}_{x}}^{2}+\|\partial_{x}^{2}r_{1}\|_{L^{\infty}_{x}}\|r\|_{H^{s}_{x}}^{2}\lesssim \|r_{1}\|_{H^{s}_{x}}\|r\|_{H^{s}_{x}}^{2}+\|\partial_{x}^{2}r_{1}\|_{L^{\infty}_{x}}\|r\|_{H^{s}_{x}}^{2}.
\end{equation} \label{VI~31}
Next, from Lemma \ref{sum3derivative} we can see that
\begin{align*}
\widetilde{VI}_{1,2}+\widetilde{VI}_{2,2}+\widetilde{VI}_{3,2}=-3\int_{\mathbb{R}}D_{x}^{s}r \ \partial_{x}r_{1} \ D_{x}^{s-1}\partial_{x}^{2}r \ dx.
\end{align*}
Thus,
\begin{equation}
   \widetilde{II}_{3,3}+ \tau_{s,1}\cdot (\widetilde{VI}_{1,2}+\widetilde{VI}_{2,2}+\widetilde{VI}_{3,2})=0.
\end{equation}
Concerning the terms $\widetilde{VI}_{3,j},\ j=3,\ldots,4$, an integration by parts, the fact that $H^{s}(\mathbb{R})$ and $H^{s-1}(\mathbb{R})$ are algebras, Lemma \ref{commestim}, along with the Sobolev embedding, yield that
\begin{align}
|\widetilde{VI}_{3,3}|+|\widetilde{VI}_{3,8}| & \lesssim(\|u\|_{H^{s}_{x}}+\|v\|_{H^{s}_{x}})\|r_{1}\|_{H^{s}_{x}}(\|r\|_{H^{s}_{x}}^{2}+\|q\|_{H^{s}_{x}}^{2}),\\
|\widetilde{VI}_{3,4}|+|\widetilde{VI}_{3,5}|&\lesssim  (\|r_{1}\|_{H^{s}_{x}}^{2}+\|\partial_{x}^{2}r_{1}\|_{L^{\infty}_{x}}\|r_{1}\|_{H^{s}_{x}})\|r\|_{H^{s}_{x}}^{2}+\|D_{x}^{s}r\|_{L^{2}_{x}}\|r\|_{L^{2}_{x}}\|r_{1}\|_{L^{\infty}_{x}}\|D_{x}^{s}\partial_{x}r_{1}\|_{L^{\infty}_{x}},\\
|\widetilde{VI}_{3,6}|+|\widetilde{VI}_{3,7}|&\lesssim \|r_{1}\|_{H^{s}_{x}}\|r_{2}\|_{H^{s}_{x}}\|r\|_{H^{s}_{x}}^{2}+\|\partial_{x}^{2}r_{1}\|_{L^{\infty}_{x}}\|r_{2}\|_{H^{s}_{x}}\|r\|_{H^{s}_{x}}^{2}.
\label{VII~38}
\end{align}
\indent We now observe that the term $\widetilde{VII}$ can be treated exactly as we did with $\widetilde{VI}$, since we have that $r$ also satisfies the following equation
\begin{equation}
     \partial_{t}r-\mathcal{H}\partial_{x}^{2}r+\partial_{x}^{3}r=\frac{1}{2}\partial_{x}(ur)-\partial_{x}(r\mathcal{H}\partial_{x}r_{2}+\mathcal{H}(r\partial_{x}r_{2}))-\partial_{x}(r_{1}\mathcal{H}\partial_{x}r+\mathcal{H}(r_{1}\partial_{x}r))+\partial_{x}\textrm{Re}(\bar{v}q),
\end{equation}
which is basically the first equation in \eqref{NLS-BO2}, with $r_{1}$ and $r_{2}$ changing roles. In this case, if $M=M(r_{1},r_{2})>0$ is the sum of all the terms on the right hand side of each inequality \eqref{VI~11}-\eqref{VII~38} and that bound $\widetilde{VI}$, then
\begin{equation}
    \left|\frac{1}{\tau_{s,2}}\widetilde{VII}+3\int_{\mathbb{R}}D_{x}^{s}r \ \partial_{x}r_{2} \ D_{x}^{s-1}\partial_{x}^{2}r \ dx\right| \lesssim M(r_{2},r_{1}),
\end{equation}
i.e, we can simply invert $r_{1}$ and $r_{2}$ in terms of $M(r_{1},r_{2})$. On the other hand, note that
\begin{equation}\label{II~44}
\widetilde{II}_{4,2,3}+\widetilde{II}_{4,4}+\widetilde{II}_{5,2,3}+\widetilde{II}_{5,4}-3\tau_{s,2}\int_{\mathbb{R}}D_{x}^{s}r \ \partial_{x}r_{2} \ D_{x}^{s-1}\partial_{x}^{2}r=0.
\end{equation}
Therefore, from \eqref{boundedrq} and \eqref{ddtE~s}-\eqref{II~44}, follow the estimate \eqref{esd2}.

Finally, we proceed with the demonstration of estimate \eqref{esd3}. First, differentiating $\widetilde{E}_{m}^{0}(t)$ in \eqref{em3} with respect to $t$, we obtain
\begin{align}\label{difftoE0}
\frac{d}{dt}\widetilde{E}_{m}^{0}(t)&=\frac{d}{dt}\left(\int_{\mathbb{R}}r^{2} \ dx\right)+\frac{d}{dt}\left(\int_{\mathbb{R}}|q|^{2} \ dx\right)-\tau_{0}\frac{d}{dt}\left(\int_{\mathbb{R}} \mathcal{H}r \ J_{x}^{-1}\mathcal{H}r \ r_{1} \ dx\right) +\frac{d}{dt}\left(\textrm{Re}\int_{\mathbb{R}}J_{x}^{-2}r \ \bar{v}q \ dx\right)\nonumber \\
&=\widetilde{\mathcal{I}}+\mathcal{\widetilde{II}}+\mathcal{\widetilde{III}}+\widetilde{\mathcal{IV}}.
\end{align}
Now, using the first equation in \eqref{NLS-BO2}, we obtain that
\begin{align*}
    \mathcal{\widetilde{I}}&=\frac{1}{2}\int_{\mathbb{R}}r\partial_{x}(ur) \ dx-\int_{\mathbb{R}}r\partial_{x}(r\mathcal{H}\partial_{x}r_{1}) \ dx-\int_{\mathbb{R}}r\partial_{x}\mathcal{H}(r\partial_{x}r_{1}) \ dx\\
    &\quad -\int_{\mathbb{R}}r\partial_{x}(r_{2}\mathcal{H}\partial_{x}r) \ dx-\int_{\mathbb{R}}r\partial_{x}\mathcal{H}(r_{2}\partial_{x}r) \ dx+\textrm{Re}\int_{\mathbb{R}}r\partial_{x}(\bar{v}q) \ dx\\
    &=\widetilde{\mathcal{I}}_{1}+\widetilde{\mathcal{I}}_{2}+\widetilde{\mathcal{I}}_{3}+\widetilde{\mathcal{I}}_{4}+\widetilde{\mathcal{I}}_{5}+\widetilde{\mathcal{I}}_{6}.
\end{align*}
Then, note that integration by parts yields that
\begin{align}
    |\widetilde{\mathcal{I}}_{1}|+  |\mathcal{\widetilde{I}}_{2}|&\lesssim (\|\partial_{x}u\|_{L^{\infty}_{x}}+\|\mathcal{H}\partial_{x}^{2}r_{1}\|_{L^{\infty}_{x}})\|r\|_{L^{2}_{x}}^{2}.
\end{align}
\indent On the other hand, it can be observed that the term \( \widetilde{\mathcal{I}}_{3} \) cannot be directly estimated. As has repeatedly occurred, this term will eventually cancel out with a subsequent one. Nevertheless, such a cancellation requires a complete description of the term \( \widetilde{\mathcal{III}} \), which will be addressed subsequently.
\indent Next, we can write
\begin{align*}
\widetilde{\mathcal{I}}_{4}=\int_{\mathbb{R}}\partial_{x}r\ r_{2}\mathcal{H}\partial_{x}r \ dx=\int_{\mathbb{R}}\partial_{x}\mathcal{H}(\partial_{x}rr_{2})r \ dx=-\widetilde{\mathcal{I}}_{5}.
\end{align*}
Hence,
\begin{equation}
    \widetilde{\mathcal{I}}_{4}+\mathcal{\widetilde{I}}_{5}=0.
\end{equation}
\indent Observe that the term $\widetilde{\mathcal{I}}_{6}$ also presents certain difficulties. Therefore, in order to handle it appropriately, it is necessary to first provide a detailed description of the term $\mathcal{\widetilde{IV}}$.
\newline
\indent The term $\mathcal{\mathcal{II}}$ can be estimated exactly as it was done for the term $\widetilde{III}$ in \eqref{estimativadeql2}. Namely, we have that
\begin{equation}
    |\widetilde{\mathcal{II}}|\lesssim (\|v\|_{L^{\infty}_{x}}+\|u\|_{L^{\infty}_{x}})(\|r\|_{L^{2}_{x}}^{2}+\|q\|_{L^{2}_{x}}^{2}).
\end{equation}
\indent Concerning the term $\widetilde{\mathcal{III}}$, notice that
\begin{align*}
\frac{1}{\tau_{0}}\widetilde{\mathcal{III}}&=-\int_{\mathbb{R}} \mathcal{H}\partial_{t}r \ J_{x}^{-1}\mathcal{H}r \ r_{1} \ dx-\int_{\mathbb{R}} \mathcal{H}r \ J_{x}^{-1}\mathcal{H}\partial_{t}r \ r_{1} \ dx-\int_{\mathbb{R}} \mathcal{H}r \ J_{x}^{-1}\mathcal{H}r \ \partial_{t}r_{1} \ dx\\
&=\widetilde{\mathcal{III}}_{1}+\widetilde{\mathcal{III}}_{2}+\widetilde{\mathcal{III}}_{3}.
\end{align*}
Now, using the first equation in \eqref{NLS-BO2}, we obtain
\begin{align*}
\widetilde{\mathcal{III}}_{1}&=\int_{\mathbb{R}}\partial_{x}^{3}\mathcal{H}r \ J_{x}^{-1}\mathcal{H}r \ r_{1} \ dx-\int_{\mathbb{R}}\mathcal{H}\mathcal{H}\partial_{x}^{2}r \ J_{x}^{-1}\mathcal{H}r \ r_{1} \ dx-\frac{1}{2}\int_{\mathbb{R}}\mathcal{H}\partial_{x}(ur) \ J_{x}^{-1}\mathcal{H}r \ r_{1} \ dx\\
&\quad +\int_{\mathbb{R}}\mathcal{H}\partial_{x}(r\mathcal{H}\partial_{x}r_{1}) \ J_{x}^{-1}\mathcal{H}r \ r_{1} \ dx+\int_{\mathbb{R}}\mathcal{H}\partial_{x}\mathcal{H}(r\partial_{x}r_{1}) \ J_{x}^{-1}\mathcal{H}r \ r_{1} \ dx\\
&\quad 
+\int_{\mathbb{R}}\mathcal{H}\partial_{x}\mathcal{H}(r_{2}\partial_{x}r) \ J_{x}^{-1}\mathcal{H}r \ r_{1} \ dx+\int_{\mathbb{R}}\mathcal{H}\partial_{x}(r_{2}\mathcal{H}\partial_{x}r) \ J_{x}^{-1}\mathcal{H}r \ r_{1} \ dx+\textrm{Re}\int_{\mathbb{R}}\mathcal{H}\partial_{x}(\bar{v}q) \ J_{x}^{-1}r \ r_{1} \ dx\\
&=\widetilde{\mathcal{III}}_{1,1}+\widetilde{\mathcal{III}}_{1,2}+\widetilde{\mathcal{III}}_{1,3}+\widetilde{\mathcal{III}}_{1,4}+\widetilde{\mathcal{III}}_{1,5}+\widetilde{\mathcal{III}}_{1,6}+\widetilde{\mathcal{III}}_{1,7}+\widetilde{\mathcal{III}}_{1,8}.
\end{align*}
\indent The term $\mathcal{\widetilde{III}}_{1,1}$ will be handled later, since we first need to describe $\widetilde{\mathcal{III}}_{2}$ and $\widetilde{\mathcal{III}}_{3}$.
\newline
\indent Considering the term $\mathcal{\widetilde{III}}_{1,2}$, integration by parts yields
\begin{align*}
\widetilde{\mathcal{III}}_{1,2}&= -\int_{\mathbb{R}}\partial_{x}r \ J_{x}^{-1}\mathcal{H}\partial_{x}r \ r_{1} \ dx-\int_{\mathbb{R}}\partial_{x}r \ J_{x}^{-1}\mathcal{H}r \ \partial_{x}r_{1} \ dx  \\
&=\int_{\mathbb{R}}rJ_{x}^{-1}\mathcal{H}\partial_{x}^{2}r \ r_{1} \ dx+2\int_{\mathbb{R}}r \ J_{x}^{-1}\mathcal{H}\partial_{x}r \ \partial_{x}r_{1} \ dx+\int_{\mathbb{R}}r \ J_{x}^{-1}\mathcal{H}r \ \partial_{x}^{2}r \ dx\\
&=\widetilde{\mathcal{III}}_{1,2,1}+\widetilde{\mathcal{III}}_{1,2,2}+\widetilde{\mathcal{III}}_{1,2,3}.
\end{align*}
But observe that we can write
\begin{align*}
\widetilde{III}_{1,2,1}&=\int_{\mathbb{R}}r(J_{x}^{-1}\mathcal{H}\partial_{x}^{2}-\partial_{x})r \ r_{1} \ dx+\int_{\mathbb{R}}r\partial_{x}r\ r_{1} \ dx=\int_{\mathbb{R}}r(J_{x}^{-1}\mathcal{H}\partial_{x}^{2}-\partial_{x})r \ r_{1} \ dx-\frac{1}{2}\int_{\mathbb{R}}r^{2}\partial_{x}r_{1} \ dx.
\end{align*}
Hence, an application of Hölder's inequality and Lemma \ref{boundedoperator1} yield
\begin{align}
    |\widetilde{\mathcal{III}}_{1,2,1}|&\lesssim (\|r_{1}\|_{L^{\infty}_{x}}+\|\partial_{x}r_{1}\|_{L^{\infty}_{x}})\|r\|_{L^{2}_{x}}^{2},\\
    |\widetilde{\mathcal{III}}_{1,2,2}|+|\widetilde{\mathcal{III}}_{1,2,3}|&\lesssim (\|\partial_{x}r_{}
   \|_{L^{\infty}_{x}}+\|\partial_{x}^{2}r\|_{L^{\infty}_{x}} )\|r\|_{L^{2}_{x}}^{2}.
\end{align}
\indent Now, an integration by parts, along with Sobolev's embedding, yields
\begin{align}
|\widetilde{\mathcal{III}}_{1,3}|+|\widetilde{\mathcal{III}}_{1,8}|&\lesssim (\|u\|_{L^{\infty}_{x}}+\|v\|_{L^{\infty}_{x}})(\|r_{1}\|_{L^{\infty}}+\|\partial_{x}r_{1}\|_{L^{\infty}_{x}})(\|r\|_{L^{2}_{x}}^{2}+\|q\|_{L^{2}_{x}}^{2}), \\
|\widetilde{\mathcal{III}}_{1,4}|&\lesssim \|\mathcal{H}\partial_{x}r_{1}\|_{L^{\infty}_{x}}(\|r_{1}\|_{L^{\infty}_{x}}+\|\partial_{x}r_{1}\|_{L^{\infty}_{x}})\|r\|_{L^{2}_{x}}^{2},\\
|\widetilde{\mathcal{III}}_{1,5}|&\lesssim  \|r_{1}\|_{L^{\infty}_{x}}\|\partial_{x}r_{1}\|_{L^{\infty}_{x}}\|r\|_{L^{2}_{x}}^{2}+\|\partial_{x}r_{1}\|_{L^{\infty}_{x}}^{2}\|r\|_{L^{2}_{x}}^{2}.
\end{align}
\indent The term $\widetilde{\mathcal{III}}_{1,6}$ presents an additional difficulty due to the presence of a derivative with respect to r in the nonlinear component. However, we can write
\begin{align*}
 \widetilde{\mathcal{III}}_{1,6}&=- \int_{\mathbb{R}}\partial_{x}^{2}(r_{2} r)J_{x}^{-1}\mathcal{H}r \ r_{1} \ dx+\int_{\mathbb{R}}\partial_{x}(\partial_{x}r_{2}r) \ J_{x}^{-1}\mathcal{H}r \ r_{1} \ dx\\
&=\widetilde{\mathcal{III}}_{1,6,1}+\widetilde{\mathcal{III}}_{1,6,2}.
\end{align*}
On the one hand, by integration by parts
\begin{align*}
\widetilde{\mathcal{III}}_{1,6,1}&=-\int_{\mathbb{R}}r_{2}r \ J_{x}^{-1}\mathcal{H}\partial_{x}^{2}r \ r_{1} \ dx-2\int_{\mathbb{R}}r_{2}r \ J_{x}^{-1}\mathcal{H}\partial_{x}r \ \partial_{x}r_{1} \ dx -\int_{\mathbb{R}}r_{2}r \ J_{x}^{-1}\mathcal{H}r \ \partial_{x}^{2}r_{1} \ dx\\
&=\widetilde{\mathcal{III}}_{1,6,1,1}+\widetilde{\mathcal{III}}_{1,6,1,2}+\widetilde{\mathcal{III}}_{1,6,1,3}.
\end{align*}
Now, note that
\begin{align*}
    \widetilde{\mathcal{III}}_{1,6,1,1}&=-\int_{\mathbb{R}}r_{2}r(J^{-1}_{x}\mathcal{H}\partial_{x}^{2} -\partial_{x}r)\ r_{1} \ dx-\int_{\mathbb{R}}r_{2}r\ \partial_{x}r\ r_{1} \ dx\\
    &=-\int_{\mathbb{R}}r_{2}r(J^{-1}_{x}\mathcal{H}\partial_{x}^{2} \mathcal{H}-\partial_{x}r)\ r_{1} \ dx+\frac{1}{2}\int_{\mathbb{R}}r^{2}\partial_{x}(r_{1}r_{2}) \ dx.
\end{align*}
Hence, Lemma \ref{boundedoperator1} yields
\begin{equation}
|\widetilde{\mathcal{III}}_{1,6,1,1}|\lesssim \|r_{1}\|_{L^{\infty}_{x}}\|r_{2}\|_{L^{\infty}_{x}}\|r\|_{L^{2}_{x}}^{2}+(\|r_{1}\|_{L^{\infty}_{x}}\|\partial_{x}r_{2}\|_{L^{\infty}_{x}}+\|r_{2}\|_{L^{\infty}_{x}}\|\partial_{x}r_{1}\|_{L^{\infty}_{x}})\|r\|_{L^{2}_{x}}^{2}.
\end{equation}
Also, from Hölder's inequality, we have 
\begin{align}
    |\widetilde{\mathcal{III}}_{1,6,1,2}|+|\widetilde{\mathcal{III}}_{1,6,1,3}|&\lesssim (\|r_{2}\|_{L^{\infty}_{x}}\|\partial_{x}r_{1}\|_{L^{\infty}_{x}}+\|\partial_{x}^{2}r_{1}\|_{L^{\infty}_{x}})\|r\|_{L^{2}_{x}}^{2}.
\end{align}
On the other hand,
\begin{align*}
\widetilde{\mathcal{III}}_{1,6,2}=-\int_{\mathbb{R}}\partial_{x}r_{2}r \ J_{x}^{-1}\mathcal{H}\partial_{x}r \ r_{1} \ dx-\int_{\mathbb{R}}\partial_{x}r_{2}r \ J_{x}^{-1}\mathcal{H}r \ \partial_{x}r_{1} \ dx 
\end{align*}
Thus,
\begin{equation}
|\widetilde{\mathcal{III}}_{1,6,2}|\lesssim \|r_{1}\|_{L^{\infty}_{x}}\|\partial_{x}r_{2}\|_{L^{\infty}_{x}}\|r\|_{L^{2}_{x}}^{2}+\|\partial_{x}r_{1}\|_{L^{\infty}_{x}}\|\partial_{x}r_{2}\|_{L^{\infty}_{x}}\|r\|_{L^{2}_{x}}^{2}.
\end{equation}
By applying to $\widetilde{\mathcal{III}}_{1,7}$ the same arguments used for the term $\widetilde{\mathcal{III}}_{1,6}$, we obtain that
\begin{align}
|\widetilde{\mathcal{III}}_{1,7}|&\lesssim \|r_{1}\|_{L^{\infty}_{x}}\|r_{2}\|_{L^{\infty}_{x}}\|r\|_{L^{2}_{x}}^{2}+\|r_{1}\|_{L^{\infty}_{x}}\|\partial_{x}r_{2}\|_{L^{\infty}_{x}}\|r\|_{L^{2}_{x}}^{2}+\|r_{2}\|_{L^{\infty}_{x}}\|\partial_{x}r_{1}\|_{L^{\infty}_{x}}\|r\|_{L^{2}_{x}}^{2}\nonumber \\
&\quad +\|r_{2}\|_{L^{\infty}_{x}}\|\partial_{x}^{2}r_{1}\|_{L^{\infty}_{x}}\|r\|_{L^{2}_{x}}^{2}+\|\partial_{x}r_{1}\|_{L^{\infty}_{x}}\|\partial_{x}r_{2}\|_{L^{\infty}_{x}}\|r\|_{L^{2}_{x}}^{2}.
\end{align}
\indent Next, using again the first equation in \eqref{NLS-BO2}, we can write
\begin{align*}
\widetilde{\mathcal{III}}_{2}&=\int_{\mathbb{R}} \mathcal{H}r \ J_{x}^{-1}\mathcal{H}\partial_{x}^{3}r  \ r_{1} \ dx-\int_{\mathbb{R}} \mathcal{H}r \ J_{x}^{-1}\mathcal{H}\mathcal{H}\partial_{x}^{2}r  \ r_{1} \ dx-\frac{1}{2}\int_{\mathbb{R}} \mathcal{H}r \ J_{x}^{-1}\mathcal{H}\partial_{x}(ur)  \ r_{1} \ dx\\
&\quad +\int_{\mathbb{R}} \mathcal{H}r \ J_{x}^{-1}\mathcal{H}\partial_{x}(r\mathcal{H}\partial_{x}r_{1})  \ r_{1} \ dx+\int_{\mathbb{R}} \mathcal{H}r \ J_{x}^{-1}\mathcal{H}\partial_{x}\mathcal{H}(r\partial_{x}r_{1})  \ r_{1} \ dx\\
&\quad +\int_{\mathbb{R}} \mathcal{H}r \ J_{x}^{-1}\mathcal{H}\partial_{x}\mathcal{H}(r_{2}\partial_{x}r)  \ r_{1} \ dx+\int_{\mathbb{R}} \mathcal{H}r \ J_{x}^{-1}\mathcal{H}\partial_{x}(r_{2}\mathcal{H}\partial_{x}r)  \ r_{1} \ dx
\\
&\quad -\textrm{Re}\int_{\mathbb{R}} \mathcal{H}r \ J_{x}^{-1}\mathcal{H}\partial_{x}(\bar{v}q)  \ r_{1} \ dx\\
&=\widetilde{\mathcal{III}}_{2,1}+\widetilde{\mathcal{III}}_{2,2}+\widetilde{\mathcal{III}}_{2,3}+\widetilde{\mathcal{III}}_{2,4}+\widetilde{\mathcal{III}}_{2,5}+\widetilde{\mathcal{III}}_{2,6}+\widetilde{\mathcal{III}}_{2,7}+\widetilde{\mathcal{III}}_{2,8}.
\end{align*}
\noindent The term $\widetilde{\mathcal{III}}_{2,1}$ will be handled later along with the term $\widetilde{\mathcal{III}}_{1,1}$. Now, note that
\begin{align*}
 \widetilde{\mathcal{III}}_{2,2}&=-\int_{\mathbb{R}}\mathcal{H}r \ \mathcal{H}(J^{-1}\mathcal{H}\partial_{x}^{2}-\partial_{x})r \ r_{1} \ dx-\int_{\mathbb{R}}\mathcal{H}r \ \mathcal{H}\partial_{x}r \ r_{1} \ dx\\
 &=-\int_{\mathbb{R}}\mathcal{H}r \ \mathcal{H}(J^{-1}\mathcal{H}\partial_{x}^{2}-\partial_{x})r \ r_{1} \ dx+\frac{1}{2}\int_{\mathbb{R}}(\mathcal{H}r)^{2}\partial_{x}r_{1} \ dx.
\end{align*}
Therefore, the use of Hölder's inequality and Lemma \ref{boundedoperator1} lead to
\begin{align}
    |\widetilde{\mathcal{III}}_{2,2}|&\lesssim (\|r_{1}\|_{L^{\infty}_{x}}+\|\partial_{x}r_{1}\|_{L^{\infty}_{x}})\|r\|_{L^{2}_{x}}^{2},\\
    |\widetilde{\mathcal{III}}_{2,3}|+|\widetilde{\mathcal{III}}_{2,8}|&\lesssim \|r_{1}\|_{L^{\infty}_{x}}(\|u\|_{L^{\infty}_{x}}+\|v\|_{L^{\infty}_{x}})(\|r\|_{L^{2}_{x}}^{2}+\|q\|_{L^{2}_{x}}^{2}),\\
    |\widetilde{\mathcal{III}}_{2,4}|+|\widetilde{\mathcal{III}}_{2,5}|&\lesssim \|r_{1}\|_{L^{\infty}_{x}}(\|\mathcal{H}\partial_{x}r_{1}\|_{L^{\infty}_{x}}+\|\partial_{x}r_{1}\|_{L^{\infty}_{x}})\|r\|_{L^{2}_{x}}^{2}.
\end{align}
\indent As before, the term $\widetilde{\mathcal{III}}_{2,6}$ is more problematic since it contains a derivative with respect to $r$ in a nonlinear term. But we can write
\begin{align*}
\widetilde{\mathcal{III}}_{2,6}&=\int_{\mathbb{R}}\mathcal{H}r \ J_{x}^{-1}\mathcal{H}^{2}\partial_{x}^{2}(r_{2}r)\ r_{1} \ dx-\int_{\mathbb{R}}\mathcal{H}r \ J_{x}^{-1}\mathcal{H}^{2}\partial_{x}(\partial_{x}r_{2}\ r) \ r_{1} \ dx\\
&=\int_{\mathbb{R}}\mathcal{H}r\ \mathcal{H}(J_{x}^{-1}\mathcal{H}\partial_{x}^{2}-\partial_{x})(r_{2}r) \ r_{1} \ dx+\int_{\mathbb{R}}\mathcal{H}r\ \mathcal{H}\partial_{x}(r_{2}r) \ r_{1} \ dx -\int_{\mathbb{R}}\mathcal{H}r \ J_{x}^{-1}\mathcal{H}^{2}\partial_{x}(\partial_{x}r_{2}\ r) \ r_{1} \ dx\\\
&=\widetilde{\mathcal{III}}_{2,6,1}+\widetilde{\mathcal{III}}_{2,6,2}+\widetilde{\mathcal{III}}_{2,6,3}.
\end{align*}
Thus, an application of Lemma \ref{boundedoperator1} yields  
\begin{equation}
|\widetilde{\mathcal{III}}_{2,6,1}|\lesssim \|r_{1}\|_{L^{\infty}_{x}}\|r_{2}\|_{L^{\infty}_{x}}\|r\|_{L^{2}_{x}}^{2}.
\end{equation}
On the other hand, note that
\begin{align*}
\widetilde{\mathcal{III}}_{2,6,2}&=\int_{\mathbb{R}}\mathcal{H}r\ [D_{x}^{1};r_{2}]r \ r_{1} \ dx+\int_{\mathbb{R}}\mathcal{H}r\ r_{2}D_{x}^{1}r \ r_{1} \ dx\\ 
&=\int_{\mathbb{R}}\mathcal{H}r\ [D_{x}^{1};r_{2}]r \ r_{1} \ dx-\frac{1}{2}\int_{\mathbb{R}}(\mathcal{H}r)^{2}\partial_{x}(r_{1}r_{2}) \ dx.
\end{align*}
Therefore, from Lemma \ref{commestim},  it follows that
\begin{equation}
    |\widetilde{\mathcal{III}}_{2,6,2}|\lesssim \|r\|_{L^{2}_{x}}\|\partial_{x}r_{2}\|_{L^{\infty}_{x}}\|r\|_{L^{2}_{x}}\|r_{1}\|_{L^{\infty}_{x}}+(\|r_{1}\|_{L^{\infty}_{x}}\|\partial_{x}r_{2}\|_{L^{\infty}_{x}}+\|r_{2}\|_{L^{\infty}_{x}}\|\partial_{x}r_{1}\|_{L^{\infty}_{x}})\|r\|_{L^{2}_{x}}^{2}.
\end{equation}
Furthermore, the Hölder inequality yields
\begin{equation}
|\widetilde{\mathcal{III}}_{2,6,3}|\lesssim \|r_{1}\|_{L^{\infty}_{x}}\|\partial_{x}r_{2}\|_{L^{\infty}_{x}}\|r\|_{L^{2}_{x}}^{2}.
\end{equation}
Similarly, we can deduce that
\begin{equation}
|\widetilde{\mathcal{III}}_{2,7}|\lesssim \|r_{1}\|_{L^{\infty}_{x}}(\|r_{2}\|_{L^{\infty}_{x}}+\|\partial_{x}r_{2}\|_{L^{\infty}_{x}})\|r\|_{L^{2}_{x}}^{2}+\|r_{2}\|_{L^{\infty}_{x}}\|\partial_{x}r_{1}\|_{L^{\infty}_{x}})\|r\|_{L^{2}_{x}}^{2}
\end{equation}
\indent In the case of $\widetilde{\mathcal{III}}_{3}$, using now the first equation in \eqref{BO-NLS}, we can write
\begin{align*}
\widetilde{\mathcal{III}}_{3}&=\int_{\mathbb{R}}\mathcal{H}r \ J_{x}^{-1}\mathcal{H}r\ \partial_{x}^{3}r_{1}  \ dx-\int_{\mathbb{R}}\mathcal{H}r \ J_{x}^{-1}\mathcal{H}r\ \mathcal{H}\partial_{x}^{2}r_{1}  \ dx-\int_{\mathbb{R}}\mathcal{H}r \ J_{x}^{-1}\mathcal{H}r\ r_{1}\partial_{x}r_{1} \  dx\\
&\quad +\int_{\mathbb{R}}\mathcal{H}r \ J_{x}^{-1}\mathcal{H}r\  \ \partial_{x}(r_{1}\mathcal{H}\partial_{x}r_{1}) dx+\int_{\mathbb{R}}\mathcal{H}r \ J_{x}^{-1}\mathcal{H}r\ \partial_{x}\mathcal{H}(r_{1}\partial_{x}r_{1}) \ dx-\int_{\mathbb{R}}\mathcal{H}r \ J_{x}^{-1}\mathcal{H}r\ \partial_{x}(|q_{1}|^{2})  \ dx\\
&=\widetilde{\mathcal{III}}_{3,1}+\widetilde{\mathcal{III}}_{3,2}+\widetilde{\mathcal{III}}_{3,3}+\widetilde{\mathcal{III}}_{3,4}+\widetilde{\mathcal{III}}_{3,5}+\widetilde{\mathcal{III}}_{3,6}.
\end{align*}
Then, from Lemma \ref{sum3derivative}, we have
\begin{align*}
\widetilde{\mathcal{III}}_{1,1}+\widetilde{\mathcal{III}}_{2,1}+\widetilde{\mathcal{III}}_{3,1}&=3\int_{\mathbb{R}}D_{x}^{1}r \ J_{x}^{-1}\mathcal{H}\partial_{x}r \ \partial_{x}r_{1} \ dx\\
&=3\int_{\mathbb{R}}D_{x}^{1}r \ r \ \partial_{x}r_{1} \ dx+3\int_{\mathbb{R}}D_{x}^{1}r \ (J_{x}^{-1}\partial_{x}-1)r \ \partial_{x}r_{1} \ dx\\
&=\widetilde{\mathcal{III}}_{3,1,1}+\widetilde{\mathcal{III}}_{3,1,2}.
\end{align*}
Thus,
\begin{equation}
   \widetilde{\mathcal{I}}_{3}+ \tau_{0}\cdot \widetilde{\mathcal{III}}_{3,1,1}=0.
\end{equation}
On the other hand, integration by parts, Lemma \ref{boundedoperator1}, and the Sobolev embedding, yield that
\begin{align}
   |\widetilde{\mathcal{III}}_{3,1,2}|+|\widetilde{\mathcal{III}}_{3,2}|&\lesssim \|r\|_{L^{2}_{x}}^{2}(\|\partial_{x}r_{1}\|_{L^{\infty}_{x}}+\|\partial_{x}^{2}r_{1}\|_{L^{\infty}_{x}}+\|\mathcal{H}\partial_{x}^{2}r_{1}\|_{L^{\infty}_{x}}),\\
    |\widetilde{\mathcal{III}}_{3,3}|+ |\widetilde{\mathcal{III}}_{3,6}|&\lesssim \|r\|_{L^{2}_{x}}^{2}(\|r_{1}\|_{L^{2}_{x}}\|\partial_{x}r_{1}\|_{L^{\infty}_{x}}+\|q_{1}\|_{L^{2}_{x}}\|\partial_{x}q_{1}\|_{L^{\infty}_{x}}),\\
    |\widetilde{\mathcal{III}}_{3,4}|+|\widetilde{\mathcal{III}}_{3,5}|&\lesssim \|r\|_{L^{2}_{x}}^{2}(\|\partial_{x}r_{1}\|_{L^{\infty}_{x}}\|\mathcal{H}\partial_{x}r_{1}\|_{L^{\infty}_{x}}+\|r_{1}\|_{L^{2}_{x}}\|\mathcal{H}\partial_{x}^{2}r_{1}\|_{L^{\infty}_{x}}+\|\partial_{x}^{2}r_{1}\|_{L^{\infty}_{x}}\|r_{1}\|_{L^{2}_{x}}).
\end{align}
\indent Finally, considering the term $\widetilde{\mathcal{IV}}$, we can write
\begin{align*}
\widetilde{\mathcal{IV}}&=\textrm{Re}\int_{\mathbb{R}}J_{x}^{-2}\partial_{t}r \ \bar{v}q \ dx+\textrm{Re}\int_{\mathbb{R}}J_{x}^{-2}r \ \partial_{t}\bar{v}q \ dx+\textrm{Re}\int_{\mathbb{R}}J_{x}^{-2}r \ \bar{v}\partial_{t}q \ dx\\
&=\widetilde{\mathcal{IV}}_{1}+\widetilde{\mathcal{IV}}_{2}+\widetilde{\mathcal{IV}}_{3}.
\end{align*}
Thus, using the second equation in \eqref{NLS-BO2}, it follows that
\begin{align*}
\widetilde{\mathcal{IV}}_{1}&=-\textrm{Re}\int_{\mathbb{R}}J_{x}^{-2}\partial_{x}^{3}r \ \bar{v}q \ dx+\textrm{Re}\int_{\mathbb{R}}J_{x}^{-2}\mathcal{H}\partial_{x}^{2}r \ \bar{v}q \ dx+\frac{1}{2}\textrm{Re}\int_{\mathbb{R}}J_{x}^{-2}\partial_{x}(ur) \ \bar{v}q \ dx\\
&\quad -\textrm{Re}\int_{\mathbb{R}}J_{x}^{-2}\partial_{x}(r\mathcal{H}\partial_{x}r_{1}) \ \bar{v}q \ dx-\textrm{Re}\int_{\mathbb{R}}J_{x}^{-2}\partial_{x}\mathcal{H}(r\partial_{x}r_{1}) \ \bar{v}q \ dx-\textrm{Re}\int_{\mathbb{R}}J_{x}^{-2}\partial_{x}(r_{2}\mathcal{H}\partial_{x}r) \ \bar{v}q \ dx\\
&\quad -\textrm{Re}\int_{\mathbb{R}}J_{x}^{-2}\partial_{x}\mathcal{H}(r_{2}\partial_{x}r) \ \bar{v}q \ dx+\textrm{Re}\int_{\mathbb{R}}J_{x}^{-2}\partial_{x}(\bar{v}q) \ \bar{v}q \ dx\\
&=\widetilde{\mathcal{IV}}_{1,1}+\widetilde{\mathcal{IV}}_{1,2}+\widetilde{\mathcal{IV}}_{1,3}+\widetilde{\mathcal{IV}}_{1,4}+\widetilde{\mathcal{IV}}_{1,5}+\widetilde{\mathcal{IV}}_{1,6}+\widetilde{\mathcal{IV}}_{1,7}+\widetilde{\mathcal{IV}}_{1,8}.
\end{align*}
Now, an integration by parts ensures that
\begin{align*}
\widetilde{\mathcal{VI}}_{1,1}\
&=-\textrm{Re}\int_{\mathbb{R}}(J_{x}^{-2}\partial_{x}^{3}+\partial_{x})r \ \bar{v}q \ dx-\textrm{Re}\int_{\mathbb{R}}r\partial_{x}(\bar{v}q) \ dx\\
&=\widetilde{\mathcal{IV}}_{1,1,1}+\widetilde{\mathcal{IV}}_{1,1,2}.
\end{align*}
Then, Lemma \ref{boundedoperator1} yields
\begin{equation}
    |\widetilde{\mathcal{IV}}_{1,1,1}|\lesssim \|v\|_{L^{\infty}_{x}}\|r\|_{L^{2}_{x}}\|q\|_{L^{2}_{x}}.
\end{equation}
Furthermore, we have
\begin{equation}
    \widetilde{\mathcal{IV}}_{1,1,2}+\mathcal{\widetilde{I}}_{6}=0.
\end{equation}
On the other hand, an application of integration by parts, along with Hölder's inequality, leads us to
\begin{align}
\sum_{j=2}^{8} |\widetilde{\mathcal{IV}}_{1,j}|
&\lesssim
\|v\|_{L^{\infty}_x}
(
1 + \|u\|_{L^{\infty}_x}
+ \|\mathcal{H}\partial_x r_1\|_{L^{\infty}_x}
+ \|\partial_x r_1\|_{L^{\infty}_x}
+ \|r_2\|_{L^{\infty}_x}
+ \|\partial_x r_2\|_{L^{\infty}_x}
+ \|v\|_{L^{\infty}_x}
)
\nonumber\\
&\quad \times (\|r\|_{L^2_x}^2 + \|q\|_{L^2_x}^2).
\end{align}
\indent Next, using equation \eqref{equationofv} and integrating by parts, we have
\begin{align*}
\widetilde{\mathcal{IV}_{2}}
&=\textrm{Im}\int_{\mathbb{R}}J_{x}^{-2}\partial_{x}r\ \partial_{x}\bar{v}q \ dx+\textrm{Im}\int_{\mathbb{R}}J_{x}^{-2}r \ \partial_{x}\bar{v}\partial_{x}q \ dx-\textrm{Re}\int_{\mathbb{R}}J_{x}^{-2}r\ i(r_{1}q_{1}+r_{2}q_{2})q \ dx\\
&=\widetilde{\mathcal{IV}}_{2,1}+\widetilde{\mathcal{IV}}_{2,2}+\widetilde{\mathcal{IV}}_{2,3}.
\end{align*}
Thus, note that
\begin{align}
|\widetilde{\mathcal{IV}}_{2,1}|+|\widetilde{\mathcal{IV}}_{2,3}|&\lesssim (\|\partial_{x}v\|_{L^{\infty}_{x}}+\|r_{1}\|_{L^{\infty}_{x}}\|q_{1}\|_{L^{\infty}}+\|r_{2}\|_{L^{\infty}_{x}}\|q_{2}\|_{L^{\infty}_{x}})\|r\|_{L^{2}_{x}}\|q\|_{L^{2}_{x}}.
\end{align}
Concerning the term $\widetilde{\mathcal{IV}}_{2,2}$, we observe that it will be canceled by a term in $\widetilde{\mathcal{IV}}_{3}$.
In fact, using the second equation in \eqref{NLS-BO2} and integration by parts, we obtain that
\begin{align*}
\widetilde{\mathcal{IV}}_{3}&=\textrm{Im}\int_{\mathbb{R}}J_{x}^{-2}\partial_{x}^{2}r \ \bar{v}q \ dx+\textrm{Im}\int_{\mathbb{R}}J_{x}^{-2}\partial_{x}r \ \partial_{x}\bar{v} \ q \ dx\\
&-\textrm{Im}\int_{\mathbb{R}}J_{x}^{-2}r \ \partial_{x}\bar{v}\partial_{x}q \ dx+\frac{1}{2}\textrm{Re}\int_{\mathbb{R}}J_{x}^{-2}r \ \bar{v}\ i(rv+qu) \ dx\\
&=\widetilde{\mathcal{IV}}_{3,1}+\widetilde{\mathcal{IV}}_{3,2}+\widetilde{\mathcal{IV}}_{3,3}+\widetilde{\mathcal{IV}}_{3,4}.
\end{align*}
Thus, it follows that
\begin{align}
    |\widetilde{\mathcal{IV}}_{3,1}|+|\widetilde{\mathcal{IV}}_{3,2}|&\lesssim (\|v\|_{L^{\infty}_{x}}+\|\partial_{x}v\|_{L^{\infty}_{x}})\|r\|_{L^{2}_{x}}\|q\|_{L^{2}_{x}};\\
    |\widetilde{\mathcal{IV}}_{3,4}|&\lesssim \|v\|_{L^{\infty}_{x}}^{2}\|r\|_{L^{2}_{x}}^{2}+\|v\|_{L^{\infty}_{x}}\|u\|_{L^{\infty}_{x}}\|r\|_{L^{2}_{x}}\|q\|_{L^{2}_{x}}.
\end{align}
On the other hand, 
\begin{equation}\label{IV22~i}
    \widetilde{\mathcal{IV}}_{3,3}+ \widetilde{\mathcal{IV}}_{2,2}=0.
\end{equation}
Then, the estimate \eqref{esd3} follows from combining the estimates \eqref{boundedrq}, \eqref{difftoE0}-\eqref{IV22~i}, and the Sobolev embedding $H^{s-1}(\mathbb{R})\hookrightarrow L^{\infty}(\mathbb{R})$.
\end{proof}
\begin{remark}
The same arguments used in Propositions \ref{propenergyestimate} and \ref{propenergyestimatedifference} can be readily adapted to obtain analogous results for the higher-order Benjamin-Ono equation \eqref{higher-order-BO}.
\end{remark}

\subsection{Well-posedness for smooth initial data}

This subsection is devoted to establish our first local well-posedness result for the HBOS system in \eqref{BO-NLS}. The result is derived from the energy estimates proven in Propositions \ref{propenergyestimate} and \ref{propenergyestimatedifference}. It ensures, in particular, the existence of smooth solutions to the system \eqref{BO-NLS} for sufficiently regular initial data. Our approach is based on parabolic regularization combined with the Bona–Smith approximation method. Furthermore, it is crucial to remark that the existence result is a key ingredient in implementing the compactness argument associated with the gauge transformation.

\begin{theorem}\label{existenceofsmoothsolutions}
Let $s>\frac{5}{2}$. Then, for any $(r_{0},q_{0})\in H^{s}(\mathbb{R})\times H^{s}(\mathbb{R})$, there exist a positive time $T=T(\|(u_{0},q_{0})\|_{H^{s}\times H^{s}})$ and a unique maximal solution $(r,q)$ of the IVP \eqref{BO-NLS} such that
\begin{equation*}
    (r,q)\in C([0,T^{*}):H^{s}(\mathbb{R})\times H^{s}(\mathbb{R})),
\end{equation*}
with $T<T^{*}\leq \infty$. If the maximal time of existence $T^{*}$ is finite, then
\begin{equation*}
    \lim_{t\uparrow T^{*}}\|(r(t),q(t))\|_{H^{s}\times H^{s}}=+\infty.
\end{equation*}
\indent Moreover, for any $0<T'<T$, there exists a neighborhood $\mathcal{V}$ of $(r_{0},q_{0})\in H^{s}\times H^{s}$ such that the flow map data-to-solution 
\begin{align*}
    S: \mathcal{V}\longrightarrow C([0,T']:H^{s}(\mathbb{R})\times H^{s}(\mathbb{R})), \quad (\tilde{r}_{0},\tilde{q}_{0})\mapsto (\tilde{u},\tilde{q}),
\end{align*}
is continuous.
\end{theorem}
We observe that, due to the rescaling argument given in \eqref{rescaledsolutions}, it suffices to prove Theorem \ref{existenceofsmoothsolutions} in the case where the initial data has a sufficiently small norm, i.e., 
\begin{equation}\label{dadopequenosuave}
\|(r_0, q_0)\|_{H^s \times H^s} \leq \varepsilon \ll 1.
\end{equation}
In this setting, the existence and uniqueness of solutions follow by combining Propositions \ref{propenergyestimate} and \ref{propenergyestimatedifference}, along with their corresponding Remarks \ref{remarkenergyestimate1} and \ref{remarkenergyestimate2}, using a standard parabolic regularization argument (see, for instance, \cite{M-Paulsen-System} and \cite{Tanaka-ThirdOrderBenjaminOno}). In what follows, we will present only a sketch of the proof.
\newline
\indent In order to implement the regularization parabolic argument, for all $0<\delta<1$, we consider the following regularized initial value problem associated with \eqref{BO-NLS}
\begin{equation}\label{BO-NLS-regularized}
\left\{
\begin{array}{ll}
     & \partial_{t}r-a\partial_{x}^{3}r-b\mathcal{H}\partial_{x}^{2}r=cr\partial_{x}r-d\partial_{x}(r\mathcal{H}\partial_{x}r+\mathcal{H}(r\partial_{x}r))+\beta \partial_{x}(|q|^{2})\textcolor{red}{-\delta D_{x}^{\frac{5}{2}}r} , \\
     & i\partial_{t}q-\alpha \partial_{x}^{2}q=-\beta qr\textcolor{red}{-i\delta D_{x}^{\frac{5}{2}}q}, \\
     & r(x,0)=r_{0}(x), \ \ q(x,0)=q_{0}(x).
\end{array}
\right.
\end{equation}  
 where $ (x,t)\in \mathbb{R}\times \mathbb{R}_{+}$. Let $\{V_{\delta}(t)\}$ and $\{U_{\delta}(t)\}$ be the semi-groups associated, respectively, to the linear equations
 $$
 \partial_{t}r-a\partial_{x}^{3}r-b\mathcal{H}\partial_{x}^{2}r+\delta D_{x}^{\frac{5}{2}}r=0, \quad \textrm{and} \quad i\partial_{t}q-\alpha\partial_{x}^{2}q+\delta D_{x}^{\frac{5}{2}}q=0,
 $$
i.e., 
 \begin{equation}
     \mathcal{F}_{x}(V_{\delta}(t)f)(\xi)=e^{it(-b\xi|\xi|+a\xi^{3})-\delta t|\xi|^{\frac{5}{2}}}\mathcal{F}_{x}(f), \quad \textrm{and} \quad \mathcal{F}_{x}(U_{\delta}(t)f)(\xi)=e^{it\alpha\xi^{2}-\delta t|\xi|^{\frac{5}{2}}}\mathcal{F}_{x}(f), 
 \end{equation}
for all $\xi\in \mathbb{R}$ and $t\geq 0$. Thus, it is not difficult to see that $V_{\delta}(t)$ and $U_{\delta}(t)$ have the following smoothing effect (see \cite[Theorem 2.1]{IORIO-BENJAMIN-ONO})

\begin{lemma}\label{smoothingeffectgroup}
Let $0<\delta<1$ and $l\geq 0$. It holds that
\begin{align}
    \|V_{\delta}(t)f\|_{H^{l+2}}\leq C(l)\left[ 1+(\delta|t|)^{-\frac{4}{5}}\right]  \|f\|_{H^{l}} \quad \textrm{and} \quad \|U_{\delta}(t)\|_{H^{l+2}}\leq C(l)\left[ 1+(\delta|t|)^{-\frac{4}{5}}\right] \|f\|_{H^{l}},
\end{align}
for all $t>0$.
\end{lemma}
\noindent By using Lemma \ref{smoothingeffectgroup} and a standard argument based on the contraction principle, the following theorem can be obtained regarding the local well-posedness of the system \eqref{BO-NLS-regularized}
\begin{theorem}\label{thmparabolicregularization}
Let $s\geq 2$ and $0<\delta <1$. For any $(r_{0},q_{0})\in H^{s}(\mathbb{R})\times H^{s}(\mathbb{R})$ there exist $T_{\delta}\in (0,\infty]$ and a unique solution $(r_{\delta,}q_{\delta})\in C([0,T_{\delta}):H^{s}(\mathbb{R})\times H^{s}(\mathbb{R}))$ to the IVP \eqref{BO-NLS-regularized} such that $T_{\delta}<\infty$ implies that
\begin{equation}\label{explosionregularized}
  \lim_{t\rightarrow T_{\delta}} \|(r_{\delta}(t),q_{\delta}(t))\|_{H^{2}\times H^{2}}=\infty
\end{equation}
holds. Moreover, $(r_{\delta},q_{\delta})$ satisfies
\begin{equation}
(r_{\delta},q_{\delta})\in C((0,T_{\delta}):H^{\infty}(\mathbb{R})\times H^{\infty}(\mathbb{R})).
\end{equation}
\end{theorem}
\begin{remark}
We observe that, with slight modifications, Propositions \ref{propenergyestimate} and \ref{propenergyestimatedifference} also apply to the regularized system \eqref{BO-NLS-regularized}.
\end{remark}
We are now ready to present the proof of Theorem \ref{existenceofsmoothsolutions}.

\begin{proof}[Proof of Theorem \ref{existenceofsmoothsolutions}]

Let $(r_{0},q_{0})\in H^{s}(\mathbb{R})\times H^{s}(\mathbb{R})$ be an initial data satisfying \eqref{dadopequenosuave}. Thus, for any $0<\delta <1$, consider
$$
(r_{\delta},q_{\delta})\in C([0, T_{\delta}):H^{s}\times H^{s}),
$$
the solution of the IVP \eqref{BO-NLS-regularized} obtained by Theorem \ref{thmparabolicregularization}. Then, define the following set
\begin{equation}
    X_{\delta}:=\big\{t\in [0,T_{\delta}) \ : \ \|(r_{\delta}(t),q_{\delta}(t))\|_{H^{s}\times H^{s}}>8\|(r_{0},q_{0})\|_{H^{s}}\big\}.
\end{equation}
First, assume that $X_{\delta}$ is not empty for some $\delta \in (0,1)$. In this case, it follows that there is $T_{\delta}^{*}=\inf X_{\delta}$. Note that, by continuity, $0<T_{\delta}^{*}\leq T_{\delta}.$ On the other hand, for any $0<t<T_{\delta }^{*}$, by definition of $T_{\delta}^{*}$, we have that
\begin{equation*}
    \|(r_{\delta}(t),q_{\delta}(t))\|_{H^{s}\times H^{s}}\leq 8\|(r_{0},q_{0})\|_{H^{s}\times H^{s}}.
\end{equation*}
Thus, considering $T\in (0,T_{\delta}^*]$ arbitrary, as 
\begin{equation}
    \sup_{[0,T]}\|(r_{\delta}(t),q_{\delta}(t))\|_{H^{s}\times H^{s}}\leq 8\|(r_{0},q_{0})\|_{H^{s}\times H^{s}}\leq 8\varepsilon,
\end{equation}
then, taking $\varepsilon$ small enough, we can apply Proposition \ref{propenergyestimate} and Sobolev's embedding to obtain that
\begin{align}
    \sup_{ [0,T]}\|(r_{\delta}(t),q_{\delta}(t))\|_{H^{s}_{x}\times H^{s}_{x}}
        &\leq 2\textrm{exp}\left\{B_{s}T \left(1+16\|(r_{0},q_{0})\|_{H^{s}\times H^{s}}\right)\right\}\|(r_{0},q_{0})\|_{H^{s}\times H^{s}},
\end{align}
Therefore, by defining $ T:=\min\left\{T_{\delta}^{*}, \frac{1}{B_{s}(1+16\|(r_{0},q_{0})\|_{H^{s}\times H^{s}})}\right\}$, it follows that
\begin{equation}\label{supsolutionHs111}
    \sup_{[0,T]}\|(r_{\delta}(t),q_{\delta}(t))\|_{H^{s}\times H^{s}}\leq 2e\|(r_{0},q_{0})\|_{H^{s}\times H^{s}}.
\end{equation}
Then, from this and the definition of $T_{\delta}^{*}$, we have $T=\frac{1}{B_{s}(1+16\|(r_{0},q_{0})\|_{H^{s}\times H^{s}})}$, which is independent of $\delta$. Next, suppose that $X_{\delta}$ is empty. By property \eqref{explosionregularized}, it follows that $T_{\delta}=\infty $. Furthermore, note that the following estimate holds for $(r_{\delta},q_{\delta})$
\begin{equation}\label{supsolutionHs}
    \sup_{[0,T]}\|(r_{\delta}(t),q_{\delta}(t))\|_{H^{s}\times H^{s}}\leq 8\|(r_{0},q_{0})\|_{H^{s}\times H^{s}}.
\end{equation}
In fact, by \eqref{supsolutionHs111}, it follows that estimate \eqref{supsolutionHs} holds true for all $0<\delta<1$.

Now, since we have already obtained a uniform existence time \( T \) for the solutions \( (r_{\delta}, q_{\delta}) \), we can apply estimate \eqref{esd3} to \( (u,v)=(u_{\delta,\delta'},v_{\delta,\delta'}) := (r_{\delta} - r_{\delta'}, q_{\delta} - q_{\delta'} )\) (after taking \( \varepsilon \) sufficiently small, if necessary), together with estimate \eqref{supsolutionHs}, to deduce, on the interval \( [0,T] \), that
\begin{align*}
    \frac{d}{dt}\tilde{E}_{m}^{0}(u,v)(t)&\leq C\tilde{E}_{m}^{0}(u,v)(t)+|\delta-\delta'|\|D_{x}^{\frac{5}{2}}r_{\delta}\|_{L^{2}_{x}} +|\delta-\delta'|\|D_{x}^{\frac{5}{2}}q_{\delta}\|_{L^{2}_{x}}\\
    &\leq C\tilde{E}_{m}^{0}(u,v)(t)+C_{1}(|\delta-\delta'|+\delta'),
\end{align*}
where the last two terms in the first inequality arise from the additional terms present in the regularized system \eqref{BO-NLS-regularized}. Thus, estimate \eqref{esd1} and Grownwall's inequality yields that
\begin{equation}
    \sup_{[0,T]}\|(u(t),v(t))\|_{L^{2}_{x}\times L^{2}_{x}}\leq 2C_{1}|\delta-\delta'|Te^{CT}.
\end{equation}
Therefore, the sequence \( (u_{\delta, \delta'}, v_{\delta, \delta'}) \) converges to \( (r, q) \in C([0,T]; L^2(\mathbb{R}) \times L^2(\mathbb{R})) \) as \( \delta, \delta' \to 0 \). Next, given \( l \in [0,s) \), this convergence, together with estimate \eqref{supsolutionHs}, and an interpolation inequality between \( L^2(\mathbb{R}) \) and \( H^{s}(\mathbb{R}) \), ensures that \( (r, q) \in C([0,T]; H^{l}(\mathbb{R}) \times H^{l}(\mathbb{R})) \). Furthermore, it is not difficult to verify that $(r,q)$ satisfies \eqref{BO-NLS} on $[0,T]$.
\newline
\indent At this stage, the persistence of regularity and the continuous dependence on initial conditions can be established by employing the estimate in \eqref{esd2} along with the Bona–Smith approximation method (see \cite{IORIO-BENJAMIN-ONO}, \cite{M-Paulsen-System}, \cite{Bona-Smith-Kdv}, and \cite{Tanaka-ThirdOrderBenjaminOno}).
\end{proof}

\section{A priori estimates in $H^{s}(\mathbb{R})\times H^{s'}(\mathbb{R})$}\label{aprioriestimates1,2}

Given that Theorem \ref{existenceofsmoothsolutions} already establishes the existence of smooth solutions to the system \eqref{BO-NLS}, the aim of this section is to derive \textit{a priori} estimates for these solutions in the Sobolev space setting \( H^{s} \times H^{s'} \), where \( s \geq 1 \) and \( s\leq s'\leq s+ 1 \). In what follows, we shall employ the linear estimates established in Subsection 2.3 of \cite{HOBOinH1-Didier-Molinet} (specifically, Lemmas 2.5 and 2.6) to do this. Then, we emphasize that the proofs of these estimates do not impose any restrictions on the parameter~$a$ in the equation considered in this work. Hence, they remain valid for all~$a \neq 0$. In the next section, we shall show that the estimates derived here are sufficient to establish both local and global well-posedness, as stated in Theorems \ref{maintheorem} and \ref{globalwellposedness}.


\subsection{Gauge tranformation} \label{subsectionGT}

We will use the gauge transformation introduced by Tao in \cite{GWP-BO-Tao} and also employed by Molinet and Didier in \cite{HOBOinH1-Didier-Molinet}, which represents a nonlinear transformation on the first equation of \eqref{BO-NLS} to weaken the high-low frequency interaction in the nonlinearity. First, we will define an anti-derivative $F=F[r]$ of $r$ (i.e. $\partial_{x}F=r$), where $r$ satisfies the first equation in \eqref{BO-NLS}. So, we define, in terms of this function $F$, the gauge function which will be used in this work. For this, consider a smooth function $\psi\in C^{\infty}_{0}(\mathbb{R})$ such that $\int_{\mathbb{R}}\psi\  dy=1$. Thus, define 
\begin{equation}
    F(x,t):=A\int_{\mathbb{R}}\psi(y)\int_{y}^{x}r(z,t) \ dz dy+G(t),  \quad x,t\in \mathbb{R}.
\end{equation}
where $A>0$ will be chosen later. Then, from the first equation in \eqref{BO-NLS}, if we define $G=G(t)$ by
\begin{equation*}  G(t):=A\int_{0}^{t}\int_{\mathbb{R}}\psi'(y)\partial_{y}r(y,t')+\psi(y)\left[b\mathcal{H}\partial_{y}r+\frac{c}{2}r^{2}-d(r\mathcal{H}\partial_{y}r+\mathcal{H}(r\partial_{y}r))+\beta|q|^{2}\right](y,t') dydt',
\end{equation*}
we get that $F$ will satisfy the following equations
\begin{align}\label{eqofF}
\begin{cases}
&\partial_{t}F-b\mathcal{H}\partial_{x}^{2}F-a\partial_{x}^{3}F=\displaystyle A^{-1}\frac{c}{2}(\partial_{x}F)^{2}-A^{-1}d(\partial_{x}F\mathcal{H}\partial_{x}^{2}F+\mathcal{H}(\partial_{x}F\partial_{x}^{2}F))+A\beta (|q|^{2}),\\
& \partial_{x}F=Ar.
\end{cases}
\end{align}
\begin{remark}
The function $G=G(t)$ is well-defined since the functions $r$ and $q$ belong to the space $H^{1}(\mathbb{R})$.
\end{remark}
\noindent Next, we introduce the new unknown (Gauge function)
\begin{equation}\label{gaugefunction}
    W:=P_{+hi}(e^{iF}) \quad \textrm{and} \quad w:=\partial_{x}W=iAP_{+hi}(e^{iF}r).
\end{equation}
Then, choosing $A=\frac{2d}{3a}$ (this choice allows us to cancel terms), using the equations in \eqref{eqofF}, and the identities
\begin{equation*}
\mathcal{H}=-i+2iP_{-} \quad \textrm{and} \quad P_{+}\mathcal{H}=-iP_{+}, 
\end{equation*}
we can deduce that the function $w$ satisfies the equation
\begin{align}\label{eqofw}
\partial_{t}w+ib\partial_{x}^{2}w-a\partial_{x}^{3}w =&\partial_{x}P_{+hi}( e^{iF}(r^{2}+r^{3}))+\partial_{x}P_{+hi}(WP_{-}(\partial_{x}r))+\partial_{x}P_{+hi}(P_{lo}(e^{iF})P_{-}\partial_{x}r)
    \nonumber \\
    &+\partial_{x}P_{+hi}(wP_{-}\partial_{x}r)
    +\partial_{x}P_{+hi}(P_{lo}(e^{iF}r)P_{-}\partial_{x}r)+\partial_{x}P_{+hi}(WP_{-}(r\partial_{x}r))\nonumber \\
    &+\partial_{x}P_{+hi}(P_{lo}(e^{iF})P_{-}(r\partial_{x}r))
    +\partial_{x}P_{+hi}(e^{iF}|q|^{2})\nonumber \\
    :=&N_{1}(e^{iF},v,W,w)+N_{2}(e^{iF},q).
\end{align}
where, for simplicity, we assume all complex constants that may arise in the nonlinearity are equal to 1.
\newline
\indent Moreover, we can recover the function $r$ in terms of the function $w$ as follows
\begin{equation}\label{rintermsofw}
    iAr=e^{-iF}\partial_{x}(e^{iF})=e^{-iF}w+e^{-iF}\partial_{x}P_{lo}(e^{iF})+e^{-iF}\partial_{x}P_{-hi}(e^{iF}).
\end{equation}
In particular, applying the operator $P_{+HI}$ in the above equation and using frequency localization, we have
\begin{equation}\label{rintermofw2}
    iAP_{+HI}r=P_{+HI}(e^{-iF}w)+P_{+HI}(P_{+hi}e^{-iF}\partial_{x}P_{lo}(e^{iF}))+P_{+HI}(P_{+HI}e^{-iF}\partial_{x}P_{-hi}(e^{iF})).
\end{equation}

We also observe that, if $(r,q)$ is a solution to the system in \eqref{BO-NLS}, then using the first equation in \eqref{eqofF}, we can deduce that the function $e^{iF}q$ satisfies the following differential equation of the Schrödinger type
\begin{align}\label{equationexpq}
    i\partial_{t}(e^{iF}q)-\alpha \partial_{x}^{2}(e^{iF}q)&=-Ae^{iF}q\left[b\mathcal{H}\partial_{x}r+a\partial_{x}^{2}r+\displaystyle \frac{c}{2}r^{2}-d(r\mathcal{H}\partial_{x}r+\mathcal{H}(r\partial_{x}r))+\beta (|q|^{2})\right] \nonumber \\
    &\quad -\alpha \partial_{x}^{2}(e^{iF})q-2\alpha \partial_{x}(e^{iF})\partial_{x}q-\beta e^{iF}rq.
\end{align}
\indent Then, we have the following \textit{a priori} estimates for $r$ and $q$ (in the case of $r$ in terms of $w$).

\begin{prop}\label{estimatesaprioriofrq}
Let $s\geq 1,\ s\leq s'\leq s+1,\ 0<T\leq 1,\ 0\leq \theta \leq 1,$ and $(r,q)$ be a solution to \eqref{BO-NLS} in the time interval $[0,T]$. Then, it holds that
\begin{equation}\label{estimatebourgainr}
\|r\|_{X^{s-2\theta,\theta}_{T}}\lesssim T^{\frac{1}{2}}\|r\|_{L^{\infty}_{T}H^{s}_{x}}+T^{\frac{1}{2}}\|r\|_{L^{\infty}_{T}H^{s}_{x}}^{2}+\|J^{s}_{x}r\|_{L^{4}_{T,x}}^{2}+T^{\frac{1}{2}}\|q\|_{L^{\infty}_{T}H^{s}_{x}}^{2},
\end{equation}
\begin{equation}\label{estimatebourgainq}
    \|q\|_{Y^{s,1}_{T}}\lesssim \|q_{0}\|_{H^{s}}+T^{\frac{1}{2}}(\|r\|_{L^{\infty}_{T}H^{1}_{x}}\|q\|_{L^{\infty}_{T}H^{s}_{x}}+\|r\|_{L^{\infty}_{T}H^{s}_{x}}\|q\|_{L^{\infty}_{T}H^{1}_{x}}).
\end{equation}
Moreover, if $1\leq s\leq \frac{3}{2}$ and $(p,q)=(\infty,2)$ or $(4,4)$, it holds that
\begin{align}\label{estimatebourgainexpq}
\|e^{iF}q\|_{Y^{s-\theta,\theta}_{T}}
&\lesssim T^{\frac{1}{2}}(1+\|r\|_{L^{\infty}_{T}H^{1}_{x}}^{3}+\|q\|_{L^{\infty}_{T}H^{1}_{x}}^{3})\big[\|q\|_{L^{\infty}_{T}H^{s}_{x}}+\|r\|_{L^{\infty}_{T}H^{s}_{x}}\big]\nonumber \\
&\quad + (1+\|r\|_{L^{\infty}_{T}H^{1}_{x}}^{2})\big[(T^{\frac{1}{2}}\|r\|_{L^{\infty}_{T}H^{1}_{x}}+\|J^{1}_{x}\partial_{x}r\|_{\widetilde{L^{\infty}_{x}L^{2}_{T}}})(\|q\|_{L^{2}_{x}L^{\infty}_{T}}+\|D_{x}^{s-1}q\|_{\widetilde{L^{2}_{x}L^{\infty}_{T}}})\nonumber \\
&\quad +\|q\|_{L^{2}_{x}L^{\infty}_{T}}\|J^{s}_{x}\partial_{x}r\|_{\widetilde{L^{\infty}_{x}L^{2}_{T}}}   \big],
\end{align}
\begin{equation}\label{estimateofrinhs}
    \|J^{s}_{x}r\|_{L^{p}_{T}L^{q}_{x}}\lesssim \|r_{0}\|_{H^{1}} +\|q\|_{L^{\infty}_{T}H^{1}_{x}}^{2}+(1+\|r\|_{L^{\infty}_{T}H^{1}_{x}}^{2})(\|w\|_{X^{s,\frac{1}{2},1}_{T}}+\|r\|_{L^{\infty}_{T}H^{1}_{x}}^{2}),
\end{equation}
\begin{equation}\label{estimateofqHs'}
    \|J^{s'}_{x}q\|_{L^{\infty}_{T}L^{2}_{x}}\lesssim \|q_{0}\|_{H^{s'}}+\|r\|_{L^{\infty}_{T}H^{1}_{x}}\|q\|_{L^{\infty}_{T}H^{s'}_{x}}+\|q\|_{L^{2}_{x}L^{\infty}_{T}}\|J^{s}_{x}\partial_{x}r\|_{\widetilde{L^{\infty}_{x}L^{2}_{T}}},
\end{equation}
\begin{equation}\label{estimatel2linfq}
    \|D_{x}^{s-1}q\|_{L^{2}_{x}L^{\infty}_{T}}\lesssim \|q_{0}\|_{H^{s'}}+\|q\|_{L^{\infty}_{T}H^{1}_{x}}\|r\|_{L^{\infty}_{T}H^{s}_{x}}+\|q\|_{L^{\infty}_{T}H^{s'}_{x}}\|r\|_{L^{\infty}_{T}H^{1}_{x}}, 
\end{equation}
\begin{equation}\label{estimatel2linfr}
    \|r\|_{L^{2}_{x}L^{\infty}_{T}}\lesssim \|r_{0}\|_{H^{1}}+\|r\|_{L^{\infty}_{T}H^{1}_{x}}^{2}+\|q\|_{L^{\infty}_{T}H^{1}_{x}}^{2}+(1+\|r\|_{L^{\infty}_{T}H^{1}_{x}})\|w\|_{X^{1,\frac{1}{2},1}_{T}}+\|r\|_{L^{\infty}_{T}H^{1}_{x}}\|r\|_{L^{2}_{x}L^{\infty}_{T}},
\end{equation}
and
\begin{align}\label{estimatelinfl2r}
    \|J^{s}_{x}\partial_{x}r\|_{\widetilde{L^{\infty}_{x}L^{2}_{T}}} & \lesssim \|r_{0}\|_{H^{1}}+\|q\|_{L^{\infty}_{T}H^{1}_{x}}^{2}+\|w\|_{X^{s,\frac{1}{2},1}_{T}}+\|r\|_{L^{\infty}_{T}H^{1}_{x}}(\|J^{s}_{x}\partial_{x}r\|_{\widetilde{L^{\infty}_{x}L^{2}_{T}}}+\|r\|_{L^{\infty}_{T}H^{1}_{x}})\nonumber\\
    &+\|w\|_{X^{1,\frac{1}{2},1}_{T}}\|J^{s}_{x}\partial_{x}r\|_{\widetilde{L^{\infty}_{x}L^{2}_{T}}}.
\end{align}
\end{prop}
\begin{proof}
We begin with the estimate \eqref{estimatebourgainr}. To this end, we will construct a suitable extension $\tilde{r}$ de $r$ following the same procedure used in Proposition 3.2 in \cite{BOinL2-Didier}. Let $v$ be a function defined by $v(t):=V(-t)r(t)$ for all $t\in [0,T]$ and extend this function on $[-2T,2T]$ by setting $\partial_{t}v=0$ on $[-2T,2T]\setminus [0,T]$. And so, define $\tilde{r}$ by 
\begin{equation*}
    \tilde{r}(x,t)=\eta_{T}(t)V(t)v(t), \quad  x,t\in \mathbb{R},
\end{equation*}
where we define $\eta_{T}(t)=\eta\left(\frac{t}{T}\right)$. First, note that $\tilde{r}_{|[0,T]}=r.$ Furthermore, from the properties of Bourgain spaces and the previous construction, we obtain that
\begin{align}\label{eqbourgaintheta1}
    \|\tilde{r}\|_{X^{s-2,1}}
    \lesssim T^{\frac{1}{2}}\|v\|_{L^{\infty}_{T}H^{s-2}_{x}}+\|\partial_{t}v\|_{L^{2}_{T}H^{s-2}_{x}} \quad \textrm{and } \quad 
    \|\tilde{r}\|_{X^{s,0}}\lesssim T^{\frac{1}{2}}\|v\|_{L^{\infty}_{T}H^{s}_{x}}.
\end{align}
On the other hand, by interpolation in Bourgain spaces, we have
$
    \|\tilde{r}\|_{X^{1-2\theta,\theta}}\lesssim \|\tilde{r}\|_{X^{s-2,1}}+\|\tilde{r}\|_{X^{s,0}}.
$
Therefore, \eqref{estimatebourgainr} follows easily from \eqref{eqbourgaintheta1}, the fractional Leibniz rule and the Sobolev embedding.

Regarding the estimate \eqref{estimatebourgainq}, consider $\tilde{q}$ an extension of $q$ supported in $[-2T,2T]$ such that $\tilde{q}_{|[0,T]}=q$ and $\partial_{t}\tilde{q}=0$ in $[-2T,2T]\setminus [0,T]$. Thus, using the Duhamel formula satisfied by $q$ and the estimates \eqref{linearestimateNLS} and \eqref{nonlinearestimateNLS}, along with Lemma \ref{fracleib}, we can obtain
\begin{align}\label{estimateqbourgain1}
\|q\|_{Y^{s,1}_{T}}\lesssim \|\eta(t)\tilde{q}\|_{Y^{s,1}}&\lesssim \|q_{0}\|_{H^{s}}+\|\tilde{q}r\|_{Y^{s,0}}\nonumber \\
&=\|q_{0}\|_{H^{s}}+\|\tilde{q}r\|_{L^{2}_{t}H^{s}_{x}}\nonumber\\
&\lesssim \|q_{0}\|_{H^{s}_{x}}+T^{\frac{1}{2}}(\|q\|_{L^{\infty}_{T}H^{1}_{x}}\|r\|_{{L^{\infty}_{T}H^{s}_{x}}}+\|q\|_{L^{\infty}_{T}H^{s}_{x}}\|r\|_{{L^{\infty}_{T}H^{1}_{x}}}),
\end{align}
Therefore, from \eqref{estimateqbourgain1}, the estimate \eqref{estimatebourgainq} follows.

Now, we will apply the same arguments used in the proof of \eqref{estimatebourgainr} to prove \eqref{estimatebourgainexpq}. Then, let $u$ be an extension to $e^{iF}q$ on the interval $[0,T]$ constructed exactly as done for $r$ before, but using the group $\{U(t)\}$. Thus, since $e^{iF}q$ satisfies the equation \eqref{equationexpq}, using Lemma \ref{lemmaestimateexponencialinHs}, together with its version for the range $0\leq s\leq \frac{1}{2}$ (see Lemma 3.3 in \cite{ILWinL2-Chapouto}), it follows that (see \eqref{eqbourgaintheta1})
\begin{equation}\label{estimatebourgainexpq1}
    \|u\|_{Y^{s,0}}\lesssim T^{\frac{1}{2}}\|e^{iF}q\|_{L^{\infty}_{T}H^{s}_{x}}\lesssim T^{\frac{1}{2}}(1+\|r\|_{L^{\infty}_{T}H^{1}_{x}}^{2})\|q\|_{L^{\infty}_{T}H^{s}_{x}},
\end{equation}
and,
\begin{align}
&\|u\|_{Y^{s-1,1}}-T^{\frac{1}{2}}(1+\|r\|_{L^{2}_{T}H^{1}_{x}}^{2})\|q\|_{L^{\infty}_{T}H^{s}_{x}}\nonumber \\
&\lesssim \|Ae^{iF}q[b\mathcal{H}\partial_{x}r+a\partial_{x}^{2}r+\displaystyle \frac{c}{2}r^{2}-d(r\mathcal{H}\partial_{x}r+\mathcal{H}(r\partial_{x}r))+\beta (|q|^{2})]\|_{L^{2}_{T}H^{s-1}_{x}}\nonumber \\
&\quad + \|\alpha \partial_{x}^{2}(e^{iF})q+2\alpha \partial_{x}(e^{iF})\partial_{x}q +\beta e^{iF}rq\|_{L^{2}_{T}H^{s-1}_{x}} \nonumber \\
&\lesssim (1+\|r\|_{L^{\infty}_{T}L^{2}_{x}})\big(\|q\mathcal{H}\partial_{x}r\|_{L^{2}_{T}H^{s-1}_{x}}+\|q\partial_{x}^{2}r\|_{L^{2}_{T}H^{s-1}_{x}}+\|qr^{2}\|_{L^{2}_{T}H^{s-1}_{x}}+\|q(r\mathcal{H}\partial_{x}r+r\partial_{x}r)\|_{L^{2}_{T}H^{s-1}_{x}} \nonumber \\
& \quad + \|q|q|^{2}\|_{L^{2}_{T}H^{s-1}_{x}}\big)+\|\partial_{x}(e^{iF}r)q\|_{L^{2}_{T}H^{s-1}_{x}}+\|(e^{iF}r)\partial_{x}q\|_{L^{2}_{T}H^{s-1}_{x}}+\| e^{iF}rq\|_{L^{2}_{T}H^{s-1}_{x}}.
\end{align}
On the one hand, using that $H^{s}$ is a Banach algebra and the Lemma \ref{lemmaestimateexponencialinHs}, we have
\begin{align}
    \|qr^{2}\|_{L^{2}_{T}H^{s-1}_{x}}+\|q|q|^{2}\|_{L^{2}_{T}H^{s-1}_{x}}&\lesssim T^{\frac{1}{2}}(\|q\|_{L^{\infty}_{T}H^{1}_{x}}^{2}+\|r\|_{L^{\infty}_{T}H^{1}_{x}}^{2})(\|q\|_{L^{\infty}_{T}H^{s}_{x}}+\|r\|_{L^{\infty}_{T}H^{s}_{x}});\\
    \|e^{iF}rq\|_{L^{2}_{T}H^{s-1}_{x}}&\lesssim T^{\frac{1}{2}}(1+\|r\|_{L^{\infty}_{T}H^{1}_{x}}^{2})(\|q\|_{L^{\infty}_{T}H^{1}_{x}}\|r\|_{L^{\infty}_{T}H^{s}_{x}}+\|q\|_{L^{\infty}_{T}H^{s}_{x}}\|r\|_{L^{\infty}_{T}H^{1}_{x}}). 
\end{align}
On the other hand, as $s-1\geq 0$ an application of Hölder's inequality, Sobolev's embedding, and Lemma \ref{fracleib} (or the variant version given by Theorem 1 in \cite{Grafakos03062014}) yields that 
\begin{align}
    \|q\mathcal{H}\partial_{x}r\|_{L^{2}_{T}H^{s-1}_{x}}&\lesssim T^{\frac{1}{2}}(\|J^{s-1}_{x}q\|_{L^{\infty}_{T,x}}\|\mathcal{H}\partial_{x}r\|_{L^{\infty}_{T}L^{2}_{x}}+\|q\|_{L^{\infty}_{T,x}}\|J^{s-1}_{x}\mathcal{H}\partial_{x}r\|_{L^{\infty}_{T}L^{2}_{x}})\nonumber\\
    &\lesssim T^{\frac{1}{2}}(\|q\|_{L^{\infty}_{T}H^{s}_{x}}\|r\|_{L^{\infty}_{T}H^{1}_{x}}+\|q\|_{L^{\infty}_{T}H^{1}_{x}}\|r\|_{L^{\infty}_{T}H^{s}_{x}}).
\end{align}
By the same arguments, along with Lemma \ref{lemmaestimateexponencialinHs}, we have that
\begin{align}
\|q(r\mathcal{H}\partial_{x}r+\mathcal{H}(r\partial_{x}r))\|_{L^{2}_{T}H^{s-1}_{x}} &\lesssim T^{\frac{1}{2}}(\|q\|_{L^{\infty}_{T}H^{1}_{x}}^{2}+\|r\|_{L^{\infty}_{T}H^{1}_{x}}^{2})(\|q\|_{L^{\infty}_{T}H^{s}_{x}}+\|r\|_{L^{\infty}_{T}H^{s}_{x}});\\
\|\partial_{x}(e^{iF}r)q\|_{L^{2}_{T}H^{s-1}_{x}}+\|(e^{iF}r)\partial_{x}q\|_{L^{2}_{T}H^{s-1}_{x}}&\lesssim T^{\frac{1}{2}}(1+\|r\|_{L^{\infty}_{T}H^{1}_{x}}^{2})(\|r\|_{L^{\infty}_{T}H^{1}_{x}}\|q\|_{L^{\infty}_{T}H^{s}}+\|r\|_{L^{\infty}_{T}H^{s}_{x}}\|q\|_{L^{\infty}_{T}H^{1}_{x}}).
\end{align}
Therefore, it suffices to estimate the more complicated term $\|q\partial_{x}^{2}r\|_{L^{2}_{T}H^{s-1}_{x}}.$ But, from Bernstein's inequality, note that
\begin{align}\label{estimatebourgainexpq2}
    \|q\partial_{x}^{2}r\|_{L^{2}_{T}H^{s-1}_{x}}&\lesssim \|P_{LO}(q\partial_{x}^{2}r)\|_{L^{2}_{T}H^{s-1}_{x}}+\|P_{HI}(q\partial_{x}^{2}r)\|_{L^{2}_{T}H^{s-1}_{x}}\nonumber\\
    &\lesssim \|q\|_{L^{2}_{x}L^{\infty}_{T}}\|\partial_{x}^{2}r\|_{L^{\infty}_{x}L^{2}_{T}} +\|P_{HI}(q\partial_{x}^{2}r)\|_{L^{2}_{T}H^{s-1}_{x}}\nonumber \\
    &\lesssim \|q\|_{L^{2}_{x}L^{\infty}_{T}}(T^{\frac{1}{2}}\|r\|_{L^{\infty}_{T}H^{1}_{x}}+\|J^{1}_{x}\partial_{x}r\|_{\widetilde{L^{\infty}_{x}L^{2}_{T}}})+\|P_{HI}(q\partial_{x}^{2}r)\|_{L^{2}_{T}H^{s-1}_{x}}.
\end{align}
On the other hand, by using the frequency localization due to the projections, we can rewrite for any $l\geq 4$
\begin{equation}\label{frequencylocproduct}
    P_{2^{l}}(fg)=P_{2^{l}}\left(\sum_{j\geq l-3}P_{2^{j}}f\ P_{\leq 2^{j}}g\right)+P_{2^{l}}\left(\sum_{j\geq l-3}P_{\leq 2^{j-1}}f\ P_{2^{j}}g\right).
\end{equation}
Thus, using the Hölder and Bernstein inequalities, we have for all $l\geq 4$
\begin{align*}
&\|P_{2^{l}}(q\partial_{x}^{2}r)\|_{L^{2}_{T}H^{s-1}_{x}}\nonumber \\
&\lesssim 2^{l(s-1)}\sum_{j\geq l-3}\big \|P_{2^{j}}\partial_{x}^{2}r\|_{L^{\infty}_{x}L^{2}_{T}}\|P_{\leq 2^{j}}q\|_{L^{2}_{x}L^{\infty}_{T}}+2^{l(s-1)}\sum_{j\geq l-3}\big \|P_{\leq 2^{j-1}}\partial_{x}^{2}r\|_{L^{\infty}_{x}L^{2}_{T}}\|P_{2^{j}}q\|_{L^{2}_{x}L^{\infty}_{T}} \nonumber\\
&\lesssim \|q\|_{L^{2}_{x}L^{\infty}_{T}}\sum_{j\geq l-3} 2^{(l-j)(s-1)} \|P_{2^{j}}J_{x}^{s}\partial_{x}r\|_{L^{\infty}_{x}L^{2}_{T}}+\|\partial_{x}^{2}r\|_{L^{\infty}_{x}L^{2}_{T}}\sum_{j\geq l-3}2^{(l-j)(s-1)} \|D_{x}^{s-1}P_{2^{j}}q\|_{L^{2}_{x}L^{\infty}_{T}}.
\end{align*}
Now, observe that the expressions in the last inequality can be interpreted as discrete convolutions between the sequence $(2^{j})_{j \leq -3} \in \ell^{1}(\mathbb{Z})$ and the sequences of norms $$\left( \|P_{2^{j}} J_{x}^{s} \partial_{x} r \|_{L^{\infty}_{x} L^{2}_{T}} \right)_{j \geq 1}\quad \textrm{ and} \quad \left( \| D_{x}^{s-1} P_{2^{j}} q \|_{L^{2}_{x} L^{\infty}_{T}} \right)_{j \geq 1}.$$ Therefore, Young's inequality and Plancherel's identity yield
\begin{equation}\label{estimatebourgainexpq3}
    \|P_{HI}(q\partial_{x}^{2}r)\|_{L^{2}_{T}H^{s-1}_{x}}\lesssim \|q\|_{L^{2}_{x}L^{\infty}_{T}}\|J_{x}^{s}\partial_{x}r\|_{\widetilde{L^{\infty}_{x}L^{2}_{T}}}+(T^{\frac{1}{2}}\|r\|_{L^{\infty}_{T}H^{1}_{x}}+\|J^{1}_{x}\partial_{x}r\|_{\widetilde{L^{\infty}_{x}L^{2}_{T}}})\|D_{x}^{s-1}q\|_{\widetilde{L^{2}_{x}L^{\infty}_{T}}}.
\end{equation}
Thus, gathering the estimates \eqref{estimatebourgainexpq1}-\eqref{estimatebourgainexpq2} and \eqref{estimatebourgainexpq3} and using interpolation, follow the estimate \eqref{estimatebourgainexpq}.

Next, we turn to the proof of \eqref{estimateofrinhs}. First, since $r$ is a real valued function, then $P_{-}r=\overline{P_{+}r}$, and therefore, for $(p,q)=(2,\infty)$ or $(4,4)$, we have
\begin{equation}\label{eqlpqr}
    \|J_{x}^{s}r\|_{L^{p}_{T}L^{q}_{x}}\lesssim \|P_{LO}r\|_{L^{p}_{T}L^{q}_{x}}+\|D_{x}^{s}P_{+HI}r\|_{L^{p}_{T}L^{q}_{x}}.
\end{equation}
On the one hand, since $P_{+HI}r$ satisfies the identity \eqref{rintermofw2}, then by Proposition $2.7$ in \cite{HOBOinH1-Didier-Molinet}, we can obtain the following
\begin{equation}\label{eqhighrpq}
\|P_{+HI}D_{x}^{s}r\|_{L^{p}_{T}L^{q}_{x}}\lesssim \|r_{0}\|_{H^{1}}+(1+\|r\|_{L^{\infty}_{T}H^{1}_{x}})(\|w\|_{X_{T}^{s,\frac{1}{2},1}}+\|r\|_{L^{\infty}_{T}H^{1}_{x}}^{2}).
\end{equation}
Regarding the low frequencies, we use the Sobolev embedding $\dot{H}^{\frac{1}{4}}(\mathbb{R})\hookrightarrow L^{\frac{1}{4}}(\mathbb{R})$ to obtain
\begin{equation*}
    \|P_{LO}r\|_{L^{p}_{T}L^{q}_{x}}\lesssim T^{\frac{1}{p}}\|P_{LO}r\|_{L^{\infty}_{T}L^{2}_{x}}.
\end{equation*}
Furthermore, since $r$ is a solution of \eqref{BO-NLS}, $P_{LO}r$ satisfies the following integral equation.  
\begin{equation*}
    P_{LO}r=V(t)r_{0}+\int_{0}^{t}V(t-t')P_{LO}[cr\partial_{x}r-d\partial_{x}(r\mathcal{H}\partial_{x}r\mathcal{H}(r\partial_{x}r))+\beta\partial_{x}(|q|^{2})](t')dt'.
\end{equation*}
Thus, by Minkowski's inequality, Hölder inequality, and the Sobolev embedding $H^{1}(\mathbb{R})\hookrightarrow L^{\infty}(\mathbb{R})$, it follows that
\begin{align}\label{eqlowrpq}
\|P_{LO}r\|_{L^{p}_{T}L^{q}_{x}}
&\lesssim \|r_{0}\|_{L^{2}}+\int_{0}^{T}\|P_{LO}[cr\partial_{x}r-d\partial_{x}(r\mathcal{H}\partial_{x}r+\mathcal{H}(r\partial_{x}r))+\beta\partial_{x}(|q|^{2})](t')\|_{L^{2}_{x}}dt'\nonumber \\
&\lesssim\|r_{0}\|_{H^{1}}+\|r\|_{L^{\infty}_{T,x}}\|\partial_{x}r\|_{L^{\infty}_{T}L^{2}_{x}}+\|r\|_{L^{\infty}_{T,x}}\|\mathcal{H}\partial_{x}r\|_{L^{\infty}_{x}L^{2}_{x}}+\|q\|_{L^{\infty}_{T,x}}\|q\|_{L^{\infty}_{T}L^{2}_{x}}\nonumber \\
&\lesssim \|r_{0}\|_{H^{1}_{x}}+\|r\|_{L^{\infty}_{T}H^{1}_{x}}^{2}+\|q\|_{L^{\infty}_{T}H^{1}}^{2}.
\end{align}
From \eqref{eqlpqr}, \eqref{eqhighrpq}, and \eqref{eqlowrpq} follows the estimate \eqref{estimateofrinhs}.

To estimate the norm $\|q\|_{L^{\infty}_{T}H^{s'}_{x}}$, we first need to consider the integral equation satisfied by $q$, i.e.
\begin{equation}
    q(t)=U(t)q_{0}+i\beta\int_{0}^{t}U(t-t')(qr)(t')dt'.
\end{equation}
Thus, by Hölder's inequality, frequency localization, and Sobolev embedding, it follows that
\begin{align}\label{estimateqhsk0}
    \|q\|_{L^{\infty}_{T}H^{s'}_{x}}&\lesssim \|q_{0}\|_{H^{s'}}+\int_{0}^{T}\|(qr)(t')\|_{H^{s'}}dt'\nonumber \\
    &\lesssim \|q_{0}\|_{H^{s'}}+T^{\frac{1}{2}}\|qr\|_{L^{2}_{T}H^{s'}_{x}}\nonumber \\
    &\lesssim \|q_{0}\|_{H^{s'}}+\|P_{LO}(qr)\|_{L^{2}_{T}H^{s'}_{x}}+\|P_{HI}(qr)\|_{L^{2}_{T}H^{s'}_{x}}\nonumber \\
    &\lesssim \|q_{0}\|_{H^{s'}}+\|q\|_{L^{\infty}_{T}L^{2}_{x}}\|r\|_{L^{\infty}_{T}H^{1}_{x}}+\|P_{HI}(qr)\|_{L^{2}_{T}H^{s'}_{x}}.
\end{align}
Therefore, we need only estimate the norm $\|P_{HI}(qr)\|_{L^{2}_{T}H^{s'}_{x}}$, which cannot be directly estimated for $s<s'\leq s+1$ because $r$ has a lower regularity than $q$. But, note that the identity \eqref{frequencylocproduct}, the Sobolev embedding and the fact that $s'\leq s+1$ imply that
\begin{align}\label{estimatesobolevq1}
    &\|P_{2^{l}}(qr)\|_{L^{2}_{T}H^{s'}_{x}}\nonumber\\
    &\lesssim  2^{ls'}\left(\sum_{j\geq l-3}\|P_{2^{j}}q\|_{L^{2}_{T,x}}\|P_{\leq 2^{j}}r\|_{L^{\infty}_{T,x}}+\sum_{j\geq l-3}\|P_{\leq 2^{j-1}}q\|_{L^{2}_{x}L^{\infty}_{T}}\|P_{2^{j}}r\|_{L^{\infty}_{x}L^{2}_{T}}\right)\nonumber \\
    & \lesssim  \|r\|_{L^{\infty}_{T}H^{1}_{x}}\sum_{j\geq l-3}2^{(l-j)s'}\|P_{2^{j}}J_{x}^{s'}q\|_{L^{2}_{T,x}}+\|q\|_{L^{2}_{x}L^{\infty}_{x}}\sum_{j\geq l-3}2^{(l-j)s'}\|P_{2^{j}}J_{x}^{s}\partial_{x}r\|_{L^{\infty}_{x}L^{2}_{T}}. 
\end{align}
Thus, we can apply Young's inequality and Plancherel's identity in \eqref{estimatesobolevq1} to obtain that
\begin{equation}\label{estimatesobolevq2}
\|P_{HI}(qr)\|_{L^{2}_{T}H^{s+k}_{x}}\lesssim \|r\|_{L^{\infty}_{T}H^{1}_{x}}\|q\|_{L^{2}_{T}H^{s'}_{x}}+\|q\|_{L^{2}_{x}L^{\infty}_{T}}\|J^{s}_{x}\partial_{x}r\|_{\widetilde{L^{\infty}_{x}L^{2}_{T}}}.
\end{equation}
Thus, the estimate \eqref{estimateofqHs'} follows from \eqref{estimateqhsk0} and \eqref{estimatesobolevq2}.

Now, concerning the solutions of the linear Schrödinger equation, we have an estimate for the maximal operator $ U^{*}(f):=\sup_{[0,T]} |U(t)f|$ in Sobolev spaces, as stated in Lemma \ref{maximalestimateNLS}. Thus, it is possible for us to handle the norm \(\|D^{s-1}_{x}q\|_{\widetilde{L^2_x L^\infty_T}}\). In fact, from \eqref{eqmaximalestimateNLS}, the Sobolev embedding, along with Duhamel's formula satisfied by $q$, we have that for all $N$ dyadic 
\begin{align}\label{estimatesobolevq3}
    \|D^{s-1}_{x}P_{N}q\|_{L^{2}_{x}L^{\infty}_{T}}&\lesssim \|D_{x}^{s-1}P_{N}U(t)q_{0}\|_{L^{2}_{x}L^{\infty}_{T}}+\int_{0}^{T}\|D_{x}^{s-1}P_{N}U(t-t')(qr)(t')\|_{L^{2}_{x}L^{\infty}_{T}}dt'\nonumber \\
    &\lesssim \|P_{N}D_{x}^{s-1}q_{0}\|_{H^{\frac{1}{2}+}}+\int_{0}^{T}\|D_{x}^{s-1}P_{N}(qr)(t')\|_{H^{\frac{1}{2}+}_{x}}dt'\nonumber \\
    &\lesssim \|P_{N}q_{0}\|_{H^{s'}}+\|P_{N}(qr)\|_{L^{2}_{T}H^{s}_{x}}
\end{align}
Therefore, taking the square in the estimate \eqref{estimatesobolevq3}, summing over $N$, and using Lemma \ref{fracleib}, the estimate \eqref{estimatel2linfq} follows.

Regarding the estimates \eqref{estimatel2linfr} and \eqref{estimatelinfl2r}, we first look at the low frequencies. Therefore, the Sobolev embedding, the maximal estimate (2.28) in \cite{HOBOinH1-Didier-Molinet}, and similar arguments to those applied in \eqref{eqlowrpq} yield
\begin{align}\label{plowl2linfr}
\|P_{LO}r\|_{L^{2}_{x}L^{\infty}_{T}}+\|P_{LO}J_{x}^{s}\partial_{x}r\|_{\widetilde{L^{\infty}_{x}L^{2}_{T}}}&\lesssim \|r_{0}\|_{H^{1}}+\|r\|_{L^{\infty}_{T}H^{1}_{x}}^{2}+\|q\|_{L^{\infty}_{T}H^{1}_{x}}^{2}.
\end{align}
Finally, using the identity \eqref{rintermofw2}, the norms $\|P_{+HI}r\|_{L^{2}_{x}L^{\infty}_{T}}$ and $\|P_{+HI}J_{x}^{s}\partial_{x}r\|_{\widetilde{L^{\infty}_{x}L^{2}_{T}}}$ can be estimated exactly as done in Proposition 3.2 in \cite{HOBOinH1-Didier-Molinet}, i.e.
\begin{equation}\label{P+l2linfr}
    \|P_{+HI}r\|_{L^{2}_{x}L^{\infty}_{T}}\lesssim (1+\|r\|_{L^{\infty}_{T}H^{1}_{x}})\|w\|_{X^{1,\frac{1}{2},1}_{T}}+\|r\|_{L^{\infty}_{T}H^{1}_{x}}\|r\|_{L^{2}_{x}L^{\infty}_{T}},
\end{equation}
and
\begin{equation}\label{p+linfl2r}
    \|P_{+HI}J_{x}^{s}\partial_{x}r\|_{\widetilde{L^{\infty}_{x}L^{2}_{T}}}\lesssim  \|w\|_{X^{s,\frac{1}{2},1}_{T}}+\|w\|_{X^{1,\frac{1}{2},1}_{T}}\|J_{x}^{s}\partial_{x}r\|_{\widetilde{L^{\infty}_{x}L^{2}_{T}}}+\|r\|_{L^{\infty}_{T}H^{1}_{x}} \|J_{x}^{s}\partial_{x}r\|_{\widetilde{L^{\infty}_{x}L^{2}_{T}}},
\end{equation}
Note that, to justify the preceding step, we rely on Lemmas 2.5 and 2.6 from \cite{HOBOinH1-Didier-Molinet}. Thus, the estimates \eqref{plowl2linfr}, \eqref{P+l2linfr}, and \eqref{p+linfl2r} imply \eqref{estimatel2linfr} and \eqref{estimatelinfl2r}. It completes the proof of the proposition.
\end{proof}

\subsection{Bilinear estimates}\label{subsectionbilinearestimates}

The objective of this subsection is to derive an estimate for \( \|w\|_{X^{s,\frac{1}{2},1}} \). To achieve this, we will utilize the equation \eqref{eqofw} satisfied by \( w \). Specifically, observe that the expression in \eqref{eqofw} corresponds to the same differential equation derived in (3.4) of \cite{HOBOinH1-Didier-Molinet} for \( w \), up to the term
\[
\partial_{x} P_{+hi} \left( e^{iF} |q|^{2} \right),
\]
which allows us to take advantage of the bilinear estimates already established for similar terms. In fact, we have the following proposition
\begin{prop}\label{bilinearestimate}
Let $0<T\leq 1$, $1\leq s\leq \frac{3}{2}$, $s\leq s'\leq s+1$, and $(r,q)$ be a solution to \eqref{BO-NLS} on the time interval $[0,T]$ and $w$ defined in \eqref{gaugefunction}. Then it holds that
\begin{align}\label{estimatebourgainw}
    &\|w\|_{X^{s,\frac{1}{2},1}_{T}}\lesssim (1+\|r_{0}\|_{H^{1}}^{2})\|r_{0}\|_{H^{s}}+\|w\|_{X^{s,\frac{1}{2},1}}\left(\sup_{0\leq \theta \leq 1}\|r\|_{X^{1-2\theta,\theta}}+\|r\|_{L^{4}_{T}W^{1,4}_{x}}+\|r\|_{L^{\infty}_{T}H^{1}_{x}}\|r\|_{L^{4}_{T}W^{1,4}_{x}}\right)\nonumber \\
    &+\|r\|_{L^{4}_{T,x}}^{2}+(1+\|r\|_{L^{\infty}_{T}H^{1}_{x}}^{2})(1+\|r\|_{L^{\infty}_{T}H^{1}_{x}})\left(\|r\|_{L^{\infty}_{T}H^{1}_{x}}^{2}\|r\|_{L^{\infty}_{T}H^{s}_{x}}+\|r\|_{L^\infty_xH^1_x}^{2}+\|r\|_{L^{2}_{x}L^{\infty}_{T}}\|J_{x}^{s}\partial_{x}r\|_{\widetilde{L^{\infty}_{x}L^{2}_{T}}}\right)\nonumber\\
    &+\|q\|_{L^{\infty}_{T}H^{1}_{x}}^{2}+\sup_{0\leq \theta \leq 1}\|e^{iF}q\|_{Y_{T}^{1-\theta,\theta}}\|q\|_{Y_{T}^{s,1}}+\sup_{0\leq \theta \leq 1}\|e^{iF}q\|_{Y_{T}^{s-\theta,\theta}}\|q\|_{Y_{T}^{1,1}}.
\end{align}
\end{prop}
The proof of Proposition \ref{bilinearestimate} relies on the following key bilinear estimates 
\begin{spacing}{1.2}
\begin{prop}\label{propetesimativabilinearcrucial}
For all $s\geq 1$ and $-\frac{1}{2}<\gamma<-\frac{11}{24}$, the following estimate holds true
\begin{align}
\|\partial_{x}P_{+HI}(u\bar{v})\|_{X^{s,\gamma}}&\lesssim\sup_{0\leq \theta \leq 1}\|u\|_{Y^{s-\theta,\theta}}\sup_{0\leq \theta \leq 1}\|v\|_{Y^{1-\theta,\theta}}+\sup_{0\leq \theta \leq 1}\|u\|_{Y^{1-\theta,\theta}}\sup_{0\leq \theta \leq 1}\|v\|_{Y^{s-\theta,\theta}} \nonumber\\
&\quad + \|u\|_{Y^{1,0}}\|v\|_{Y^{s,\frac{1}{4}+}}. \label{esimativabilinearcrucial2}
\end{align}
\end{prop}
\vspace{-3em}
\end{spacing}
\begin{proof}
We will only prove the case where $a<0$ in the system \eqref{BO-NLS}. The other case follows from similar arguments. First, without loss of generality, let us consider $-a=b=\alpha=1$ in \eqref{BO-NLS}. Now, observe that the following identities hold true 
\begin{equation*}
    P_{+HI}(P_{lo}fP_{lo}g)=P_{+HI}(P_{lo}fP_{-hi}g)=0.
\end{equation*}
Therefore, we need only to prove the estimate \eqref{esimativabilinearcrucial2} for the following frequency interactions
\begin{equation}\label{frequencyinteractionAB}
    \|\partial_{x}P_{+HI}(P_{A}uP_{B}\bar{v})\|_{X^{s,\gamma}},
\end{equation}
where $(A,B)\in \{(lo,+hi),(+hi,lo),(-hi,+hi),(+hi,-hi),(+hi,+hi)\}.$ By duality and the identity $$\|\bar{f}\|_{Y^{s,l}_{\tau=-\xi^{2}}}=\|f\|_{Y_{\tau=\xi^{2}}^{s,l}}=\|f\|_{Y^{s,l}},$$ to estimate these terms, it is sufficient to prove that
\begin{align*}
  |I|&\lesssim \|h\|_{L^{2}_{\xi,\tau}}\big(\sup_{0\leq \theta \leq 1}\|u\|_{Y^{s-\theta,\theta}}\sup_{0\leq \theta \leq 1}\|v\|_{Y_{\tau=-\xi^{2}}^{1-\theta,\theta}}+\sup_{0\leq \theta \leq 1}\|u\|_{Y^{1-\theta,\theta}}\sup_{0\leq \theta \leq 1}\|v\|_{Y_{\tau=-\xi^{2}}^{s-\theta,\theta}}\\
  &\quad +\|u\|_{Y^{1,0}}\|v\|_{Y^{s,\frac{1}{4}+}_{\tau=-\xi^{2}}}\big),
\end{align*}
 where
\begin{align}
    I&=\int_{\mathcal{D}}\langle \xi \rangle^{s}\langle \sigma \rangle^{\gamma}\xi h(\xi,\tau)\widehat{u}(\xi_{1},\tau_{1})\widehat{v}(\xi_{2},\tau_{2})d\mu, \label{integralI}\\
    &d\mu =d\xi_{1} d\tau_{1}  d \xi d\tau, \quad \xi=\xi_{1}+\xi_{2},\quad \tau=\tau_{1}+\tau_{2}, \label{eqinxitau} \\
    &\sigma= \tau -|\xi|\xi -\xi^{3}, \quad \sigma_{1}=\tau_{1}- \xi_{1}^{2}, \quad \sigma_{2}=\tau_{2}+ \xi_{2}^{2}, \label{symbolssigma}
\end{align}
and the set $\mathcal{D}$ depends on the frequency interactions present in each case. Concerning the first term given by $(A,B)=(lo,+hi)$ in \eqref{frequencyinteractionAB}, the set $\mathcal{D}$ is given by
\begin{equation}\label{conjDPloP+hi}
    \mathcal{D}=\{(\xi,\tau,\xi_{1},\tau_{1})\in \mathbb{R}^{4}\ \ | \ \xi\geq 8, \ |\xi_{1}|\leq 2,\ \xi_{2}\geq 1\}.
\end{equation}
Then, observe that by \eqref{eqinxitau} and \eqref{conjDPloP+hi}, we always have in $\mathcal{D}$
\begin{equation}\label{desigbilinearlo+hi}
    \xi\leq 3\xi_{2} \quad \mathrm{and} \quad \xi_{2}\leq \frac{5}{4}\xi,
\end{equation}
i.e., we have always $\xi \sim \xi_{2}$ in $\mathcal{D}$. Furthermore, from \eqref{eqinxitau}, \eqref{symbolssigma}, and \eqref{conjDPloP+hi}, the following resonance identity holds in $\mathcal{D}$
\begin{equation}\label{resonanceidentitycucial}
    \sigma-\sigma_{1}-\sigma_{2}=-\xi^{3}-2\xi\xi_{2}.
\end{equation}
Therefore, the resonance identity \eqref{resonanceidentitycucial}, along with \eqref{conjDPloP+hi} and \eqref{desigbilinearlo+hi}, yields that
\begin{equation*}
    |\sigma-\sigma_{1}-\sigma_{2}|\geq \xi^{3}-2\xi\xi_{2}\gtrsim \xi^{3},
\end{equation*}
and,
\begin{equation*}
    |\sigma-\sigma_{1}-\sigma_{2}|\lesssim \xi^{3}.
\end{equation*}
Thus, defining $\sigma_{max},\ \sigma_{med}$ and $\sigma_{min}$ as the minimum, median, and maximum of $|\sigma|,|\sigma_{1}|$ and $|\sigma_{2}|$, respectively, we can always assume that
\begin{equation}\label{estLmaxPloP+hi}
    \sigma_{max}\sim \max\{\sigma_{med},\xi^{3}\} .
\end{equation}
Therefore, we will estimate $I$ in the cases where $\sigma_{max}= |\sigma|,\ |\sigma_{1}|$ or $|\sigma_{2}|$.
\newline
\indent First, assume the case $|\sigma_{2}|=\sigma_{max}$. In this case, using that $\xi\sim \xi_{2}$, it follows that
\begin{align*}
    |I|&\lesssim  \int_{\mathcal{D}}\xi^{s+1} \langle \sigma\rangle^{\gamma} |h(\xi,\tau)||\widehat{u}(\xi_{1},\tau_{1})||\widehat{v}(\xi_{2},\tau_{2})|d\mu \\
    &\lesssim \int_{\mathcal{D}}  \frac{ \xi^{s+1} }{\xi^{s-\frac{1}{2}}\langle \sigma_{2} \rangle^{\frac{1}{2}} }\langle \sigma \rangle^{\gamma}|h(\xi,\tau)||\widehat{u}(\xi_{1},\tau_{1})||\langle \sigma_{2} \rangle^{\frac{1}{2}}(J^{s-\frac{1}{2}}_{x}v)^{\wedge}(\xi_{2},\tau_{2})|d\mu.
\end{align*}
 On the other hand, by \eqref{estLmaxPloP+hi} we have that $\langle \sigma_{2} \rangle \gtrsim \xi^{3}$. Thus, $\frac{ \xi^{s+1} }{\xi^{s-\frac{1}{2}}\langle \sigma_{2} \rangle^{\frac{1}{2}} }\lesssim 1$. Therefore, by Hölder's inequality, the immersions $X^{0,\frac{4}{9}+}\hookrightarrow L^{6}_{x,t}$ and $Y^{0,\frac{1}{4}+}\hookrightarrow L^{3}_{x,t}$, ensured by Lemmas \ref{lemmaMergulhoLpBpurgain} and \ref{lemmaimersionBourgainNLS}, along with Plancherel's identity, it follows that
\begin{align*}
|I|&\lesssim\left\|\left(\langle \sigma \rangle^{\gamma}|h|\right)^{\vee}\right\|_{L^{6}_{x,t}}\|(|\widehat{u}|)^{\vee}\|_{L^{3}_{x,t}}\|\langle \sigma_{2}\rangle^{\frac{1}{2}}(J^{s-\frac{1}{2}}_{x}v)^{\wedge}\|_{L^{2}_{\xi_{2},\tau_{2}}}\\
&\lesssim \|\langle \sigma \rangle^{\gamma+\frac{4}{9}+} h\|_{L^{2}_{\xi,\tau}}\|u\|_{Y^{0,\frac{1}{4}+}}\|v\|_{Y_{\tau=-\xi^{2}}^{s-\frac{1}{2},\frac{1}{2}}}\\
&\lesssim \|h\|_{L^{2}_{\xi,\tau}}\|u\|_{Y^{1-\left(\frac{1}{4}+\right),\frac{1}{4}+}}\|v\|_{Y_{\tau=-\xi^{2}}^{s-\frac{1}{2},\frac{1}{2}}},
\end{align*}
where we have used that $-\frac{1}{2}<\gamma<-\frac{4}{9}$.
\newline 
\indent In the case that $|\sigma_{1}|=\sigma_{max}$, the integral $I$ can be estimated using the same arguments applied to the previous case. In fact, we can obtain
\begin{equation*}
   |I|\lesssim \|h\|_{L^{2}_{\xi,\tau}}\|u\|_{Y^{\frac{1}{2},\frac{1}{2}}}\|v\|_{Y_{\tau=-\xi^{2}}^{s-\left(\frac{1}{4}+\right),\frac{1}{4}+}}.
\end{equation*}
\indent Now, assume that $|\sigma|=\sigma_{max}$. In this case, according to the relation \eqref{estLmaxPloP+hi}, we may assume that \(|\sigma| \sim \xi^3\). Indeed, if this does not hold, then it follows that \(|\sigma| \sim \sigma_{\mathrm{med}}\), where \(\sigma_{\mathrm{med}} = |\sigma_1|\) or \(|\sigma_2|\). Therefore, we are reduced to the previous cases. Thus, note that Holder's inequality, the immersions $Y^{0,\frac{3}{8}+}_{\tau=\pm \xi^{2}} \hookrightarrow L^{4}_{x,t}$ (Lemma \ref{lemmaMergulhoLpBpurgain}), and Plancherel's identity yield that  
\begin{align}\label{casePloPhisigma}
|I|&\lesssim \int_{\mathcal{D}} \xi^{s+1} \langle \xi^{3}\rangle^{\gamma} |h(\xi,\tau)||\widehat{u}
(\xi_{1},\tau_{1})||\widehat{v}(\xi_{2},\tau_{2})|d\mu \nonumber\\
&\lesssim \int_{\mathcal{D}} \frac{\xi^{s+1}\xi^{3\gamma}}{\xi^{s+1+3\gamma}}|h(\xi,\tau)||\widehat{u}(\xi_{1},\tau_{1})||({J^{s+1+3\gamma}_{x}v})^{\wedge}(\xi_{2},\tau_{2})|d\mu \nonumber \\
&\lesssim\|h\|_{L^{2}_{\xi,\tau}}\|(|\widehat{u}|)^{\vee}\|_{L^{4}_{x,t}}\|(|(J^{s+1+3\gamma}_{x}v)^{\wedge}|)^{\vee}\|_{L^{4}_{x,t}} \nonumber\\
&\lesssim \|h\|_{L^{2}_{\xi,\tau}}\|u\|_{Y^{1-\left(\frac{3}{8}+\right),\frac{3}{8}+}}\|v\|_{Y_{\tau=-\xi^{2}}^{s-\left(\frac{3}{8}+\right),\frac{3}{8}+}},
\end{align}
since $\gamma $ satisfies $-\frac{1}{2}<\gamma<-\frac{11}{24}$. 
\newline
\indent We observe that the next term given by $(A,B)=(+hi,lo)$ in \eqref{frequencyinteractionAB} can be treated exactly as done in the preceding case. 
\newline
\indent We will now consider $(A,B)=(-hi,+hi)$ in \eqref{frequencyinteractionAB}. In this case, the set $\mathcal{D}$ is given by
\begin{equation}\label{conjDP-hiP+hi}
    \mathcal{D}=\{(\xi,\tau,\xi_{1},\tau_{1})\in \mathbb{R}^{4}\ \ | \ \xi\geq 8,\  \xi_{1}\leq -1, \xi_{2}\geq 1\}.
\end{equation}
Note that by \eqref{eqinxitau} and \eqref{conjDP-hiP+hi}, we always have in $\mathcal{D}$ that
\begin{equation}\label{desigbilinear-hi+hi}
    \xi\leq \xi_{2} \quad \textrm{and} \quad |\xi_{1}|\leq \xi_{2}.
\end{equation}
Therefore, from equation \eqref{eqinxitau} and the inequalities in \eqref{desigbilinear-hi+hi}, we can always assume that one of the following cases holds:
\begin{itemize}
    \item[1)] High-low interaction: $\xi\sim \xi_{2}$ and $|\xi_{1}|\leq \xi_{2}$;
    \item[2)] High-high interaction : $|\xi_{1}|\sim \xi_{2}$ and $\xi\leq \xi_{2}$.
\end{itemize}
Now, note that the resonance identity \eqref{resonanceidentitycucial} also holds true in this case. In fact, this identity is not dependent on the set $\mathcal{D}$. Thus, from \eqref{conjDP-hiP+hi}, it follows that
\begin{equation*}
    |\sigma-\sigma_{1}-\sigma_{2}|=\xi^{3}+2\xi\xi_{2}\geq \xi^{3}.
\end{equation*}
Therefore, in the case of high-low interactions, the fact that $\xi\sim \xi_{2}$ yields
\begin{equation}\label{estLmaxP-hiP+hi}
    \sigma_{max}\sim \max\{\sigma_{med},\xi^{3}\}.
\end{equation}
In what follows, we proceed to estimate the integral \( I \) considering the high-low and high-high frequency interactions, and whether $\sigma_{max}=|\sigma|,|\sigma_{1}|$ or $|\sigma_{2}|$. 

First, consider the case $\sigma_{max}=|\sigma_{2}|$ and high-low interactions. In this case, note that from the frequency localization
\begin{align*}
    |I|&\lesssim \int_{\mathcal{D}} 
     \frac{ \xi^{s+1}}{\langle\xi_{1}\rangle ^{1-\left(\frac{1}{4}+\right)}\xi^{s-\frac{1}{2}}\langle \sigma_{2}\rangle^{\frac{1}{2}} } \langle \sigma \rangle^{\gamma}|h(\xi,\tau)|(J^{1-\left(\frac{1}{4}+\right)}_{x}u)^{\wedge}(\xi_{1},\tau_{1})|\langle \sigma_{2} \rangle^{\frac{1}{2}}|({J^{s-\frac{1}{2}}_{x}v})^{\wedge}(\xi_{2},\tau_{2})|d\mu.
\end{align*}
Therefore, from Hölder's inequality, the estimate \eqref{estLmaxP-hiP+hi}, Lemmas \ref{lemmaimersionBourgainNLS} and \ref{lemmaMergulhoLpBpurgain}, and Plancherel's identity, it follows that
\begin{align*}
    |I|&\lesssim \left\|\left(\langle \sigma \rangle^{\gamma} |h|\right)^{\vee}\right\|_{L^{6}_{x,t}}\|(|J^{1-\left(\frac{1}{4}+\right)}_{x}u)^{\wedge}|)^{\vee}\|_{L^{3}_{x,t}}\|\langle \sigma_{2} \rangle^{\frac{1}{2}}(J^{s-\frac{1}{2}}_{x}v)^{\wedge}\|_{L^{2}_{\xi_{2},\tau_{2}}}\\
    &\lesssim\|\langle \sigma \rangle^{\gamma +\frac{4}{9}+} h\|_{L^{2}_{\xi,\tau}}\|u\|_{Y^{1-\left(\frac{1}{4}+\right),\frac{1}{4}+}}\|v\|_{Y_{\tau=-\xi^{2}}^{s-\frac{1}{2},\frac{1}{2}}}\\
    &\lesssim \|h\|_{L^{2}_{\xi,\tau}}\|u\|_{Y^{1-\left(\frac{1}{4}+\right),\frac{1}{4}+}}\|v\|_{Y_{\tau=-\xi^{2}}^{s-\frac{1}{2},\frac{1}{2}}}.
\end{align*}
\indent Considering the case $\sigma_{max}=|\sigma_{1}|$ and high-low interactions, the same arguments used in the preceding case can be applied here. In fact, we can see that
\begin{equation*}
    |I|\lesssim \|h\|_{L^{2}_{\xi,\tau}}\|u\|_{Y^{\frac{1}{2},\frac{1}{2}}}\|v\|_{Y_{\tau=-\xi^{2}}^{s-\left(\frac{1}{4}+\right),\frac{1}{4}+}}.
\end{equation*}
\indent Next, assume the case $\sigma_{max}=|\sigma|$ and high-low interactions. Then, according to relation \eqref{estLmaxP-hiP+hi}, we may assume that $|\sigma|\sim \xi^{3}$. Therefore, by repeating the argument used in \eqref{casePloPhisigma}, we conclude that 
\begin{align*}
|I|
&\lesssim  \|h\|_{L^{2}_{\xi,\tau}}\|u\|_{Y^{1-\left(\frac{3}{8}+\right),\frac{3}{8}+}}\|v\|_{Y_{\tau=-\xi^{2}}^{s-\left(\frac{3}{8}+\right),\frac{3}{8}+}}.
\end{align*}
\indent Now, concerning the high–high interactions, observe that we do not need to invoke an estimate such as \eqref{estLmaxP-hiP+hi} in order to estimate the integral $I$. In fact, since $|\xi_{1}|\sim \xi_{2}$ and $\xi\leq \xi_{2}$ in this case, we have that
\begin{align*}
   |I| &\lesssim  \int_{\mathcal{D}}   \frac{\xi ^{s+1}\langle \sigma \rangle^{\gamma}}{\xi_{2}\xi_{2}^{s}}|h(\xi,\tau)||(J^{1}_{x}u)^{\wedge}(\xi_{1},\tau_{1})||(J^{s}_{x}v)^{\wedge}(\xi_{2},\tau_{2})|d\mu\\
   &\lesssim \int_{\mathcal{D}} \langle \sigma\rangle ^{\gamma}|h(\xi,\tau)||(J^{1}_{x}u)^{\wedge}(\xi_{1},\tau_{1})||(J^{s}_{x}v)^{\wedge}(\xi_{2},\tau_{2})|d\mu
\end{align*}
Thus, the same arguments used in the preceding cases ensure that
\begin{align*}
    |I|&\lesssim \left\|\left(\langle  \sigma \rangle^{\gamma}|h|\right)^{\vee}\right\|_{L^{6}_{x,t}}\|(J_{x}^{1}u)^{\wedge}\|_{L^{2}_{\xi_{1},\tau_{1}}}\|(|(J_{x}^{s}u)^{\wedge}|)^{\vee}\|_{L^{3}_{x,t}}
    \\
    &\lesssim \|\langle \sigma \rangle^{\gamma +\frac{4}{9}+}h\|_{L^{2}_{\xi,\tau}}\|u\|_{Y^{1,0}}\|v\|_{Y_{\tau=-\xi^{2}}^{s,\frac{1}{4}+}}   \\
    &\lesssim \|h\|_{L^{2}_{\xi,\tau}}\|u\|_{Y^{1,0}}\|v\|_{Y_{\tau=-\xi^{2}}^{s,\frac{1}{4}+}}.
\end{align*}
\indent Note that, considering $(A,B)=(+hi,-hi)$ in \eqref{frequencyinteractionAB}, the estimate of this term  does not present any additional difficulties. Indeed, in this case, we have 
\begin{equation*}
    \mathcal{D}=\{(\xi,\tau,\xi_{1},\tau_{1})\in \mathbb{R}^{4}\ \ | \ \xi\geq 8,\  \xi_{1}\geq 1, \xi_{2}\leq -1\}.
\end{equation*}
Thus, it follows that
\begin{equation}\label{casoP_{+hi}P_{-hi}}
    \xi\leq \xi_{1} \quad \textrm{and} \quad |\xi_{2}|\leq \xi_{1}.
\end{equation}
Then, by using \eqref{eqinxitau} and \eqref{casoP_{+hi}P_{-hi}}, we only need to consider that the following cases hold:
\begin{itemize}
    \item[1)] High-low interaction : $\xi\sim \xi_{1}$ and $|\xi_{2}|\leq \xi_{1}$;
    \item[2)] High-high interaction : $|\xi_{2}|\sim 
    \xi_{1}$ and $\xi\leq \xi_{1}$
\end{itemize}
Considering high-low interactions and using the resonance identity \eqref{resonanceidentitycucial}, we deduce once again the following relation
\begin{equation}\label{LmaxrelationP+hiP+hi}
    \sigma_{max}\sim \max\{\sigma_{med},\xi^{3}\}.
\end{equation}
Therefore, the same arguments used to estimate 
$\|\partial_{x}P_{+HI}(P_{-hi}u\,P_{+hi}\bar{v})\|_{X^{s,\gamma}}$ 
can be applied here as well, covering both the high–low and high–high frequency interactions. Furthermore, it is important to note that, as we have $\xi \sim \xi_{1}$ in the high–low interaction case, the integral $I$ is estimated with $s - \theta$ derivatives falling on $u$ rather than on $v$.
\newline
\indent Finally, we proceed with the estimate of the term given by $(A,B)=(+hi,+hi)$ in \eqref{frequencyinteractionAB}. Observe that the set $\mathcal{D}$ is given by
\begin{equation}\label{conjDP+hiP+hi}
    \mathcal{D}=\{(\xi,\tau,\xi_{1},\tau_{1})\in \mathbb{R}^{4}\ \ | \ \xi\geq 8,\  \xi_{1}\geq 1,\ \xi_{2}\geq 1\}.
\end{equation}
Thus, from \eqref{eqinxitau} and \eqref{conjDP+hiP+hi}, we have in $\mathcal{D}$
\begin{equation}\label{bilinearinequalityDP+hiP+hi}
    \xi_{1}\leq \xi \quad \textrm{and} \quad \xi_{2}\leq \xi. 
\end{equation}
In what follows, \eqref{eqinxitau} and \eqref{bilinearinequalityDP+hiP+hi} ensure that only two cases of frequency interactions hold:
\begin{itemize}
    \item[1)] High-low interaction : $\xi_{1}\sim \xi$ and $\xi_{2}\leq \xi$;
    \item[2)] High-high interaction : $\xi_{2}\sim \xi$ and $\xi_{1}\leq \xi$.
\end{itemize}
By invoking once again the resonance identity \eqref{resonanceidentitycucial}, together with \eqref{conjDP+hiP+hi} and \eqref{bilinearinequalityDP+hiP+hi}, we deduce that the same relation given in \eqref{LmaxrelationP+hiP+hi} holds in $\mathcal{D}$ in both cases under consideration.
\newline
\indent First, we will consider the case $\sigma_{max}=|\sigma_{1}|$ and high-low interactions. Thus, the integral in \eqref{integralI} can be estimated as follows
\begin{align*}
    |I|&\lesssim \int_{\mathcal{D}} \frac{\xi^{s+1}}{\xi^{s-\frac{1}{2}}\xi_{2}^{1-(\frac{1}{4}+)}\langle \sigma_{1} \rangle^{\frac{1}{2}}} \langle \sigma \rangle^{\gamma}|h(\xi,\tau)||\langle \sigma_{1} \rangle^{\frac{1}{2}}(J_{x}^{s-\frac{1}{2}}u)^{\wedge}(\xi_{1},\tau_{1})||(J_{x}^{1-(\frac{1}{4}+)}v)^{\wedge}(\xi_{2},\tau_{2})|d\mu.
\end{align*}
In this case, using Hölder's inequality, Plancherel's identity, the Lemmas \ref{lemmaMergulhoLpBpurgain} and \ref{lemmaimersionBourgainNLS}, along with \eqref{LmaxrelationP+hiP+hi}, we can deduce that
\begin{align*}
|I|& \lesssim \|(\langle \sigma \rangle^{\gamma}|h|)^\vee\|_{L^{6}_{x,t}}\|\langle \sigma_{1} \rangle^{\frac{1}{2}}(J^{s-\frac{1}{2}}_{x}u)^{\wedge}\|_{L^{2}_{\xi_{1},\tau_{1}}}\|(|(J^{1-(\frac{1}{4}+)}_{x}v)^{\wedge}|)^{\vee}\|_{L^{3}_{x,t}}\\
&\lesssim \|h\|_{L^{2}_{\xi,\tau}}\|u\|_{Y^{s-\frac{1}{2},\frac{1}{2}}}\|v\|_{Y^{s-\left(\frac{1}{4}+\right),\frac{1}{4}+}}
\end{align*}
\indent Concerning the case $\sigma_{max}=|\sigma_{1}|$ and high-high interactions, we can handle it with arguments analogous to those applied in the preceding case. Indeed, we can obtain
\begin{equation*}
    |I|\lesssim \|h\|_{L^{2}_{\xi,\tau}}\|u\|_{Y^{\frac{1}{2},\frac{1}{2}}}\|u\|_{Y_{\tau=-\xi^{2}}^{s-\left(\frac{1}{4}+\right),\frac{1}{4}+}}.
\end{equation*}
\indent Similarly, in both cases where \(|\sigma_2| = \sigma_{\max}\), involving either high-low or high-high frequency interactions, no significant difficulties are encountered. Thus, the same arguments used for the case $|\sigma_{1}|=\sigma_{max}$ are enough to estimate the integral in \eqref{integralI}. In fact, we have the following estimates
\begin{equation*}
     |I|\lesssim \|h\|_{L^{2}_{\xi,\tau}}\|u\|_{Y^{s-\left(\frac{1}{4}+\right),\frac{1}{4}+}}\|v\|_{Y_{\tau=-\xi^{2}}^{\frac{1}{2},\frac{1}{2}}} \quad \textrm{and} \quad |I|\lesssim \|h\|_{L^{2}_{\xi,\tau}}\|u\|_{Y^{1-\left(\frac{1}{4}+\right),\frac{1}{4}+}}\|v\|_{Y_{\tau=-\xi^{2}}^{s-\frac{1}{2},\frac{1}{2}}},
\end{equation*}
corresponding to the high-low and high-high frequency interactions, respectively.
\newline
\indent Finally, assume that \(\sigma_{\max} = |\sigma|\). In this case, according to the relation \eqref{LmaxrelationP+hiP+hi}, we may assume that \(|\sigma| \sim \xi^3\). Thus, considering high-low interactions and proceeding similarly to \eqref{casePloPhisigma}, we can obtain that 
\begin{align*}
   |I|\lesssim \|h\|_{L^{2}_{\xi,\tau}}\|u\|_{Y^{s-\frac{3}{8}+,\frac{3}{8}+}}\|v\|_{Y_{\tau=-\xi^{2}}^{1-\frac{3}{8}+,\frac{3}{8}+}}.
\end{align*}
Regarding the high-high interactions, similar arguments yield that
\begin{equation*}
    |I|\lesssim \|h\|_{L^{2}_{\xi,\tau}}\|u\|_{Y^{1-\frac{3}{8}+,\frac{3}{8}+}}\|v\|_{Y_{\tau=-\xi^{2}}^{s-\frac{3}{8}+,\frac{3}{8}+}}.
\end{equation*}
Therefore, it concludes the proof of the proposition.
\end{proof}

\begin{remark}\
\vspace{-1em}
\begin{itemize}
    \item[1.] Note that one could establish the bilinear estimates as before by employing the first space \( X^{s,-\frac{1}{2},1} \) instead of \( X^{s,\gamma} \). However, this would only make the proof more extensive, as it is not necessary for subsequent arguments. Nevertheless, as done in \cite{HOBOinH1-Didier-Molinet}, it is crucial to use this space to establish the bilinear estimate concerning the term \( \partial_{x}P_{+hi}(wP_{-}\partial_{x}r) \).
    \item[2.] We also emphasize that the crucial argument used to establish these bilinear estimates lies in noting that the relation \( \sigma_{\max} \sim \max\{\sigma_{med},\xi^{3}\} \) holds in certain cases of frequency interactions. Furthermore, note that we can establish such a relation only because we have a cubic symbol \( \xi^{3} \) interacting with two quadratic symbols \( \xi_{1}^{2} \) and \( \xi_{2}^{2} \), i.e., at high frequencies, the cubic term prevails.  
\end{itemize}
\end{remark}
\indent The previous proposition provides us with the estimate that we needed to handle the term $$\|\partial_{x}P_{+hi}(e^{iF}|q|^{2})\|_{X^{s,-\frac{1}{2},1}}.$$
In fact, we have the following result

\begin{prop}\label{propbilinearestimateonq2}
Let $0<T\leq 1$, $1\leq s \leq \frac{3}{2}$, $s\leq s'\leq s+1$, and $(r,q)\in L^{\infty}(\mathbb{R}:L^{2} (\mathbb{R})\times L^{2} (\mathbb{R}))$ a solution to \eqref{BO-NLS} supported in the interval $[0,2T]$. Then, it holds that
\begin{align}\label{bilinearestimateonq2}
    \|\partial_{x}P_{+hi}(e^{iF}|q|^{2})\|_{X^{s,-\frac{1}{2},1}}\lesssim \|q\|_{L^{\infty}_{t}H^{1}_{x}}^{2}+\sup_{0\leq \theta \leq 1}\|e^{iF}q\|_{Y^{1-\theta,\theta}}\|q\|_{Y^{s,1}}+\sup_{0\leq \theta \leq 1}\|e^{iF}q\|_{Y^{s-\theta,\theta}}\|q\|_{Y^{1,1}}.
\end{align}
\end{prop}
\begin{proof}
Indeed, using Proposition \ref{propimersionsbourgainspaces}, Bernstein's inequality, and the Sobolev embedding $H^{1}\hookrightarrow L^{\infty}$, we deduce that
\begin{align}\label{bilinearestimateonq21}
    \|\partial_{x}P_{+hi}(e^{iF}|q|^{2})\|_{X^{s,-\frac{1}{2},1}}&\lesssim  \|\partial_{x}(P_{+hi}-P_{+HI})(e^{iF}|q|^{2})\|_{X^{s,-\frac{1}{2},1}}+ \|\partial_{x}P_{+HI}(e^{iF}|q|^{2})\|_{X^{s,-\frac{1}{2},1}} \nonumber \\
    &\lesssim \sum_{2\leq N\leq 2^{3}}\|\partial_{x}P_{N}(e^{iF}|q|^{2})\|_{X^{1,-\frac{1}{2}+}}+\|\partial_{x}P_{+HI}(e^{iF}|q|^{2})\|_{X^{s,-\frac{1}{2},1}}\nonumber \\
    &\lesssim \sum_{2\leq N\leq 2^{3}}\|\partial_{x}P_{N}(e^{iF}|q|^{2})\|_{L^{1+}_{t}H^{s}_{x}}+\|\partial_{x}P_{+HI}(e^{iF}|q|^{2})\|_{X^{s,-\frac{1}{2},1}}\nonumber \\
    &\lesssim T^{1-}\|e^{iF}|q|^{2}\|_{L^{\infty}_{t}L^{2}_{x}}+\|\partial_{x}P_{+HI}(e^{iF}|q|^{2})\|_{X^{s,-\frac{1}{2},1}}\nonumber \\
    &\lesssim \|q\|_{L^{\infty}_{t}H^{1}_{x}}^{2}+\|\partial_{x}P_{+HI}(e^{iF}|q|^{2})\|_{X^{s,-\frac{1}{2},1}}.
\end{align}
Then, we need only to estimate the term $\|\partial_{x}P_{+HI}(e^{iF}|q|^{2})\|_{X^{s,-\frac{1}{2},1}}$. In this regard, consider any $\gamma\in (-\frac{1}{2},-\frac{11}{24})$. Therefore, by using again Proposition \ref{propimersionsbourgainspaces}, along with Proposition \ref{propetesimativabilinearcrucial}, it follows that 
\begin{align}\label{bilinearestimateonq23}
&\quad \ \|\partial_{x}P_{+HI}(e^{iF}|q|^{2})\|_{X^{s,-\frac{1}{2},1}}\nonumber \\ 
&\lesssim \|\partial_{x}P_{+HI}(e^{iF}|q|^{2})\|_{X^{s,\gamma}}\nonumber \\
&\lesssim \sup_{0\leq \theta \leq 1}\|e^{iF}q\|_{Y^{1-\theta,\theta}}\sup_{0\leq \theta \leq 1}\|q\|_{Y^{s-\theta,\theta}}+\sup_{0\leq \theta \leq 1}\|e^{iF}q\|_{Y^{s-\theta,\theta}}\sup_{0\leq \theta \leq 1}\|q\|_{Y^{1-\theta,\theta}} +\|e^{iF}q\|_{Y^{1,0}}\|q\|_{Y^{s,\frac{1}{4}+}}\nonumber\\
&\lesssim \sup_{0\leq \theta \leq 1}\|e^{iF}q\|_{Y^{1-\theta,\theta}}\|q\|_{Y^{s,1}}+\sup_{0\leq \theta \leq 1}\|e^{iF}q\|_{Y^{s-\theta,\theta}}\|q\|_{Y^{1,1}}.
\end{align}
Thus, from estimates \eqref{bilinearestimateonq21} and \eqref{bilinearestimateonq23}, the result follows.
\end{proof}

\begin{remark}
We observe that an alternative strategy to estimate the term 
$
\|\partial_{x}P_{+hi}(e^{iF}|q|^{2})\|_{X^{s,-\frac{1}{2},1}}
$
would be to apply the embeddings 
$
L^{2}_{t}H^{s}_{x}\hookrightarrow X^{s,-\frac{1}{2}+}
\hookrightarrow X^{s,-\frac{1}{2},1},
$
which would allow us to remove the problematic exponential factor in the Bourgain space. However, in order to obtain local well-posedness in the energy space, this approach would require establishing the estimate
\begin{equation*}
    \|\partial_{x}(|q|^{2})\|_{L^{2}_{t}H^{s}_{x}}
    \lesssim 
    \|q\|_{Y^{s,b}}^{2},
\end{equation*}
for some $\frac{1}{2}<b\leq 1$. On the other hand, according to the proof of Theorem 4.2 in \cite{Shap-wellposedenss-Leandro}, this estimate fails.
\end{remark}

We are now in a position to prove the previous Proposition \ref{bilinearestimate}.

\begin{proof}[Proof of Proposition \ref{bilinearestimate}]
Let $(\tilde{r},\tilde{q})$ and $\tilde{w}$ be extensions of, respectively, $(r,q)$ and $w$ such that 
\begin{equation*}
\|\tilde{v}\|_{X^{1-2\theta,\theta}}\leq 2\|v\|_{X^{1-2\theta,\theta}_{T}},\ \|\tilde{q}\|_{Y^{s,1}}\lesssim 2\|q\|_{Y^{s,1}_{T}},\ \|e^{i\tilde{F}}\tilde{q}\|_{Y^{s-\theta,\theta}}\lesssim 2\|e^{iF}q\|_{Y_{T}^{s-\theta,\theta}}  ,\  \|\tilde{w}\|_{X^{s,\frac{1}{2},1}}\leq 2\|w\|_{X^{s,\frac{1}{2},1}_{T}},
\end{equation*}
for all $0\leq \theta \leq 1$. Since \( w \) satisfies the equation in \eqref{eqofw}, it follows from Duhamel's principle that \( w \) satisfies the integral formulation given by
\begin{equation}\label{duhamelofw}
    w(t)=\eta(t)V(t)w(0)+\eta(t)\int_{0}^{t}V(t-t')[N_{1}(e^{iF},v,W,w)+N_{2}(e^{iF},q)]dt'.
\end{equation}
Therefore, gathering the Propositions $4.2,\ 4.3$ and $4.4$ in \cite{HOBOinH1-Didier-Molinet} and the estimates \eqref{linearestimate}, \eqref{nonlinearestimate}, along with Proposition \ref{propbilinearestimateonq2}, we can deduce the following 
\begin{align}\label{proofbilinearestimate1}
    &\|w\|_{X^{s,\frac{1}{2},1}_{T}}\lesssim \|w(0)\|_{H^{s}}+\|w\|_{X^{s,\frac{1}{2},1}_{T}}\left(\sup_{0\leq \theta \leq 1}\|r\|_{X_{T}^{1-2\theta,\theta}}+\|r\|_{L^{4}_{T}W^{1,4}_{x}}+\|r\|_{L^{\infty}_{T}H^{1}_{x}}\|r\|_{L^{4}_{T}W^{1,4}_{x}}\right)\nonumber \\
    &+\|r\|_{L^{4}_{T,x}}^{2}+(1+\|r\|_{L^{\infty}_{T}H^{1}_{x}}^{2})(1+\|r\|_{L^{\infty}_{T}H^{1}_{x}})\left(\|r\|_{L^{\infty}_{T}H^{1}_{x}}^{2}\|r\|_{L^{\infty}_{T}H^{s}_{x}}+\|r\|_{L^\infty_xH^1_x}^{2}+\|r\|_{L^{2}_{x}L^{\infty}_{T}}\|J_{x}^{s}\partial_{x}r\|_{\widetilde{L^{\infty}_{x}L^{2}_{T}}}\right)\nonumber\\
    &+\|q\|_{L^{\infty}_{T}H^{1}_{x}}^{2}+\sup_{0\leq \theta \leq 1}\|e^{iF}q\|_{Y_{T}^{1-\theta,\theta}}\|q\|_{Y^{s,1}_{T}}+\sup_{0\leq \theta \leq 1}\|e^{iF}q\|_{Y_{T}^{s-\theta,\theta}}\|q\|_{Y_{T}^{1,1}}.
\end{align}
On the other hand, since $1\leq s\leq \frac{3}{2}$ then from estimate \eqref{estimateexponencialinHs} follows that
\begin{equation}\label{proofbilinearestimate2}
    \|w(0)\|_{H^{s}}=\|\partial_{x}P_{+hi}(e^{iF(0)})\|_{H^{s}}\lesssim\|e^{iF(0)}r_{0}\|_{H^{s}_{x}}\lesssim (1+\|r_{0}\|_{H^{1}}^{2})\|r_{0}\|_{H^{s}}.
\end{equation}
From estimates \eqref{proofbilinearestimate1} and \eqref{proofbilinearestimate2} the result follows.
\end{proof}

\begin{remark}
It is worth emphasizing that the estimate established in Proposition~4.2 of~\cite{HOBOinH1-Didier-Molinet} remains valid even in the case where \( a < 0 \) in equation (1.1) considered in that work. Indeed, the key resonance identity in this setting is given by
\begin{equation*}
    \sigma - \sigma_{1} - \sigma_{2} = -2\xi_{1}\xi_{2} - 3\xi\xi_{1}\xi_{2},
\end{equation*}
where
\[
    \sigma = \tau - \xi|\xi| - \xi^{3}
    \quad \text{and} \quad
    \sigma_{i} = \tau_{i} - \xi_{i}|\xi_{i}| - \xi_{i}^{3}.
\]
This identity is sufficient to establish the same bilinear estimate proved in the proposition mentioned earlier.
\end{remark}

\section{Local well-posedness for the HBOS system}

In this section, we gather the estimates established in the previous Section \ref{aprioriestimates1,2} and present the local and global well-posedness for the HBOS system in \eqref{BO-NLS}, as stated in Theorem \ref{maintheorem} and \ref{globalwellposedness}. In what follows, our proof is based on the approach used in Section 4 of \cite{BOinL2-Didier} and Section 5 of \cite{HOBOinH1-Didier-Molinet} (see also \cite{ILWinL2-Chapouto}). Then, some analogous or repetitive details will be omitted. Furthermore, it is important to note that, due to Remark \ref{remarkS>3/2}, we will only consider the case $1\leq s\leq \frac{3}{2}$, and the case $s>\frac{3}{2}$ follows from similar arguments since we can obtain the necessary estimates to prove this case. 

Fix $1\leq s\leq \frac{3}{2}$. Then, note that it is sufficient to prove Theorem \ref{maintheorem} for the space $H^{s}\times H^{s+k}$, for each $k\in [0,1]$ also fixed. We decided to use the notation $H^{s}\times H^{s+k}$ instead of $H^{s}\times H^{s'}$ because it makes the definition of the quantity $N_{T}^{s,k}(r,q)$ given in \eqref{Nquantity} below simpler. Furthermore, as in Theorem \ref{existenceofsmoothsolutions}, without loss of generality, we only need to prove the theorem in the case where the initial $(r_{0},q_{0})$ has a sufficiently small $H^{1}\times H^{1+k}$-norm. Thus, suppose that
\begin{equation}\label{datainitialnorm}
    \|(r_{0},q_{0})\|_{H^{1}\times H^{1+k}}=\varepsilon \leq \varepsilon_{0},
\end{equation}
for $0<\varepsilon_{0}\leq 1$ small enough to be chosen later.


\subsection{A priori estimates for smooth solutions}\label{WPaprioriestimates}

First, we will derive \textit{a priori} estimates on smooth solutions of \eqref{BO-NLS}. Let
$$(r,q)\in C([0,T_{0}]:H^{\infty}\times H^{\infty}),\quad  T_{0}=T_{0}(\|(r_{0},q_{0})\|_{H^{3}\times H^{3}})\in (0,1],$$ be a smooth solution to the system \eqref{BO-NLS} with initial data $(r_{0},q_{0})\in H^{\infty}\times H^{\infty}$ obtained by Theorem \ref{existenceofsmoothsolutions} and satisfying \eqref{datainitialnorm}. Furthermore, let $w$ be the gauge function defined as in \eqref{gaugefunction}. Then, we define
\begin{equation}\label{Nquantity}
    N_{T}^{s,k}(r,q):=\max \left\{ \|r\|_{L^{\infty}_{T}H^{s}_{x}},\|r\|_{L^{4}_{T}W^{s,4}_{x}},\|r\|_{L^{2}_{x}L^{\infty}_{T}},\|J^{s}_{x}\partial_{x}r\|_{\widetilde{L^{\infty}_{x}L^{2}_{T}}},\|q\|_{L^{\infty}_{T}H^{s+k}_{x}},\|w\|_{X^{s,\frac{1}{2},1}_{T}}\right\}.
\end{equation}
Since $v,q$ and $w$ are smooth functions, it follows that $T\mapsto N^{s,k}_{T}(r,q)$ is a continuous and nondecreasing function (see, for instance, \cite[Lemma 8.1]{CubicNLS-Oh-Guo} and \cite[Lemma A.8]{CubicNLS-Tadahiro-Oh}). We observe also the estimates \eqref{estimatebourgainr}, \eqref{estimatebourgainq}, \eqref{estimatebourgainexpq}, and \eqref{estimatebourgainw}, along with \eqref{datainitialnorm}, ensure that
\begin{equation}\label{limitN}
    \lim_{T\rightarrow 0^{+}}N_{T}^{s,k}(r,q)\lesssim \|r_{0}\|_{H^{s}}+\|q_{0}\|_{H^{s+k}}.
\end{equation}
Furthermore, gathering estimates \eqref{estimatebourgainr}-\eqref{estimatelinfl2r}, \eqref{estimatebourgainw} and using the definition \eqref{Nquantity}, we obtain 
\begin{equation}\label{estimateNs}
  N_{T}^{s,k}(q,r)\lesssim (1+\|r_{0}\|_{H^{1}}^{2})(\|r_{0}\|_{H^{s}}+\|q_{0}\|_{H^{s+k}})+P(N_{T}^{1,k}(q,r))N_{T}^{s,k}(r,q),  
\end{equation}
where $P$ is a polynomial with positive coefficients and no constant term. On the other hand, from \eqref{limitN} with $s=1$, we can obtain a positive time $0<T_{1}\leq T_{0}$ such that
\begin{equation}
    N_{T}^{1,k}(r,q)\leq C(\|r_{0}\|_{H^{1}}+\|q_{0}\|_{H^{1+k}})+\varepsilon,
\end{equation}
for all $0<T\leq T_1$, where $C$ is a constant. Therefore, using that initial data is satisfying \eqref{datainitialnorm}, it follows that
\begin{equation}\label{limitationN1}
    N_{T}^{1,k}(r,q)\leq C\varepsilon, \quad \forall\  T\in [0,T_1],
\end{equation}
where $C$ is not dependent on $\varepsilon$. Now, by a continuity argument and using the estimate \eqref{estimateNs} with $s=1$, we can ensure that \eqref{limitationN1} holds for all $0<T\leq T_{0}$, provided that $\varepsilon_{0}$ is sufficiently small. In this case, by making $\varepsilon_0$ smaller if necessary and using the estimates \eqref{estimateNs}, \eqref{limitationN1}, together with the fact that $P$ has positive coefficients and no constant term, we obtain that
\begin{equation}\label{boundNs}
    N_{T}^{s,k}(r,q)\lesssim \|r_{0}\|_{H^{s}}+\|q_{0}\|_{H^{s+k}},
\end{equation}
for all $0<T\leq T_0$.

\begin{remark}\label{Tindependentonj}
 Observe that, since $\|(r_0,q_0)\|_{H^1 \times H^{1+k}}$ is sufficiently small, the estimate \eqref{boundNs} allows us to apply Theorem \ref{existenceofsmoothsolutions} a finite number of times to extend the solution $(r,q)$ constructed above to the interval $[0,1]$. This is important to obtain a uniform existence time for the solutions $(r_j,q_j)$ considered in the compactness argument in Subsection \ref{subectionLocalWellposedness}.
\end{remark}

\subsection{Lipschitz bound for initial data having the same low frequency part}\label{WPcontflow}

Our goal here is to establish both the uniqueness and the continuity of the data-to-solution map by considering initial data with the same low-frequency. More precisely, we derive a Lipschitz bound for the solution map restricted to an affine subspace of $H^{s}\times H^{s+k}$. Since our strategy relies on estimates for the difference between two solutions, this restriction enables us to control the difference between the primitives of these solutions, which is essential for our analysis. In the next section, we shall see that these estimates are sufficient to deduce the uniqueness and continuity of the flow map in general.

Let $(r_{1}, q_{1}), (r_{2}, q_{2}) \in C([0, T]: H^{s}(\mathbb{R})\times H^{s+k}(\mathbb{R})),\  0<T\leq 1,$ be two solutions to the system \eqref{BO-NLS} corresponding to initials data $(\varphi_{1}, \psi_{1}),\  (\varphi_{2}, \psi_{2}) \in H^{s}(\mathbb{R}) \times H^{s+k}(\mathbb{R})$, respectively. Then, assume that the low-frequency projections of the initials data $\varphi_{1}$ and $\varphi_{2}$ coincide, i.e., $P_{LO} \varphi_{1} = P_{LO} \varphi_{2}$. 
Furthermore, suppose that the initial data satisfy
\begin{equation}\label{boundinitialdatai}
    \|\varphi_{j}\|_{H^{1}} + \|\psi_{j}\|_{H^{1+k}} \leq \varepsilon \leq \varepsilon_{0}, \quad \text{for } j = 1, 2.
\end{equation}
for $0<\varepsilon_{0}\leq 1$ small enough to be chosen later. Let $F_{j}$ denote the spatial primitive of $r_{j}$ as defined in \eqref{eqofF}, and let $W_{j}$ and $w_{j}$ be the associated gauge functions constructed from $F_{j}$ as specified in \eqref{gaugefunction}. Consequently, by invoking the estimates established in the preceding subsection (\eqref{limitN}-\eqref{boundNs}), the smallness condition on the initial data given in \eqref{boundinitialdatai} yields the a priori bound
\begin{equation}\label{boundNsr1r2}
\begin{array}{ll}
    &N_{T}^{1,k}(r_{j}, q_{j}) \lesssim \varepsilon \leq \varepsilon_{0},\\
    &N_{T}^{s,k}(r_{j},q_{j})\lesssim M,
    \end{array}
\end{equation}
for $j=1,2,$ and $M:=\displaystyle\max_{j=1,2}(\|\varphi_{j}\|_{H^{s}}+\|\psi_{j}\|_{H^{s+k}})$.
Define the following functions
\begin{equation*}
    r:=r_{1}-r_{2},\quad\quad q:=q_{1}-q_{2}, \quad W=W_{1}-W_{2}, \quad w:=w_{1}-w_{2}.
\end{equation*}
 Observe that, since $P_{LO}\varphi_{1}=P_{LO}\varphi_{2}$, we can apply the same argument as in (4-3) of \cite{BOinL2-Didier} to deduce that
 \begin{equation*}
     P_{lo}F_{1}(x,0)=P_{lo}F_{2}(x,0),
 \end{equation*}
 In this case, the mean-value theorem, together with Bernstein's inequality, implies that
 \begin{equation}\label{mvtF}
     \|e^{\pm iF_{1}(0)}-e^{\pm iF_{2}(0)}\|_{L^{\infty}}\lesssim \|\varphi_{1}-\varphi_{2}\|_{L^{2}}.
 \end{equation}
 Furthermore, by once again using that $P_{LO}\varphi_{1}=P_{LO}\varphi_{2}$, together with the Duhamel principle of \eqref{eqofF}, Bernstein's inequality, and the estimate \eqref{boundNsr1r2}, we obtain that (see, for instance, Lemma 4.1 in \cite{BOinL2-Didier})
 \begin{equation}\label{mvtFtx}
     \|e^{\pm iF_{1}}-e^{\pm iF_{2}}\|_{L^{\infty}_{T,x}}\lesssim \|r\|_{L^{\infty}_{T}H^{1}_{x}}+\|q\|_{L^{\infty}_{T}H^{1}_{x}}.
 \end{equation}
Thus, note that from Lemma \ref{lemmaestimateexponencialinHs} and the estimates \eqref{boundinitialdatai}-\eqref{mvtF}, we have
\begin{align}\label{boundw0}
    \|w(0)\|_{H^{s}}  & \lesssim  \|P_{+hi}(e^{iF_{1}(0)}\varphi_{1}-e^{iF_{2}(0)}\varphi_{2})\|_{H^{s}_{x}}\nonumber \\
    &\lesssim  \|P_{+hi}(e^{iF_{1}(0)}(\varphi_{1}-\varphi_{2}))\|_{H^{s}_{s}}+\|P_{+hi}(e^{iF_{1}(0)}-e^{iF_{2}(0)})\varphi_{2}\|_{H^{s}_{x}}\nonumber \\
    &\lesssim  (1+\|\varphi_{1}\|_{H^{1}}^{2})\|\varphi_{1}-\varphi\|_{H^{s}}\nonumber \\
    &\quad + (1+\|\varphi_{1}\|_{H^{1}}^{2}+\|\varphi_{2}\|_{H^{1}}^{2})(\|\varphi_{1}-\varphi_{2}\|_{H^{1}}+\|e^{iF(0)}-e^{iF_{2}(0)}\|_{L^{\infty}})\|\varphi_{2}\|_{H^{s}}\nonumber \\
    &\lesssim  \|\varphi_{1}-\varphi_{2}\|_{H^{s}_{x}}.
\end{align}
Now, since $w_{j}$ satisfies the equation \eqref{eqofw}, we deduce that

\vspace{-2.5em}
\begin{spacing}{1.2}
\begin{equation}\label{eqdifferencew1w2}
\begin{array}{ll}
    &\quad \partial_{t}w-b\mathcal{H}\partial_{x}^{2}w-a\partial_{x}^{3}w\\
    &=\partial_{x}P_{+hi}(e^{iF_{1}}(r_{1}^{2}-r_{2}^{2}))+ \partial_{x}P_{+hi}((e^{iF_{1}}-e^{iF_{2}})r_{2}^{2})+\partial_{x}P_{+hi}(e^{iF_{1}}(r_{1}^{3}-r_{2}^{3}))
    \\&\quad +\partial_{x}P_{+hi}((e^{iF_{1}}-e^{iF_{2}})r_{2}^{3})+\partial_{x}P_{+hi}(W_{1}P_{-}\partial_{x}r)
    +\partial_{x}P_{+hi}(WP_{-}\partial_{x}r_{2})\\
    &\quad +\partial_{x}P_{+hi}(P_{lo}e^{iF_{1}}P_{-}\partial_{x}r)+\partial_{x}P_{+hi}(P_{lo}(e^{iF_{1}}-e^{iF_{2}})P_{-}\partial_{x}r_{2})+\partial_{x}P_{+hi}(w_{1}P_{-}\partial_{x}r) \\
    &\quad +\partial_{x}P_{+hi}(wP_{-}\partial_{x}r_{2})+\partial_{x}P_{+hi}(P_{lo}(e^{iF_{1}}r_{1})P_{-}\partial_{x}r)+\partial_{x}P_{+hi}(P_{lo}(e^{iF_{1}}r_{1}-e^{iF_{2}}r_{2})P_{-}\partial_{x}r_{2}) \\
    &\quad +\frac{1}{2}\partial_{x}P_{+hi}(W_{1}P_{-}\partial_{x}(r_{1}^{2}-r_{2}^{2}))+\frac{1}{2}\partial_{x}P_{+hi}(WP_{-}\partial_{x}(r_{2}^{2}))+\frac{1}{2}\partial_{x}P_{+hi}(P_{lo}e^{iF_{1}}P_{-}\partial_{x}(r_{1}^{2}-r_{2}^{2}))\\
    &\quad +\frac{1}{2}\partial_{x}P_{+hi}(P_{lo}(e^{iF_{1}}-e^{iF_{2}})P_{-}\partial_{x}(r_{2}^{2}))+\partial_{x}P_{+hi}((e^{iF_{1}}q_{1}-e^{iF_{2}}q_{2})\bar{q_{1}})+\partial_{x}P_{+hi}(e^{iF_{2}}q_{2}\bar{q}).
\end{array}
\end{equation}
\end{spacing}
\vspace{-1.5em}
\noindent Therefore, the equation \eqref{eqdifferencew1w2}, the estimates \eqref{linearestimate}, \eqref{nonlinearestimate}, \eqref{boundNsr1r2}, \eqref{mvtFtx}, and \eqref{boundw0},  Lemmas \ref{LemmaFractionalDerivativeMolinet} and \ref{lemmaestimateexponencialinHs}, Propositions \ref{estimatesaprioriofrq} and \ref{propetesimativabilinearcrucial}, along with Propositions 4.2, 4.3, and 4.4 of \cite{HOBOinH1-Didier-Molinet}, yield the following
\begin{align}\label{contflowestimatew}
\|w\|_{X^{s,\frac{1}{2},1}}
&\lesssim \|\varphi_{1}-\varphi_{2}\|_{H^{s}}+\sup_{0\leq \theta \leq 1}\|e^{iF_{1}}q_{1}-e^{iF_{2}}q_{2}\|_{Y_{T}^{1-\theta,\theta}}+\|q\|_{Y^{1,1}_{T}}+\sup_{0\leq \theta\leq 1}\|r\|_{X^{1-2\theta,\theta}_{T}} +\|r\|_{L^{4}_{T}W^{1,4}_{x}}\nonumber\\
&\quad +\|r\|_{L^{2}_{x}L^{\infty}_{T}}+\varepsilon\big( \|r\|_{L^{\infty}_{T}H^{s}_{x}}+\|q\|_{L^{\infty}_{T}H^{s+k}_{x}}+\|J_{x}^{s}\partial_{x}r\|_{\widetilde{L^{^\infty}_{x}L^{2}_{T}}} +\|w\|_{X^{s,\frac{1}{2},1}_{T}}\nonumber\\
&\quad +\sup_{0\leq \theta \leq 1}\|e^{iF_{1}}q_{1}-e^{iF_{2}}q_{2}\|_{Y_{T}^{s-\theta,\theta}}+\|q\|_{Y_{^T}^{s,1}}\big).
\end{align}
Next, note that $r$ satisfies the following equation
\begin{align}\label{contfloweqofrdiferrence}
 \partial_{t}r-b\mathcal{H}\partial_{x}^{2}r+a\partial_{x}^{3}r&=\frac{c}{2}\partial_{x}((r_{1}+r_{2})r)-d\partial_{x}(r_{1}\mathcal{H}\partial_{x}r+\mathcal{H}(r_{1}\partial_{x}r)) \nonumber -d\partial_{x}(r\mathcal{H}\partial_{x}r_{2}+\mathcal{H}(r\partial_{x}r_{2})) \nonumber \\
 &\quad +\beta\partial_{x}(q_{1}\bar{q}+q\bar{q_{2}}).
\end{align}
In this case, proceeding as in the proof of the estimate \eqref{estimatebourgainr} and using the estimate \eqref{boundNsr1r2}, we obtain
\begin{align}\label{contflowestimateBourgainr}
    \sup_{0\leq \theta \leq 1}\|r\|_{X^{1-2\theta,\theta}_{T}}
    &\lesssim \|r\|_{L^{\infty}_{T}H^{1}_{x}}+\varepsilon(\|r\|_{L^{4}_{T}W^{s,4}_{x}}+\|q\|_{L^{\infty}_{T}H^{1}_{x}}).
\end{align}
On the other hand, the equations satisfied by $q_{j}$ and $e^{iF_{j}}q_{j}$ ensure that 
\begin{align}
&i\partial_{t}q-\alpha\partial_{x}^{2}q=-\beta(q_{1}r-qr_{2}), \label{eqofqdifferencesol}
\end{align}
and,
\begin{align}\label{equationdifferenceexpqj}
    &\quad \ (i\partial_{t}-\alpha\partial_{x}^{2})(e^{iF_{1}}q_{1}-e^{iF_{2}}q_{2})\nonumber\\
    &=-Ae^{iF_{1}}q\big[bH\partial_{x}r_{1}+a\partial_{x}^{2}r_{1}+\frac{c}{2}r_{1}^{2}-d(r_{1}\mathcal{H}\partial_{x}r_{1}+\mathcal{H}(r_{1}\partial_{x}r_{1}))+\beta|q_{1}|^{2}\big] \nonumber \\
    &\quad -A(e^{iF_{1}}-e^{iF_{2}})q_{2}[bH\partial_{x}r_{1}+a\partial_{x}^{2}r_{1}+\frac{c}{2}r_{1}^{2}-d(r_{1}\mathcal{H}\partial_{x}r_{1}+\mathcal{H}(r_{1}\partial_{x}r_{1}))+\beta|q_{1}|^{2}\big] -Ae^{iF_{2}}q_{2}\times \nonumber \\
    &\quad \times[bH\partial_{x}r+a\partial_{x}^{2}r+\frac{c}{2}(r_{1}+r_{2})r-d(r_{1}\mathcal{H}\partial_{x}r+\mathcal{H}(r_{1}\partial_{x}r)+r\mathcal{H}\partial_{x}r_{2}+\mathcal{H}(r\partial_{x}r_{2}))+\beta( q_{1}\bar{q}+q\bar{q_{2}})\big] \nonumber\\
    &\quad -\alpha\partial_{x}^{2}(e^{iF_{1}})q-\alpha\partial_{x}^{2}(e^{iF_{1}}-e^{iF_{2}}) q_{2}-2\alpha \partial_{x}(e^{iF_{1}})\partial_{x}q-2\alpha\partial_{x}(e^{iF_{1}}-e^{iF_{2}})\partial_{x}q_{2}-\beta e^{iF_{1}}r_{1}q\nonumber\\
    &\quad -\beta (e^{iF_{1}}r_{1}-e^{iF_{2}}r_{2})q_{2}.
\end{align}
Therefore, using the equations \eqref{eqofqdifferencesol} and \eqref{equationdifferenceexpqj}, and proceeding exactly as done in the proof of the estimates \eqref{estimatebourgainq} and \eqref{estimatebourgainexpq}, along with Lemma 3.3 in \cite{ILWinL2-Chapouto} and the estimates \eqref{boundNsr1r2} and \eqref{mvtFtx}, we deduce that
\begin{align}
\|q\|_{Y^{s,1}_{T}}&\lesssim \|\psi_{1}-\psi_{2}\|_{H^{s}}+ \|q _{1}\|_{L^{\infty}_{T}H^{s}_{x}}\|r\|_{L^{\infty}_{T}H^{s}_{x}}+\|q\|_{L^{\infty}_{T}H^{s}_{x}}\|r_{2}\|_{L^{\infty}_{T}H^{s}_{x}},\label{flowdiffbourgainq}\\
\|q\|_{Y^{1,1}_{T}}&\lesssim \|\psi_{1}-\psi_{2}\|_{H^{1}}+\varepsilon(\|r\|_{L^{\infty}_{T}H^{1}_{x}}+\|q\|_{L^{\infty}_{T}H^{1}_{x}}),
\end{align}
\begin{align}
\sup_{0\leq \theta \leq 1}\|e^{iF_{1}}q_{1}-e^{iF_{2}}q_{2}\|_{Y_{T}^{s-\theta,\theta}}
&\lesssim \|r\|_{L^{\infty}_{T}H^{s}_{x}}+\|q\|_{L^{\infty}_{T}H^{s}_{x}}+\|D_{x}^{s-1}q\|_{\widetilde{L^{2}_{x}L^{\infty}_{T}}}+\|J_{x}^{s}\partial_{x}r\|_{\widetilde{L^{\infty}_{x}L^{2}_{T}}},\label{flowdiffbourgainexpq}\\
    \sup_{0\leq \theta \leq 1}\|e^{iF_{1}}q_{1}-e^{iF_{2}}q_{2}\|_{Y_{T}^{1-\theta,\theta}}&\lesssim \varepsilon(\|r\|_{L^{\infty}_{T}H^{1}_{x}}+\|D_{x}^{s-1}q\|_{\widetilde{L^{2}_{x}L^{\infty}_{T}}}+\|J_{x}^{s}\partial_{x}r\|_{\widetilde{L^{\infty}_{x}L^{2}_{T}}})+\|q\|_{L^{\infty}_{T}H^{1}_{x}},\label{flowdiffbourgainexpq1}
\end{align}
where we have used the following estimates obtained from the Bernstein inequality
\begin{equation}\label{estimatefrom1tos}
\|q\|_{\widetilde{L^{2}_{x}L^{\infty}_{T}}}\lesssim (1+\|r\|_{L^{\infty}_{T}H^{1}_{x}})\|q\|_{L^{\infty}_{T}H^{1}_{x}}+\|D_{x}^{s-1}q\|_{\widetilde{L^{2}_{x}L^{\infty}_{T}}} \quad \textrm{and} \quad \|J_{x}^{1}\partial_{x}r\|_{\widetilde{L^{\infty}_{x}L^{2}_{T}}}\lesssim \|r\|_{L^{\infty}_{T}H^{1}_{x}}+\|J_{x}^{s}\partial_{x}r\|_{\widetilde{L^{\infty}_{x}L^{2}_{T}}},
\end{equation}
\indent Also, since $r_{1}$ and $r_{2}$ satisfy the equation \eqref{rintermofw2}, then $r$ must satisfy
\begin{align}\label{eqPHIrdifference}
  iAP_{+HI}r&=P_{+HI}(e^{-iF_{1}}w)+P_{+HI}((e^{-iF_{1}}-e^{iF_{2}})w_{2})+P_{+HI}(P_{+hi}e^{-iF_{1}}\partial_{x}P_{lo}(e^{iF_{1}}-e^{iF_{2}}))\nonumber\\
  &\quad+P_{+HI}(P_{+hi}(e^{iF_{1}}-e^{iF_{2}})\partial_{x}P_{lo}e^{iF_{2}})+P_{+HI}(P_{+HI}e^{-iF_{1}}\partial_{x}P_{-hi}(e^{iF_{1}}-e^{iF_{2}})) \nonumber \\
  &\quad +P_{+HI}(P_{+HI}(e^{-iF_{1}}-e^{iF_{2}})\partial_{x}P_{-hi}e^{iF_{2}}).
\end{align}
Therefore, the same arguments applied in the proof of the estimate \eqref{estimateofrinhs}, together with the estimates \eqref{estimateexponencialinHs}, \eqref{estimateexponencialinHs2}, and \eqref{mvtFtx}, yield the following
\begin{align}\label{contflowestimateHsr}
&\|J^{s}_{x}r\|_{L^{p}_{T}L^{q}_{x}} \lesssim  \|\varphi_{1}-\varphi_{2}\|_{H^{s}}+ \|w\|_{X^{s,\frac{1}{2},1}_{T}}+ \varepsilon(\|r\|_{L^{\infty}_{T}H^{1}_{x}}+\|q\|_{L^{\infty}_{T}H^{1}_{x}}),
\end{align}
for $(p,q)=(\infty,2)$ or $(4,4)$, where we used the preceding argument twice (also for $\|r\|_{L^{\infty}_{T}H^{1}_{x}}$ and $\|q\|_{L^{\infty}_{T}H^{1}_{x}}$), along with the estimates \eqref{estimatebourgainw} and \eqref{boundNsr1r2}, in order to obtain the $\varepsilon$ in the estimate.
In the same way, using once again that $q$ satisfies the equation \eqref{eqofqdifferencesol}, proceeding as we have done for \eqref{estimateofqHs'}, and applying the estimate \eqref{boundNsr1r2}, we can obtain that
\begin{align}\label{contflowestimateHskq}
\|J^{s+k}_{x}q\|_{L^{\infty}_{T}L^{2}_{x}}
&\lesssim \|\psi_{1}-\psi_{2}\|_{H^{s+k}}+\varepsilon(\|J^{s}_{x}\partial_{x}r\|_{\widetilde{L^{\infty}_{x}L^{2}_{T}}}+\|q\|_{L^{\infty}_{T}H^{s+k}_{x}})+\|q\|_{\widetilde{L^{2}_{x}L^{\infty}_{T}}}+\|r\|_{L^{\infty}_{T}H^{1}_{x}}.
\end{align}
\noindent Regarding the term $\|r\|_{L^{2}_{x}L^{\infty}_{T}}$, an application of the equations \eqref{contfloweqofrdiferrence} and \eqref{eqPHIrdifference}, the estimates \eqref{boundNsr1r2} and \eqref{mvtFtx}, and proceeding as done in the proof of \eqref{estimatel2linfr}, yields that
\begin{equation}\label{contflowestimateL2Linftyr}
    \|r\|_{L^{2}_{x}L^{\infty}_{T}}\lesssim \|\varphi_{1}-\varphi_{2}\|_{H^{1}}+\varepsilon(\|r\|_{L^{\infty}_{T}H^{1}_{x}}+\|q\|_{L^{\infty}_{T}H^{1}_{x}}+\|r\|_{L^{2}_{x}L^{\infty}_{T}})+\|w\|_{X^{1,\frac{1}{2},1}}
\end{equation}
Now, once again employing the fact that $q$ satisfies \eqref{eqofqdifferencesol}, together with Duhamel’s principle, the bound \eqref{boundNsr1r2}, and following the same procedure used to derive \eqref{estimatel2linfq}, we obtain
\begin{equation}\label{contflowestimateL2LinftyQ}
    \|D_{x}^{s-1}q\|_{\widetilde{L^{2}_{x}L^{\infty}_{T}}}\lesssim \|\psi_{1}-\psi_{2}\|_{H^{s}}+\varepsilon(\|r\|_{L^{\infty}_{T}H^{s}_{x}}+\|q\|_{L^{\infty}_{T}H^{s}_{x}})+\|r\|_{L^{\infty}_{T}H^{1}_{x}}+\|q\|_{L^{\infty}_{T}H^{1}_{x}}.
\end{equation}
Lastly, on the one hand, from ideas applied to derive \eqref{estimatelinfl2r}, the fact that $r$ satisfies \eqref{contfloweqofrdiferrence} and \eqref{eqPHIrdifference}, and the estimates  \eqref{boundNsr1r2} and \eqref{mvtFtx}, it follows that
\begin{align}\label{contflowestmateLinftyL2R}
\|J^{s}_{x}\partial_{x}r\|_{\widetilde{L^{\infty}_{x}L^{2}_{T}}}\lesssim \|\varphi_{1}-\varphi_{2}\|_{H^{1}_{x}}+\varepsilon\|J^{s}_{x}\partial_{x}r\|_{\widetilde{L^{\infty}_{x}L^{2}_{T}}}+\|r\|_{L^{\infty}_{T}H^{1}_{x}}+\|q\|_{L^{\infty}_{T}H^{1}_{x}}+\|w\|_{X^{s,\frac{1}{2},1}_{T}}.
\end{align}
Consequently, by combining the estimates \eqref{contflowestimatew}, \eqref{contflowestimateBourgainr}, \eqref{flowdiffbourgainq}-\eqref{flowdiffbourgainexpq1}, and \eqref{contflowestimateHsr}-\eqref{contflowestmateLinftyL2R}, and upon choosing $\varepsilon > 0$ sufficiently small, we deduce the following 
\begin{align}\label{flowmapboundlipchistz}
    &\|J_{x}^{s}r\|_{L^{\infty}_{T}L^{2}_{x}} + \|J^{s}_{x}r\|_{L^{4}_{T,x}} + \sup_{0 \leq \theta \leq 1} \|r\|_{X^{1 - 2\theta, \theta}_{T}} + \|r\|_{L^{2}_{x}L^{\infty}_{T}} + \|J^{s}_{x}\partial_{x}r\|_{\widetilde{L^{\infty}_{x}L^{2}_{T}}} + \|w\|_{X^{s,\frac{1}{2},1}_{T}} \nonumber \\
    & \quad  + \|J^{s+k}_{x}q\|_{L^{\infty}_{T}L^{2}_{x}} + \|D^{s-1}_{x}q\|_{\widetilde{L^{2}_{x}L^{\infty}_{T}}}+\|q\|_{Y^{s,1}_{T}}+\sup_{0\leq \theta \leq 1}\|e^{iF_{1}}q_{1}-e^{iF_{2}}q_{2}\|_{Y_{T}^{s-\theta,\theta}} \nonumber \\
    &\leq C(M) \big( \|\varphi_{1} - \varphi_{2}\|_{H^{s}} + \|\psi_{1} - \psi_{2}\|_{H^{s+k}} \big),
\end{align}
where the constant $C(M)$ depends on $M$.

\subsection{Local well-posedness} \label{subectionLocalWellposedness}

This subsection is dedicated to establishing the existence, uniqueness, and continuous dependence on initial data for the HBOS system \eqref{BO-NLS}, as stated in Theorem \ref{maintheorem}. In this regard, the estimates derived in Subsections \ref{WPaprioriestimates} and \ref{WPcontflow} will play a fundamental role in the proof. 

Let $(r_{0},q_{0}) \in H^{s}(\mathbb{R})\times H^{s+k}(\mathbb{R})$ be any arbitrary initial data satisfying the condition \eqref{datainitialnorm}. Then, define the following sequences
\begin{equation}\label{WPdefinitionsequences}
    r_{0,j}:=\mathcal{F}_{x}^{-1}(\mathcal{X}_{\{|\xi|\leq j\}}\mathcal{F}_{x}r_{0}) \quad \textrm{and} \quad q_{0,j}:=\mathcal{F}_{x}^{-1}(\mathcal{X}_{\{|\xi|\leq j\}}\mathcal{F}_{x}q_{0}),
\end{equation}
for all $j\in \mathbb{Z}_{+}$. It is not difficult to obtain that 
\begin{equation}\label{limitrjqj}
    (r_{0,j},q_{0,j})\in H^{\infty}\times H^{\infty}, \quad \lim_{j\rightarrow \infty}\|r_{0,j}-r_{0}\|_{H^{s}}=\lim_{j\rightarrow\infty}\|q_{0,j}-q_{0}\|_{H^{s+k}}=0.
\end{equation}
Furthermore, we have that $P_{LO} r_{0,j} = P_{LO} r_{0,j'}$ and $P_{LO} q_{0,j} = P_{LO} q_{0,j'}$ for all $j, j' \geq 10$. Now, let $(r_{j}, q_{j}) \in C([0,T]; H^{\infty}(\mathbb{R}) \times H^{\infty}(\mathbb{R}))$ be the solutions obtained from Theorem~\ref{existenceofsmoothsolutions} and associated with the initial data $(r_{0,j}, q_{0,j})$, and $w_{j}$ the gauge function as defined in \eqref{gaugefunction} (the time $T$ can be taken to be independent of $j$, according to Remark \ref{Tindependentonj}). Therefore, by the arguments presented in Subsections \ref{WPaprioriestimates} and \ref{WPcontflow}, which were used to derive the estimates \eqref{boundNs} and \eqref{flowmapboundlipchistz}, respectively, it follows that
\begin{equation*}
    N_{T}^{s,k}(r_{j},q_{j})\lesssim \|(r_{0,j},q_{0,j})\|_{H^{s}\times H^{s+k}}\leq \|(r_{0},q_{0})\|_{H^{s}\times H^{s+k}}
\end{equation*}
and,
\begin{align}\label{WPCauchysequence}
    & \|J_{x}^{s}(r_{j} - r_{j'})\|_{L^{\infty}_{T} L^{2}_{x}} + \|J_{x}^{s}(r_{j} - r_{j'})\|_{L^{4}_{T,x}} + \sup_{0 \leq \theta \leq 1} \|r_{j} - r_{j'}\|_{X^{1 - 2\theta, \theta}_{T}} + \|r_{j} - r_{j'}\|_{L^{2}_{x} L^{\infty}_{T}} \nonumber \\
    & \quad + \|J_{x}^{s} \partial_{x} (r_{j} - r_{j'})\|_{\widetilde{L^{\infty}_{x} L^{2}_{T}}} + \|w_{j} - w_{j'}\|_{X^{s, \frac{1}{2}, 1}_{T}} + \|J_{x}^{s+k} (q_{j} - q_{j'})\|_{L^{\infty}_{T} L^{2}_{x}} + \|D_{x}^{s-1}(q_{j} - q_{j'})\|_{\widetilde{L^{2}_{x} L^{\infty}_{T}}} \nonumber \\
    &\quad +  \|q_{j}-q_{j'}\|_{Y^{s,1}_{T}}+\|e^{iF_{j}}q_{j}-e^{iF_{j'}}q_{j'}\|_{Y^{s-\theta,\theta}_{T}} \nonumber \\
    & \lesssim \|r_{0,j} - r_{0,j'}\|_{H^{s}} + \|q_{0,j} - q_{0,j'}\|_{H^{s+k}},
\end{align}
for all $j, j'$ sufficiently large. Thus, by using \eqref{limitrjqj}, it follows that there exist $(r,q)\in C([0,T]:H^{s}(\mathbb{R})\times H^{s+k}(\mathbb{R}))$ and $w\in X^{s,\frac{1}{2},1}_{T}$ which are limits of the sequences $(r_{j},q_{j})$ and $w_{j}$, respectively, such that 
\begin{align*}
r\in L^{4}_{T}W^{s,4}_{x}\cap L^{2}_{x}L^{\infty}_{T}\cap X^{s-2\theta,\theta}_{T},  \quad \textrm{and} \quad q\in Y^{s,1}_{T}\cap L^{2}_{x}L^{\infty}_{T},
\end{align*}
for all $\theta \in [0,1]$. It is a standard argument to verify that $(r,q)$ is a solution to the system \eqref{BO-NLS}, that $w$ satisfies the identity \eqref{identityofw}, and that the initial condition is satisfied. This proves the existence.
\newline
\indent The proof of the uniqueness follows readily by combining the estimate \eqref{flowmapboundlipchistz} and a scaling argument.
\newline
\indent Finally, we establish the continuity of the flow map. Let $(r_{0},q_{0})\in H^{s}\times H^{s+k}$ satisfy the bound \eqref{datainitialnorm}, and denote by $(r,q)$ the corresponding solution to \eqref{BO-NLS} on $[0,T]$ with this initial data. Fix $\lambda>0$. We shall prove that there exists $\delta>0$ such that, if $(u,v)$ is another solution to \eqref{BO-NLS} on $[0,T]$ with initial condition $(u(0),v(0))=(u_{0},v_{0})$ and 
$
     \|r_{0}-u_{0}\|_{H^{s}}+ \|q_{0}-v_{0}\|_{H^{s+k}} \leq \delta,
$ then 
\begin{equation*}
    \|r-u\|_{L^{\infty}_{T}H^{s}_{x}}+ \|q-v\|_{L^{\infty}_{T}H^{s+k}_{x}}  \leq \lambda.
\end{equation*}
In fact, let $(r_{0,j},q_{0,j})$ and $(u_{0,j},v_{0,j})$ be the sequences constructed as in \eqref{WPdefinitionsequences}, and denote by $(r_{j},q_{j})$ and $(u_{j},v_{j})$ the corresponding solutions to \eqref{BO-NLS} with initial data $(r_{0,j},q_{0,j})$ and $(u_{0,j},v_{0,j})$, respectively. Then we observe that
\begin{align}\label{WPcontflow1}
\|r-u\|_{L^{\infty}_{T}H^{s}_{x}}+ \|q-v\|_{L^{\infty}_{T}H^{s+k}_{x}}
&\leq \|r_{j}-r\|_{L^{\infty}_{T}H^{s}_{x}}+\|r_{j}-u_{j}\|_{L^{\infty}_{T}H^{s}_{x}}+\|u_{j}-u\|_{L^{\infty}_{T}H^{s}_{x}} \nonumber \\
&\quad+ \|q_{j}-q\|_{L^{\infty}_{T}H^{s+k}_{x}}+\|q_{j}-v_{j}\|_{L^{\infty}_{T}H^{s+k}_{x}}+\|v_{j}-v\|_{L^{\infty}_{T}H^{s+k}_{x}}.
\end{align}
On the one hand, from the estimate \eqref{WPCauchysequence}, it follows that there exists $j_{0}$ sufficiently large such that
\begin{equation}\label{WPcontflow2}
    \|r_{j}-r\|_{L^{\infty}_{T}H^{s}_{x}}+\|q_{j}-q\|_{L^{\infty}_{T}H^{s+k}_{x}}
    +\|u_{j}-u\|_{L^{\infty}_{T}H^{s}_{x}}+\|v_{j}-v\|_{L^{\infty}_{T}H^{s+k}_{x}}
    \leq \tfrac{\lambda}{2},
\end{equation}
for every $j\geq j_{0}$. On the other hand, by the definition of the initial data, we have the following estimate
\begin{equation*}
    \|r_{0,j}-u_{0,j}\|_{H^{3}}+\|q_{0,j}-v_{0,j}\|_{H^{3}}
    \lesssim j^{3-s}\|r_{0}-u_{0}\|_{H^{s}}
    + j^{3-s-k}\|q_{0}-v_{0}\|_{H^{s+k}}
    \leq j^{3-s}\delta.
\end{equation*}
Therefore, by the continuity of the flow map ensured by Theorem \ref{existenceofsmoothsolutions}, and taking $\delta$ sufficiently small, we deduce that
\begin{equation}\label{WPcontflow3}
    \|r_{j_{0}}-u_{j_{0}}\|_{L^{\infty}_{T}H^{s}_{x}}
    +\|q_{j_{0}}-v_{j_{0}}\|_{L^{\infty}_{T}H^{s+k}_{x}}
    \leq \tfrac{\lambda}{2}.
\end{equation}
Combining estimates \eqref{WPcontflow1}, \eqref{WPcontflow2}, and \eqref{WPcontflow3}, the desired result follows.

\subsection{Global well-posedness}

In this subsection, we take advantage of the conserved quantities \eqref{Hamiltoniano}-\eqref{Momento} to prove the global result stated in Theorem \ref{globalwellposedness}. Indeed, the idea is that these conserved quantities can provide an estimate of $H^{1}\times H^{1}$-norm of the solutions of HBOS system \eqref{BO-NLS}, which is independent on time. 
\newline
\indent Assume that $a<0$ in the system \eqref{BO-NLS}, and let $s\geq 1$. Let $(r,q)\in C([0,T])
:H^{s}(\mathbb{R})\times H^{s}(\mathbb{R}))$ be a solution associated with the initial data $(r_{0},q_{0})\in H^{s}\times H^{s}$ obtained from Theorem \ref{maintheorem}, where $T=T(\|(r_{0},q_{0})\|_{H^{1}\times H^{1}})$. Suppose that $\|(r_0,q_{0})\|_{H^{1}\times H^1}\leq \delta_{0}$ for $\delta_{0}>0$ small enough to be fixed later. By the continuous dependence provided by Theorem \ref{BO-NLS}, the solution $(r,q)$ satisfies the conserved quantities \eqref{Hamiltoniano}, \eqref{Mass}, and \eqref{Momento}. In particular, using the conserved quantity \eqref{Hamiltoniano}, we deduce that
\begin{align}\|\partial_{x}r\|_{L^{2}_{x}}^{2}+\|\partial_{x}q\|_{L^{2}_{x}}^{2}&\lesssim |E_{1}(r_{0},q_{0})|+\left|\int_{\mathbb{R}}r\mathcal{H}\partial_{x}rdx\right|+\left|\int_{\mathbb{R}}r^{3}dx\right|+\left| \int_{\mathbb{R}}r|q|^{2}dx\right|+\left|\int_{\mathbb{R}}r^{2}\mathcal{H}\partial_{x}r  dx\right| \nonumber \\
&=|E_{1}(r_{0},q_{0})|+I_{1}+I_{2}+I_{3}+I_{4}, \label{gweq0}
\end{align}
where the constant in the inequality depends only on $a,b,c,d,\alpha$ and $\beta$. On the other hand, from Young's inequality, it follows that
\begin{equation}
    I_{1}\lesssim  \|r\|_{L^{2}_{x}}^{2}+\varepsilon\|\partial_{x}r\|_{L^{2}_{x}}^{2}, \label{gweq1}
\end{equation}
for any $\varepsilon>0$.
Now, the Gagliardo-Nirenberg inequality, Young's inequality, and the conserved quantity \eqref{Mass} yield 
\begin{align}
    I_{2}&\lesssim \|r\|_{L^{2}_{x}}^{\frac{5}{2}}\|\partial_{x}r\|_{L^{2}_{x}}^{\frac{1}{2}}\lesssim \|r\|_{L^{2}_{x}}^{\frac{10}{3}}+\varepsilon\|\partial_{x}r\|_{L^{2}_{x}}^{2}, \label{gweq2}\\
    I_{3}&\lesssim \|r\|_{L^{2}_{x}}^{2}+\|q\|_{L^{4}_{x}}^{4}\lesssim \|r\|_{L^{2}_{x}}^{2}+\|q_{0}\|_{L^{2}}^{3}\|\partial_{x}q\|_{L^{2}_{x}}\lesssim \|r\|_{L^{2}_{x}}^{2}+\|q_{0}\|_{L^{2}}^{6}+\varepsilon \|\partial_{x}q\|_{L^{2}_{x}}^{2}, \label{gweq3}\\
    I_{4}&\lesssim \|r\|_{L^{4}_{x}}^{2}\|\partial_{x}r\|_{L^{2}_{x}}\lesssim \|r\|_{L^{2}_{x}}^{\frac{3}{2}}\|\partial_{x}r\|_{L^{2}_{x}}^{\frac{3}{2}}\lesssim \|r\|_{L^{2}_{x}}^{6}+\varepsilon\|\partial_{x}r\|_{L^{2}_{x}}^{2}. \label{gweq4}
\end{align}
Therefore, gathering the estimates \eqref{gweq0}-\eqref{gweq4} along with \eqref{Mass}, and choosing $\varepsilon$ small enough, we obtain
\begin{equation}\label{gweq8}
\|(r,q)\|_{H^{1}_x\times H^{1}_x}^{2}=\|q\|_{L^{2}_{x}}^{2}+\|r\|_{L^{2}}^{2}+\|\partial_{x}r\|_{L^{2}_{x}}^{2}+\|\partial_{x}q\|_{L^{2}_{x}}^{2}\leq C_{0}+C_{1}\|r\|_{L^{2}}^{6},
\end{equation}
where $C_{0}=C_{0}(\|r_{0}\|_{H^{1}},\|q_{0}\|_{H^{1}})$ and $C_{1}$ are positive constants. We observe that the same arguments used to estimate $I_{j}$ ensure that $|E_{1}(r_{0},q_{0})|\lesssim C(\|r_{0}\|_{H^{1}},\|r_{0}\|_{H^{1}})$. Now, using the conserved quantity \eqref{Momento} and Hölder's inequality, it follows that
\begin{equation}\label{gweq9}
    \|r\|_{L^{2}_{x}}^{2}\leq 2\left(|E_{3}(r_{0},q_{0})|+\|q_{0}\|_{L^{2}}\|\partial_{x}q\|_{L^{2}_{x}}\right).
\end{equation}
Thus, from the estimates \eqref{gweq8} and \eqref{gweq9}, we conclude
\begin{equation}\label{gweq10}
    \gamma(t)^{2}\leq \tilde{C}_{0}+\tilde{C_{1}}\delta_{0}^{3}\gamma(t)^{3},
\end{equation}
where $\gamma(t):=\|(r(t),q(t))\|_{H^{1}_x\times H^{1}_x}$ and $\tilde{C_{0}}=\tilde{C_{0}}(\|r_{0}\|_{H^{1}},\|q_{0}\|_{H^{1}})$.
Therefore, by taking $\delta_{0}$ small enough, the estimate \eqref{gweq10}, together with a continuity argument, leads us to
\begin{equation}\label{estimateGlobal}
    \sup_{0\leq t\leq T}\gamma(t)\leq M,
\end{equation}
where $M=2\max\{\tilde{C_{0}}^{\frac{1}{2}},\gamma(0)\}$ depends only on $\|(r_{0},q_{0})\|_{H^{1}\times H^{1}}$. The estimate \eqref{estimateGlobal} allows us to repeatedly apply Theorem \ref{maintheorem}, thereby extending the solution $(r,q)$ globally in time. It completes the proof of the theorem.

\noindent\textbf{Acknowledgements.} The author is grateful to Prof. Dr. Felipe Linares for valuable discussions and helpful comments related to this work. This work was partially supported by CNPq grant 141534/2023-0.

\nocite{*}
\bibliographystyle{plain}
\bibliography{referencias}

\noindent F. Santana, \textsc{IMPA, Instituto de Matemática Pura e Aplicada, Rio de Janeiro,
RJ, 22460-320, Brazil.}\\
\noindent \textit{E-mail address}: \texttt{fauster.silva@impa.br.}

\end{document}